\documentclass[11pt]{amsart}

\usepackage{geometry}
\usepackage{amsmath}
\usepackage{latexsym}
\usepackage{amssymb,amsthm,amsfonts,amsrefs}

\usepackage[colorlinks=true]{hyperref}

\usepackage{esint}

\newcommand{\eps}{\varepsilon }
\renewcommand{\epsilon}{\varepsilon }

\newcommand{\dist}{{\rm{dist}}}

\renewcommand{\le}         {\leqslant}
\renewcommand{\leq}         {\leqslant}
\renewcommand{\ge}         {\geqslant}
\renewcommand{\geq}         {\geqslant}

\newtheorem{lem}{Lemma}[section]
\newtheorem{pro}[lem]{Proposition}
\newtheorem{thm}[lem]{Theorem}

\newtheorem{cor}[lem]{Corollary}

\newcommand{\R}{\mathbb{R}}

\newcommand{\N}{\mathbb{N}}
\newcommand{\ve}{\varepsilon}

\numberwithin{equation}{section}

\author[J. D{\'a}vila]{Juan D{\'a}vila}
\author[M. del Pino]{Manuel del Pino}
\author[S. Dipierro]{Serena Dipierro}
\author[E. Valdinoci]{Enrico Valdinoci}

\address[Juan D{\'a}vila and Manuel del Pino]{Departamento de Ingenier\'ia
Matem\'atica and Centro de Modelamiento Matem\'atico,
Universidad de Chile,
Casilla 170 Correo 3,
8370459 Santiago,
Chile}
\email{jdavila@dim.uchile.cl, delpino@dim.uchile.cl}

\address[Serena Dipierro]{School of Mathematics, University of Edinburgh,
King's Buildings, Mayfield Road, Edinburgh EH9 3JZ, Scotland, UK
}
\email{serena.dipierro@ed.ac.uk}

\address[Enrico Valdinoci]{Weierstra{\ss} Institut f\"ur Angewandte Analysis und Stochastik, Hausvogteiplatz 11A, 10117 Berlin, Germany}
\email{enrico.valdinoci@wias-berlin.de}

\begin{document}

\title[Concentration phenomena for nonlocal equations]{Concentration phenomena
\\ for the nonlocal Schr\"odinger equation \\ with Dirichlet datum}

\begin{abstract}
For a smooth, bounded domain $\Omega$, $s\in(0,1)$,
$p\in \left(1,\frac{n+2s}{n-2s}\right)$
 we
consider the nonlocal equation
$$ \epsilon^{2s} (-\Delta)^s u+u=u^p \quad {\mbox{in }}\Omega $$
with zero Dirichlet datum and a small parameter $\ve>0$.
We construct a family of solutions that concentrate as $\ve \to 0$ at 
an interior point of the domain in the form of a scaling of the ground state in entire space. 
Unlike the classical case $s=1$, the leading order of the associated 
reduced energy functional in a variational reduction procedure 
is of polynomial instead of exponential order on 
the distance from the boundary, due to the nonlocal effect. Delicate analysis is needed to overcome the lack of localization, in particular establishing the rather unexpected asymptotics for the Green function 
of $ \epsilon^{2s} (-\Delta)^s +1$ in the expanding domain $\ve^{-1}\Omega$ with zero exterior 
datum.

\end{abstract}

\maketitle

\section{Introduction}

Given $s\in(0,1)$, $n\in \N$ with $n>2s$, $p\in \left(1,\frac{n+2s}{n-2s}\right)$ and a bounded smooth domain $\Omega\subset\R^n$,
we consider the fractional Laplacian problem
\begin{equation}\label{EQ}
\left\{
\begin{matrix}
\eps^{2s} (-\Delta)^s U+U=U^p & {\mbox{ in $\Omega$,}}\\
U=0 & {\mbox{ in $\R^n\setminus\Omega$,}}
\end{matrix}
\right.
\end{equation}
where $\eps>0$ is a small parameter.

As usual, the operator~$(-\Delta)^s$ is the fractional
Laplacian defined at any point~$x\in\R^n$ as
$$ (-\Delta)^s U(x):= c(n,s)\,\int_{\R^n}
\frac{2U(x)-U(x+y)-U(x-y)}{|y|^{n+2s}}\,dy,$$
for a suitable positive normalizing constant~$c(n,s)$.
We refer to~\cites{LAND, Sthesis, guide}
for an introduction to the fractional Laplacian operator.

\medskip
We provide in the appendix
a  heuristic physical motivation
of the problem considered and of the relevance of our results
in the light of a nonlocal quantum mechanics theory.

\medskip
The goal of this paper
is to construct solutions of problem \eqref{EQ}
that concentrate at interior points of the domain for sufficiently small
values of  $\eps$. More precisely, we shall establish the existence of a solution $U_\ve$ that at main order looks like  \begin{equation}\label{forma} U_\ve(x) \approx  w\left( \frac{x- \tilde\xi_\ve} {\ve}\right ).\end{equation} Here $\tilde\xi_\ve$ is a point lying at a uniformly positive distance from the boundary $\partial \Omega$, and
 $w$ designates the unique radial positive {\em least energy solution}
of the problem
\begin{equation}\label{EQ w}
(-\Delta)^s w+w=w^p, \, \quad w\in {\mbox{$H^s(\R^n)$.}}
\end{equation}
See for instance~\cite{FQT} for
the existence of such a solution and its basic properties. See \cites{AT,FL,fall} for the (delicate) proof of uniqueness in special situations and \cite{FLS} for the general case.
The solution $w$ is smooth and has the asymptotic behavior
\begin{equation}\label{decay w}
\alpha |x|^{-(n+2s)}\leq w(x)\le \beta |x|^{-(n+2s)} \quad \text{ for }
|x|\ge1,
\end{equation}
for some positive constants $\alpha,\beta$, see
Theorem 1.5 of \cite{FQT}, and the lower bound in formula~(IV.6) of~\cite{CMSJ}.

\medskip

Our main result is the following.

\begin{thm}\label{TH}
If $\eps$ is sufficiently small, there exist a point $\tilde\xi_\eps\in\Omega$
and a solution $U_\eps$ of
problem \eqref{EQ} such that
\begin{equation}\label{vicina}
\left| U_\eps(x)-w\left( \frac{x-\tilde\xi_\eps}{\eps}\right)\right|\le C \eps^{n+2s},
\end{equation}
and $\dist(\tilde\xi_\eps,\partial\Omega)\ge c$. 
Here, $c$ and $C$ are positive constants independent of $\ve$ and $\Omega$.
\smallskip

Besides, the point $\xi_\ve:=\tilde\xi_\eps/\eps$ is such that
\begin{equation}\label{def}
{\mathcal{H}}_\eps(\xi_\ve) =  \min_{\xi\in \Omega_\ve}{\mathcal{H}}_\eps(\xi) + O(\ve^{n+4s})
\end{equation}
for the functional ${\mathcal{H}}_\eps(\xi)$ defined in \eqref{def H cal} below, 
where
\begin{equation}\label{om eps}
\Omega_\eps:=\frac\Omega\eps=\left\{ \frac{x}{\eps},\ x\in\Omega\right\}.
\end{equation}
\end{thm}

The basic idea of the proof, which also leads to the characterization \eqref{def} 
of the location of the point $\tilde \xi_\ve$ goes as follows. 
Letting $u(x):=U(\eps x)$, problem \eqref{EQ} becomes
\begin{equation}\label{EQ eps}
\left\{
\begin{matrix}
(-\Delta)^s u+u=u^p & {\mbox{ in $\Omega_\eps$,}}\\
u=0 & {\mbox{ in $\R^n\setminus\Omega_\eps$,}}
\end{matrix}
\right.
\end{equation}
where $\Omega_\eps$ is defined in \eqref{om eps}. 

For a given $\xi\in\Omega_\eps$,
a first approximation $\bar u_\xi$ for the solution of problem \eqref{EQ eps} consistent with the desired form \eqref{forma} and the Dirichlet exterior condition, can be taken as the solution of the linear problem
\begin{equation}\label{EQ bar u}
\left\{
\begin{matrix}
(-\Delta)^s \bar u_\xi+\bar u_\xi=w^p_\xi & {\mbox{ in $\Omega_\eps$,}}\\
\bar u_\xi=0 & {\mbox{ in $\R^n\setminus\Omega_\eps$,}}
\end{matrix}
\right.
\end{equation}
where
$$
w_\xi(x):=w(x-\xi).
$$
The actual solution will be obtained as a small perturbation from~$\bar u_\xi$ for a suitable point $\xi= \xi_\ve$.
Problem \eqref{EQ eps} is variational. It corresponds to the Euler-Lagrange equation for the functional
\begin{equation}\label{energy functional}
I_{\eps}(u)=\frac{1}{2}\int_{\Omega_\eps}\left((-\Delta)^su(x)\, u(x)+u^2(x)\right) dx-\frac{1}{p+1}\int_{\Omega_\eps}u^{p+1}(x)\, dx, \qquad u\in H^s_0(\Omega_\epsilon),
\end{equation}
where
$$ H^s_0(\Omega_\epsilon)=\left\lbrace u\in H^s(\R^n) 
{\mbox{ s.t. }}
u=0 {\mbox{ a.e. in }} \R^n\setminus\Omega_\eps\right\rbrace. $$
Since the solution we look for should be close to $\bar u_\xi$ for $\xi=\xi_\ve$,
the functional $\xi \mapsto I_\ve (\bar u_\xi)$ should have a critical point near the 
$\xi =\xi_\ve$. We shall next argue 
that this functional actually has a global minimizer 
located at distance $\sim \frac 1\ve $ from $\partial \Omega_\ve$.

\medskip
The expansion of $I_\eps(\bar u_\xi)$ involves the regular part of the Green function for the operator $(-\Delta)^s+1$ in $\Omega_\eps$, which we define next.
In $\R^n$ the operator $ (-\Delta)^s+1 $  has unique decaying  fundamental solution $\Gamma$ which solves
\begin{equation}\label{EQ Gamma}
(-\Delta)^s \Gamma+\Gamma=\delta_0.\end{equation}
The function $\Gamma$ is radially symmetric, positive and satisfies
\begin{equation}\label{GA}
\frac{\alpha}{|x|^{n+2s}}\le\Gamma(x)\le \frac{\beta}{|x|^{n+2s}}
\end{equation}
for  $|x|\ge1$ and $\alpha,\beta>0$ see for instance
Lemma C.1 in \cite{FLS}.
%Moreover, by testing \eqref{EQ Gamma} against the function constantly equal to $1$, we see that
%\begin{equation}\begin{split}\label{EQ3}
%& 1=\int_{\R^n} (-\Delta)^s\Gamma(z)\cdot 1\,dz+\int_{\R^n}\Gamma(z)\,dz
%\\ & \qquad=\int_{\R^n}\Gamma(z)\cdot (-\Delta)^s1\,dz+\int_{\R^n}\Gamma(z)\,dz=
%\int_{\R^n}\Gamma(z)\,dz = 1
%\end{split}\end{equation}

The Green function $G_\eps$ for $(-\Delta)^s+1$ in $\Omega_\eps$ solves
\begin{equation}\label{EQ G}
\left\{
\begin{matrix}
(-\Delta)^s G_\eps(x,y) +G_\eps(x,y)=\delta_y & {\mbox{ if $x\in\Omega_\eps$,}}\\
G_\eps(x,y)=0 & {\mbox{ if $x\in\R^n\setminus\Omega_\eps$.}}
\end{matrix}
\right.
\end{equation}
In other words
\begin{equation}\label{EQ Rob}
G_\eps(x,y):=\Gamma(x-y)-H_\eps(x,y)
\end{equation}
where $H_\ve(x,y)$, the regular part, satisfies, for fixed $y\in\R^n$,
\begin{equation}\label{EQ H}
\left\{
\begin{matrix}
(-\Delta)^s H_\eps(x,y)+H_\eps(x,y)=0 & {\mbox{ if $x\in\Omega_\eps$,}}\\
H_\eps(x,y)=\Gamma(x-y) & {\mbox{ if $x\in\R^n\setminus\Omega_\eps$.}}
\end{matrix}
\right.
\end{equation}
%As in the classical case, $G_\eps$ and $H_\eps$ are symmetric, and
%by the Maximum Principle,
%\begin{equation}\label{HGamma}
%H_{\eps}(x,y)\le\Gamma(x-y), \qquad {\mbox{for any }}x,y\in\R^{n}.
%\end{equation}

%Since the true solution will be close to $\bar u_\xi$, its energy will be close to $I_\eps(\bar u_\xi)$.
We will show in Theorem~\ref{TH expansion} that 
for $\dist(\xi,\partial\Omega_{\eps})\ge\delta/\eps$, with~$\delta>0$ fixed
and appropriately small, we have that
\begin{equation}\label{estIintro}
I_{\eps}(\bar u_\xi)= I_0+\frac{1}{2}{\mathcal{H}}_\eps(\xi)+o(\eps^{n+4s}),
\end{equation}
where $I_0$ is the energy of $w$ computed in $\R^n$, and
${\mathcal{H}}_\eps(\xi)$ is given by
\begin{equation}\label{def H cal}
{\mathcal{H}}_\eps(\xi):=\int_{\Omega_\eps}\int_{\Omega_\eps} H_\eps(x,y)
\,w_\xi^p(x)\,w_\xi^p(y)\,dx\,dy .
\end{equation}
We will show that $\mathcal H_\eps$ satisfies
\begin{align}
\label{Heps}
\frac{\alpha}{\dist(\xi,\partial\Omega_\eps)^{n+4s}}\leq{\mathcal{H}}_\eps(\xi)\le \frac{\beta}{\dist(\xi,\partial\Omega_\eps)^{n+4s}},
\end{align}
where $\alpha,\beta>0$, for all points $\xi \in \Omega_\eps$ such that
$\dist(\xi,\partial\Omega_\eps)\in [5,\,\bar\delta/\eps]$, for $\bar\delta>0$ fixed, suitably small, and~$\eps\ll\bar\delta$.
 
{F}rom \eqref{Heps} and estimate \eqref{estIintro}, 
we deduce the existence of a global minimizer $\xi_\ve$ for the functional
$I_\ve(\bar u_\xi)$ for all small $\eps>0$, 
which is located at distance $\sim \frac 1\ve$ from $\partial \Omega_\ve$.
The actual proof {\em reduces} the problem of finding a solution close to 
$w_\xi$ via a Lyapunov-Schmidt procedure, 
to that of finding a critical point $\xi_\ve$ of a 
functional with a similar expansion to \eqref{estIintro}, 
as we will see in Section~\ref{S:a3}.

%The exponent in \eqref{Heps} is
%rather unexpected, namely it is different from the decay of the fundamental solution of $(-\Delta)^s$ (that is $n-2s$),
%from the one of $(-\Delta)^s+1$ (that is $n+2s$).

\bigskip
In the classical case (i.e. when~$s=1$ and the operator
boils down to the classical Laplacian), there is a broad
literature on concentration phenomena: we recall here
the seminal papers~\cites{Ni, Li} and we refer to~\cite{AM}
for detailed discussions and more precise references.
In particular, we recall that~\cites{Ni, Li,DF,DFW} construct
solutions of the classical Dirichlet problem that concentrate
at points which maximize the distance from the boundary:
in this sense, Theorem~\ref{TH} may be seen as the
nonlocal counterpart of these results. In our case,
the determination of the concentrating point is less
explicit than in the classical case, due to the nonlocal
behavior of the energy expansion. More precisely, for $s=1$ one gets the expansion parallel to \eqref{estIintro},
\begin{equation*}
I_{\eps}(\bar u_\xi)= I_0+\frac{1}{2}{\mathcal{H}}_\eps(\xi)+ O(e^{-\frac{(2+\sigma)\dist(\xi,\partial\Omega_\eps)}{\eps}}),
\end{equation*}
where now
\begin{equation}\label{H classic}
{\mathcal{H}}_\eps(\xi)\approx  e^{-2 \dist(\xi,\partial\Omega_\eps)/\eps}
\end{equation}
see for instance Y.-Y. Li, L. Nirenberg \cite{Li} (compare~\eqref{H classic} with~\eqref{Heps}).

\medskip
In the nonlocal case, much less is known.
Multi-peak solutions of a fractional
Schr\"odinger equation set in the whole of~$\R^n$
were considered recently in~\cite{DDW}.
The analysis needed in this paper is considerably more involved.
Concentrating solutions for fractional problems involving critical 
or almost critical exponents were considered in \cite{choi-kim-lee}.
See also~\cite{chen} for some concentration
phenomena in particular cases, and also~\cite{secchi} and references
therein for related problems about Schr\"odinger-type
equations in a fractional setting.

%Since
%$$
%c e^{-\frac{-2 \dist(\xi,\partial\Omega_\eps)}{\eps} - \frac\delta\eps}
%\leq \frac12 \int_{\Omega_\eps}\int_{\Omega_\eps} H^{(c)}_\eps(x,y) w(x-\xi)^p w(y -\xi) \, d x \, d y \leq C e^{-\frac{-2 \dist(\xi,\partial\Omega_\eps)}{\eps} + \frac\delta\eps} .
%$$

\bigskip

The paper is organized as follows.
%In Sections~\ref{S:1w} and~\ref{green section}
%we collect some estimates on the ground state and on the fundamental solution
%of~$(-\Delta)^s+1$.
The rather delicate analysis
of the behavior of the regular part of Green's function is contained in Section~\ref{s2:H}.
We estimate the function~$\bar u_\xi$ in Section~\ref{S3:u}, thus
obtaining a first approximation of
the energy expansion in Section~\ref{S:energy}.

The remainders of this expansion need to be carefully estimated:
for this, we provide some decay and regularity estimates
in Sections~\ref{S:a1} and~\ref{S:a2}.

\medskip
The Lyapunov-Schmidt method will be resumed in Section~\ref{S:a3}
where we discuss the linear theory and the
bifurcation from it. A key ingredient
is the linear nondegeneracy of
the least energy solution $w$: this is an important
result that was completely
achieved only recently in~\cite{FLS},
after preliminary works in particular cases
discussed in~\cites{AT,FL, fall}.
Then we complete the proof of Theorem \ref{TH}
in Section~\ref{end:p}.

\section{Estimates on the Robin function $H_\eps$
and on the leading term of the energy functional}\label{s2:H}

Given $\xi\in\Omega_\eps$ and $x\in\R^n$, we define
$$ \beta_\xi(x):=\int_{\R^n\setminus\Omega_\eps}\Gamma(z-\xi)\, \Gamma(x-z)\,dz.$$
Notice that, for any $x\in\Omega_\eps$ and $z\in{\R^n\setminus\Omega_\eps}$ we have
\begin{equation}\label{EQ4}
\Big((-\Delta)^s+1\Big) \Gamma(x-z)=\delta_0(x-z)=0,\end{equation}
and so
\begin{equation}\label{EQ:beta}
%L_s\beta_\xi(x)=
\Big((-\Delta)^s+1\Big)\beta_\xi(x)=\int_{\R^n\setminus\Omega_\eps}\Gamma(z-\xi)\,
\Big((-\Delta)^s+1\Big) \Gamma(x-z)\,dz=0
\end{equation}
for any $x\in\Omega_\eps$.
Our purpose is to use $\beta_\xi(x)$
as a barrier, from above and below, for the Robin function $H_\eps(x,\xi)$,
using \eqref{EQ H}, \eqref{EQ:beta} and the Comparison Principle.
For this scope, we estimate the behavior of $\beta_\xi$ outside $\Omega_\eps$:

\begin{lem}
There exists $c\in(0,1)$ such that
\begin{equation}\label{double}
c H_\eps(x,\xi)\le \beta_\xi(x)\le c^{-1}H_\eps(x,\xi)\end{equation}
for any $x\in\R^n\setminus\Omega_\eps$ and $\xi\in\Omega_\eps$
with
\begin{equation}\label{13}
\dist(\xi,\partial\Omega_\eps)\ge 1.
\end{equation}
\end{lem}

\begin{proof}
First we observe that for any $x\in \R^n\setminus\Omega_\eps$,
\begin{equation}\label{geo}
|B_{1/2}(x)\setminus\Omega_\eps|\ge c_\star
\end{equation}
for a suitable $c_\star>0$. For concreteness, one can take
$c_\star$ as the measure of the spherical segment
$$ \Sigma:= \big\{ z=(z',z_n)\in \R^{n-1}\times \R
{\mbox{ s.t. }} |z|< 1/2 {\mbox{ and }}
z_n\ge 1/4\big\}.$$
To prove \eqref{geo}, we argue as follows.
If $B_{1/2}(x)\subseteq(\R^n\setminus\Omega_\eps)$ we are done.
If not, let $p\in(\partial\Omega_\eps)\cap B_{1/2}(x)$,
with $\dist(x,\partial\Omega_\eps)=|x-p|$.
Notice that the ball centered at $x$ of radius $|x-p|$
is tangent to $\Omega_\eps$ from the outside at $p$, and $|x-p|\le 1/2$.

Up to a rigid motion, we suppose that $p=0$ and $x=|x|e_n$.
By scaling back, the ball of radius $|\hat x|$
centered at $\hat x:=\eps x=\eps |x| e_n$
is tangent to $\Omega$ from the outside at the origin,
and $|\hat x|=\eps|x|=\eps|x-p|\le\eps/2$.

{F}rom the regularity of $\Omega$, we have that
there exists a ball of universal radius $r_o>0$
touching $\Omega$ from the outside at any point,
so in particular $B_{r_o}(r_o e_n)$ touches $\Omega$
from the outside at the origin, hence
\begin{equation}\label{geo2}
B_{r_o}(r_o e_n)\subseteq\R^n\setminus\Omega.
\end{equation}
We observe that
\begin{equation}\label{geo3}
\hat x+\eps\Sigma\subseteq B_{r_o}(r_o e_n).
\end{equation}
Indeed, if $z=(z',z_n)\in\hat x+\eps\Sigma$
then $\eps\Sigma\ni z-\hat x=(z',z_n-|\hat x|)$
and so $z_n-|\hat x|\in[ \eps/4,\eps/2]$ and $|z'|\le|z|\le\eps/2$. Hence,
for small $\eps$, we have that $r_o-z_n\ge r_o-|\hat x|-(\eps/2)\ge r_o-\eps\ge 0$,
and $r_o-z_n\le r_o-|\hat x|-(\eps/4)\le r_o-(\eps/4)$ and so
$$ |z_n-r_o|=r_o-z_n\le r_o-\frac\eps4$$
that gives
\begin{eqnarray*} &&|z-r_o e_n|^2=|z'|^2+|z_n-r_o|^2\le
\left(\frac\eps2\right)^2+
\left(r_o-\frac\eps4\right)^2\\ &&\qquad=
r_o^2+\frac{\eps^2}{16}+\frac{\eps^2}{4}-\frac{2r_o\eps}{2}<r_o\end{eqnarray*}
if $\eps$ is sufficiently small. This proves \eqref{geo3}.

As a consequence of \eqref{geo2} and \eqref{geo3}, we conclude that $\hat
x+\eps\Sigma
\subseteq\R^n\setminus\Omega$, that is, by scaling back,
$x+\Sigma\subseteq\R^n\setminus\Omega_\eps$. Accordingly,
$$ (B_{1/2}(x)\setminus\Omega_\eps)\supseteq
B_{1/2}(x)\cap (x+\Sigma)=x+\Sigma$$
and this ends the proof of \eqref{geo}.

Now we observe that if $a$, $b\in\R^n$
satisfy $|a-b|\le |b-\xi|/2$, $\min\{|a-\xi|,|b-\xi|\}\ge 1$ then
\begin{equation}\label{Gab} \Gamma(a-\xi)\le C\Gamma(b-\xi)\end{equation}
for some $C>0$. Indeed,
$$ |a-\xi|\ge |b-\xi|-|a-b|\ge
\frac{|b-\xi|}2$$
and so, from \eqref{GA},
\begin{eqnarray*}
&& \Gamma(a-\xi) \le \frac{C}{|a-\xi|^{n+2s}}\le
\frac{2^{n+2s}C}{|b-\xi|^{n+2s}}\le
2^{n+2s}C^2 \Gamma(b-\xi).
\end{eqnarray*}
This proves \eqref{Gab}, up to relabeling the constants.
As a consequence, given $x\in\R^n\setminus\Omega_\eps$,
we apply \eqref{Gab} with $a:=x$ and $b:=z\in B_{1/2}(x)\setminus \Omega_\eps$,
we recall \eqref{13} and \eqref{geo} and we obtain that
\begin{eqnarray*}
\beta_\xi(x) &\ge&
\int_{B_{1/2}(x)\setminus\Omega_\eps}\Gamma(z-\xi)\, \Gamma(x-z)\,dz\\
&\ge& C^{-1}\int_{B_{1/2}(x)\setminus\Omega_\eps}\Gamma(x-\xi)\, \Gamma(x-z)\,dz\\
&\ge& C^{-1} \Gamma(x-\xi) \inf_{y\in B_{1/2}} \Gamma(y) \,|B_{1/2}(x)\setminus\Omega_\eps|
\\ &\ge& c_\star\, C^{-1} \Gamma(x-\xi) \inf_{y\in B_{1/2}} \Gamma(y).
\end{eqnarray*}
This proves the first inequality in \eqref{double}, since
$x\in\R^n\setminus\Omega_\eps$ and so
\begin{equation}\label{6}
\Gamma(x-\xi)=H_\eps(x,\xi).\end{equation}
Now we prove the second inequality in \eqref{double}.
For this, we use \eqref{Gab} once again (applied here
with $a:=z$, $b:=x$ and recalling \eqref{13})
and the fact that 
\begin{equation}\label{EQ3}  
\int_{\R^n} \Gamma (z)\, dz \ =\ 1
\end{equation} 
 to see that
\begin{equation}\label{up1}
I_1:=\int_{B_{|x-\xi|/2}(x)\setminus\Omega_\eps} \Gamma(z-\xi)\, \Gamma(x-z)\,dz\le
C \int_{B_{|x-\xi|/2}(x)} \Gamma(x-\xi)\, \Gamma(x-z)\,dz\le
C\, \Gamma(x-\xi).
\end{equation}
On the other hand,
if $z\not\in B_{|x-\xi|/2}(x)$, we have that
$|x-z|\ge |x-\xi|/2$ and so, by \eqref{GA}
$$ \Gamma(x-z)\le \frac{C}{|x-z|^{n+2s}}\le
\frac{2^{n+2s} C}{|x-\xi|^{n+2s}}\le 2^{n+2s} C^2\Gamma(x-\xi).$$
Consequently
\begin{equation*}
I_2:=\int_{\R^n\setminus B_{|x-\xi|/2}(x)}
\Gamma(z-\xi)\, \Gamma(x-z)\,dz\le C'
\int_{\R^n\setminus B_{|x-\xi|/2}(x)}
\Gamma(z-\xi)\, \Gamma(x-\xi)\,dz\le C'\Gamma(x-\xi),\end{equation*}
for some $C'>0$, thanks to \eqref{EQ3}. {F}rom this and \eqref{up1}
we obtain that
$$ \beta_\xi(x)\le
I_1+I_2\le C''\Gamma(x-\xi),$$
for some $C''>0$. This, together with \eqref{6},
completes the proof of \eqref{double}.
\end{proof}

\begin{cor}\label{Co}
There exists $c\in(0,1)$ such that
$$ c H_\eps(x,\xi)\le \beta_\xi(x)\le c^{-1}H_\eps(x,\xi)$$
for any $x\in\R^n$ and $\xi\in\Omega_\eps$ with
$\dist(\xi,\partial\Omega_\eps)\ge 1$.
\end{cor}

\begin{proof} The desired estimate holds true outside $\Omega_\eps$,
thanks to \eqref{double}. Then it holds true inside $\Omega_\eps$ as well,
in virtue of \eqref{EQ:beta}, \eqref{EQ H}
and the Comparison Principle.
\end{proof}

The above result implies an interesting lower bound
on the symmetric version of the Robin function $H_\eps(\xi,\xi)$,
and in general for the values of the Robin function sufficiently close
to the diagonal, according to the following

\begin{pro}\label{pro below}
Let $\delta\in(0,1)$.
Let $\xi\in\Omega_\eps$ with
$$ d:=\dist(\xi,\partial\Omega_\eps)\in \left[2,\,\frac{\delta}{\eps}\right].$$ Let $x,y\in B_{d/2}(\xi)$. Then
$$ H_\eps(x,y)\ge \frac{c_o}{d^{n+4s}}$$
for a suitable $c_o\in(0,1)$, as long as $\delta$ is sufficiently small.
\end{pro}

\begin{proof} Let $z\in\R^n\setminus\Omega_\eps$. Notice that
\begin{equation}\label{4}
|\xi-y|\le \frac{d}{2}\le \frac{|z-\xi|}2
\end{equation}
and so
\begin{equation}\label{14}
|z-y|\le |z-\xi|+|\xi-y|\le \frac{3}{2} |z-\xi|.
\end{equation}
Similarly,
\begin{equation}\label{15}
|z-x|\le \frac{3}{2} |z-\xi|.
\end{equation}
Another consequence of \eqref{4} is that
\begin{equation}\label{16}
|z-y|\ge |z-\xi|-|\xi-y|\ge \frac{|z-\xi|}2\ge\frac{d}2\ge 1
\end{equation}
hence $\dist(y,\partial\Omega_\eps)\ge 1$ (as a matter
of fact, till now we only exploited that $d\ge2$).
Notice that in the same way, one has that
\begin{equation}\label{17}
|z-x|\ge 1.
\end{equation}
Therefore we can use Corollary \ref{Co} with $\xi$ replaced by $y$
and so, recalling \eqref{GA}, \eqref{16}, \eqref{17}, \eqref{14}
and \eqref{15}, we conclude that
\begin{equation}\label{7}
\begin{split}
H_\eps(x,y)\,&\ge c \beta_y (x) \\
&= c\int_{\R^n\setminus\Omega_\eps}\Gamma(z-y)\, \Gamma(x-z)\,dz\\
&\ge \frac{c}{C^2} \int_{\R^n\setminus\Omega_\eps}
\frac{1}{|y-z|^{n+2s} |x-z|^{n+2s}}\,dz \\ &\ge
c' \int_{\R^n\setminus\Omega_\eps}
\frac{1}{|z-\xi|^{2n+4s}}\,dz
\end{split}\end{equation}
for a suitable $c'>0$.

Now we make some geometric considerations. By the smoothness
of the domain, we can touch $\Omega$ from the outside at any point
with balls of universal radius, say $r_o>0$. By scaling, we
can touch $\Omega_\eps$ from the exterior by balls of radius $r_o\,\eps^{-1}$,
and so of radius $d$ (notice indeed that $d\le \delta\eps^{-1}\le r_o\,\eps^{-1}$
if $\delta$ is small enough).
Let $\eta\in\partial\Omega_\eps$ be such that $|\xi-\eta|=d$.
By the above considerations, we can touch $\Omega_\eps$ from
the outside at $\eta$ with a ball ${\mathcal{B}}$ of radius $d$
(i.e. of diameter $2d$).
We stress that ${\mathcal{B}}\subseteq\R^n\setminus\Omega_\eps$,
that $|{\mathcal{B}}|\ge \bar c d^n$ for some $\bar c>0$
and that if $z\in {\mathcal{B}}$ then
$$ |z-\xi|\le |z-\eta|+|\eta-\xi|\le 2d+d=3d.$$
These observations and \eqref{7} yield that
\begin{eqnarray*}
&& H_\eps(x,y) \ge
c' \int_{\mathcal{B}}
\frac{1}{|z-\xi|^{2n+4s}}\,dz \ge c'\,(3d)^{-(2n+4s)}\,|{\mathcal{B}}|=
c'\,3^{-(2n+4s)}\,\bar c\,d^{-(n+4s)},
\end{eqnarray*}
as desired.
\end{proof}

There is also an upper bound similar to the lower
bound obtained in Proposition \ref{pro below}:

\begin{pro}\label{pro upper}
Let $\xi\in\Omega_\eps$ with
$ d:=\dist(\xi,\partial\Omega_\eps)\ge2$, and $x,y\in B_{d/2}(\xi)$. Then
$$ H_\eps(x,y)\le \frac{C_o}{d^{n+4s}}$$
for a suitable $C_o>0$.
\end{pro}

\begin{proof}
As noticed in the proof of Proposition \ref{pro below}
we can use Corollary \ref{Co} with $\xi$ replaced by $y$.
Then, since $B_d(\xi)\subseteq \Omega_\eps$, hence $(\R^n\setminus\Omega_\eps)
\subseteq(\R^n\setminus B_d(\xi))$, and therefore we obtain that
$$ H_\eps(x,y)\le c^{-1}\beta_y(x)\le c^{-1}
\int_{\R^n\setminus B_d(\xi)}\Gamma(z-\xi)\, \Gamma(x-z)\,dz.$$
Also, if $z\in \R^n\setminus B_d(\xi)$,
we have that
$$ |z-x|\ge |z-\xi|-|\xi-x|\ge d-|\xi-x|\ge \frac{d}{2},$$
hence, by \eqref{GA},
\begin{eqnarray*}
H_\eps(x,y)&\le& c^{-1} C^2
\int_{\R^n\setminus B_d(\xi)}\frac{1}{|z-\xi|^{n+2s}|z-x|^{n+2s}}\,dz
\\ &\le& c^{-1} C^2 (2/d)^{n+2s}
\int_{\R^n\setminus B_d(\xi)}\frac{1}{|z-\xi|^{n+2s}}\,dz.
\end{eqnarray*}
By computing the latter integral in polar coordinates,
we obtain the desired result.
\end{proof}

It will be convenient to define for any $\xi\in\Omega_\eps$,
\begin{equation}\label{EQ Pi}
\Pi_\eps(x,\xi):=\int_{\Omega_\eps} H_\eps(x,y) w^p_\xi (y)\,dy.\end{equation}
As a consequence of Proposition~\ref{pro upper}, we have

\begin{lem}\label{EST PI}
Let $\xi\in\Omega_\eps$ with
$d:=\dist(\xi,\partial\Omega_\eps)\ge2$. Let $x\in B_{d/8}(\xi)$. Then
$$ \Pi_\eps (x,\xi)\le \frac{C}{d^{n+4s}}$$
for some $C>0$, where $\Pi_\eps(x,\xi)$ is defined in \eqref{EQ Pi}.
\end{lem}

\begin{proof} We split the integral into two contributions, one in $B_{d/4}(\xi)$ and one
outside such ball.

We can use Proposition \ref{pro upper} to obtain that, for $y\in\Omega_\eps\cap B_{d/4}(\xi)$,
it holds that $H_\eps(x,y)\le C_o d^{-n-4s}$ and so
$$ \pi_1:=\int_{\Omega_\eps\cap B_{d/4}(\xi)} H_\eps(x,y) w^p_\xi (y)\,dy
\le C_o d^{-n-4s} \int_{\R^n}w^p_\xi (y)\,dy\le C_1 d^{-n-4s},$$
for some $C_1>0$.

Now we consider the case in which $y\in\Omega_\eps\setminus B_{d/4}(\xi)$.
In this case, we use \eqref{decay w} to see that $w_\xi^p(y)\le C_2
|y-\xi|^{-p(n+2s)}$
for some $C_2>0$. Also, in this case, $$|y-x|\ge |y-\xi|-|x-\xi|
\ge \frac{d}4-\frac{d}8=\frac{d}8$$
hence, since by maximum principle 
\begin{equation}\label{HGamma}  H_\eps(x,y)\le \Gamma(x-y)\le \frac{C_3}{|x-y|^{n+2s}}\le
\frac{C_4}{d^{n+2s}},\end{equation}
for some $C_3$, $C_4>0$. As a consequence
$$ \pi_2:=\int_{\Omega_\eps\setminus B_{d/4}(\xi)} H_\eps(x,y) w^p_\xi
(y)\,dy
\le \frac{C_2 C_4}{d^{n+2s}} \int_{\R^n\setminus B_{d/4}(\xi)}
|y-\xi|^{-p(n+2s)}\,dz=\frac{C_5}{d^{2s+p(n+2s)}}$$
for some $C_5>0$. In particular, since $d\ge1$ and $p>1$,
we see that $\pi_2\le C_5 d^{-n-4s}$ and therefore, recalling \eqref{EQ
Pi},
we conclude that $\Pi_\eps (x,\xi)\le
\pi_1+\pi_2\le (C_1+C_5)d^{-n-4s}$.
\end{proof}

%Now we define
%\begin{equation}\label{def H cal}
%{\mathcal{H}}_\eps(\xi):=\int_{\Omega_\eps}\int_{\Omega_\eps} H_\eps(x,y)
%\,w_\xi^p(x)\,w_\xi^p(y)\,dx\,dy.
%\end{equation}
The function ${\mathcal{H}}_\eps$ defined in \eqref{def H cal} will represent
the first interesting order in the expansion of
the reduced energy functional (see the forthcoming Theorem~\ref{TH expansion} for a precise statement). To show that
this reduced energy functional has a local minimum,
we will show that ${\mathcal{H}}_\eps$ (and so the
reduced energy functional itself) attains,
in a certain domain, values that are smaller than the ones attained
at the boundary (concretely, this domain will
be given by the subset of $\Omega_\eps$ with
points of distance $\delta/\eps$ from the boundary,
for some $\delta\in(0,1)$ fixed suitably small
possibly in dependence of $n$, $s$ and $\Omega$).

To this extent, a detailed statement will be given in Proposition~\ref{P 3.8} and
the necessary bounds on
${\mathcal{H}}_\eps$ will be given in the forthcoming
Corollaries \ref{H BELOW} and \ref{H UPPER}, which
in turn follow from Propositions \ref{pro below} and \ref{pro upper},
respectively.

\begin{cor}\label{H BELOW}
Let $\delta\in(0,1)$.

Let $\xi\in\Omega_\eps$ with
$$ d:=\dist(\xi,\partial\Omega_\eps)\in \left[2,\,\frac{\delta}{\eps}\right].$$
Then
$$ {\mathcal{H}}_\eps(\xi)\ge \frac{c}{d^{n+4s}},$$
for a suitable $c>0$,
as long as $\delta$ is sufficiently small.

\end{cor}

\begin{proof} Notice that $B_1(\xi)\subseteq B_{d/2}(\xi)\subseteq\Omega_\eps$.
So, by Proposition \ref{pro below}, $H_\eps(x,y)\ge c_o d^{-(n+4s)}$
if $x,y\in B_1(\xi)$ and
\begin{equation*}
{\mathcal{H}}_\eps(\xi)\ge\int_{B_1(\xi)}\int_{B_1(\xi)} H_\eps(x,y)
\,w_\xi^p(x)\,w_\xi^p(y)\,dx\,dy\ge c_o d^{-(n+4s)}\left( \int_{B_1} w^p(z)\,dz
\right)^2.\qedhere
\end{equation*}
\end{proof}

\begin{cor}\label{H UPPER}
Let $\xi\in\Omega_\eps$ with
$ d:=\dist(\xi,\partial\Omega_\eps)\ge 5$.
Then
$$ {\mathcal{H}}_\eps(\xi)\le \frac{C}{d^{n+4s}},$$
for a suitable $C>0$.
\end{cor}

\begin{proof} We split the integral in \eqref{def H cal}
into three contributions: first we treat the case in which $x,y\in B_{d/2}(\xi)$,
then the case in which $x,y\in\R^n\setminus B_{d/2}(\xi)$, and finally
the case in which
$x\in B_{d/2}(\xi)$ and $y\in\R^n\setminus B_{d/2}(\xi)$
(the case in which $y\in B_{d/2}(\xi)$ and $x\in\R^n\setminus B_{d/2}(\xi)$
is, of course, symmetrical to this one).

In the first case, we use Proposition \ref{pro upper}, obtaining that
\begin{equation}\label{H.01}\begin{split}
&\int_{B_{d/2}(\xi)}\,dx \int_{B_{d/2}(\xi)}\,dy \,
H_\eps(x,y)\,w_\xi^p(x)\,w_\xi^p(y)
\\ \le\,& C_o d^{-(n+4s)} \left(\int_{B_{d/2}} w^p(z)\,dz\right)^2\\ \le\,&
C_o d^{-(n+4s)} \left(\int_{\R^n} w^p(z)\,dz\right)^2.
\end{split}\end{equation}
In the second case, we use twice the decay of $w$ given in \eqref{decay w},
\eqref{HGamma} and \eqref{EQ3}, obtaining that
\begin{equation}\label{H.02}
\begin{split}
&\int_{\R^n\setminus B_{d/2}(\xi)}\,dx \int_{\R^n\setminus B_{d/2}(\xi)}\,dy
\,H_\eps(x,y)\,w_\xi^p(x)\,w_\xi^p(y)
\\
\le\,& C^{2p} \int_{\R^n\setminus B_{d/2}(\xi)}\,dx \int_{\R^n\setminus B_{d/2}(\xi)}\,dy
\,|x-\xi|^{-p(n+2s)} |y-\xi|^{-p(n+2s)}\,\Gamma(x-y) \\
\le\,& C^{2p}\, (d/2)^{-p(n+2s)}
\int_{\R^n\setminus B_{d/2}(\xi)}\,dx \int_{\R^n\setminus B_{d/2}(\xi)}\,dy
\,|x-\xi|^{-p(n+2s)} \,\Gamma(x-y)\\
\le\,& C^{2p}\, (d/2)^{-p(n+2s)}
\int_{\R^n\setminus B_{d/2}}\,d\eta \int_{\R^n}\,d\theta
\,|\eta|^{-p(n+2s)} \, \Gamma(\theta)\\
\le\,& C' d^{-2p(n+2s)+n}\\
\le\,& C' d^{-(n+4s)},
\end{split}
\end{equation}
for some $C'>0$.

As for the third case,
we take $x\in B_{d/2}(\xi)$ and $y\in\R^n\setminus B_{d/2}(\xi)$
and we distinguish two subcases: either $|x-y|\le d/6$
or $|x-y|>d/6$.

In the first subcase, we use a translated version
of Proposition \ref{pro upper}. Namely,
if
$x\in B_{d/2}(\xi)$, $y\in\R^n\setminus B_{d/2}(\xi)$
and $|x-y|\le d/6$
we take $\hat\xi:=(x+y)/2$.
Notice that
$$ |\xi-y|\le|\xi-x|+|x-y|\le \frac{d}{2}+\frac{d}{6}$$
and therefore
$$ 2|\hat\xi-\xi|=|(x+y)-2\xi|\le |x-\xi|+|y-\xi|\le \frac{d}{2}+\left(
\frac{d}{2}+\frac{d}{6}\right)=\frac{7d}{6}.$$
As a consequence
\begin{equation}\label{AUX 0}
\hat d:=\dist(\hat\xi,\partial\Omega_\eps)\ge
\dist(\xi,\partial\Omega_\eps)-|\hat\xi-\xi|\ge
d-\frac{7d}{12}=\frac{5d}{12}.\end{equation}
In particular
\begin{equation}\label{AUX 1}
\hat d\ge2.\end{equation}
Also, by construction $x-\hat\xi=\hat\xi-y=(x-y)/2$, and so
$$ |x-\hat\xi|=|\hat\xi-y|=\frac12|x-y|\le\frac{d}{12}.$$
This and \eqref{AUX 0} say that
\begin{equation}\label{AUX 2}
x,y\in B_{d/12}(\hat\xi)\subseteq B_{\hat d/2}(\hat\xi).
\end{equation}
Thanks to \eqref{AUX 1} and \eqref{AUX 2} we can now
use Proposition \ref{pro upper} with $\xi$ and $d$ replaced by $\hat\xi$
and $\hat d$, respectively. So we obtain that, in this case,
\begin{equation}\label{AUX 5}
H_\eps(x,y)\le\frac{C_o}{\hat d^{n+4s}}\le\frac{\hat C}{d^{n+4s}},
\end{equation}
for some $\hat C>0$, where \eqref{AUX 0} was used again in the last
inequality.

So, we make use of \eqref{decay w} and \eqref{AUX 5}
to obtain that
\begin{equation}\label{H.03pre a}
\begin{split}
&\int_{B_{d/2}(\xi)}\,dx \int_{B_{d/6}(x)\setminus B_{d/2}(\xi)}\,dy
\, H_\eps(x,y)\,w_\xi^p(x)\,w_\xi^p(y)
\\ \le\,& C^p\,\hat C \,d^{-(n+4s)}
\int_{B_{d/2}(\xi)}\,dx \int_{B_{d/6}(x)\setminus B_{d/2}(\xi)}\,dy
\,w_\xi^p(x)\,|y-\xi|^{-p(n+2s)}\\
\le\,& C^p\,\hat C \,d^{-(n+4s)}
\int_{\R^n}\,dx \int_{\R^n\setminus B_{d/2}}\,dz
\,w_\xi^p(x)\,|z|^{-p(n+2s)}\\
\le\,& \tilde C\,d^{-4s-p(n+2s)}
\\
\le\,& \tilde C\,d^{-(n+4s)}
.\end{split}\end{equation}
Finally, we consider the subcase in which
$x\in B_{d/2}(\xi)$, $y\in\R^n\setminus B_{d/2}(\xi)$
and $|x-y|>d/6$.
In this circumstance we use \eqref{decay w},
\eqref{HGamma} and~\eqref{GA} to conclude that
\begin{equation}\label{H.03pre b}
\begin{split}
&\int_{B_{d/2}(\xi)}\,dx \int_{{\R^n\setminus B_{d/2}(\xi)}\atop{\{|x-y|>d/6\}}}\,dy
\, H_\eps(x,y)\,w_\xi^p(x)\,w_\xi^p(y)
\\ \le\,& C^p\,
\int_{B_{d/2}(\xi)}\,dx \int_{{\R^n\setminus B_{d/2}(\xi)}\atop{\{|x-y|>d/6\}}}\,dy
\, \Gamma(x-y)\,w_\xi^p(x)\,|y-\xi|^{-p(n+2s)}\\
\le\,& \underline C
\int_{B_{d/2}(\xi)}\,dx \int_{{\R^n\setminus B_{d/2}(\xi)}\atop{\{|x-y|>d/6\}}}\,dy
\, |x-y|^{-(n+2s)}\,w_\xi^p(x)\,|y-\xi|^{-p(n+2s)}\\
\le\,& \underline C (d/6)^{-(n+2s)}
\int_{\R^n}\,dx \int_{{\R^n\setminus B_{d/2}(\xi)}}\,dy
\, w_\xi^p(x)\,|y-\xi|^{-p(n+2s)}\\
\le\, & \overline C d^{-2s-p(n+2s)}\\
\le\,& \overline C d^{-(n+4s)}
\end{split}
\end{equation}
for suitable $\underline C$, $\overline C>0$. {F}rom \eqref{H.03pre a}
and \eqref{H.03pre b} we complete the third case, namely when
$x\in B_{d/2}(\xi)$ and $y\in\R^n\setminus B_{d/2}(\xi)$, by obtaining that
\begin{equation}\label{H.03}
\int_{B_{d/2}(\xi)}\,dx \int_{\R^n\setminus B_{d/2}(\xi)}\,dy
H_\eps(x,y)\,w_\xi^p(x)\,w_\xi^p(y)
\le (\tilde C+\overline C) d^{-(n+4s)}.
\end{equation}
The desired result follows from \eqref{def H cal},
\eqref{H.01}, \eqref{H.02} and \eqref{H.03}.
\end{proof}

For concreteness, we summarize the results of
Corollaries \ref{H BELOW} and \ref{H UPPER} in the following

\begin{pro}\label{P 3.8}
Let $\delta>0$ be suitably small and
\begin{equation}\label{omega delta}
\Omega_{\eps,\delta}:=\{ x\in\Omega_\eps {\mbox{ s.t. }}
\dist(x,\partial\Omega_\eps)>\delta/\eps\}.
\end{equation}
Then ${\mathcal{H}}_\eps$ attains an interior minimum
in $\Omega_{\eps,\delta}$, namely there exist $c_1$, $c_2>0$ such that
$$ \min_{\Omega_{\eps,\delta}} {\mathcal{H}}_\eps\le
c_1\eps^{n+4s}<c_2\,\left(\frac{\eps}\delta\right)^{n+4s}
\le \min_{\partial\Omega_{\eps,\delta}} {\mathcal{H}}_\eps.$$
\end{pro}

\begin{proof} Let $\delta_\star$ to be the maximal distance that
a point of $\Omega$ may attain from the boundary of $\Omega$.
By scaling, the maximal distance that
a point of $\Omega_\eps$ may attain from the boundary of $\Omega_\eps$
is $\delta_\star/\eps$. Let $\xi_\star$ be such a point, i.e.
$$ d_\star:=\dist(\xi_\star,\partial\Omega_\eps)=\frac{\delta_\star}{\eps}.$$
For $\delta$ sufficiently small we have that $\xi_\star\in\Omega_{\eps,\delta}$.
So, by Corollary \ref{H UPPER},
$$ \min_{\Omega_{\eps,\delta}} {\mathcal{H}}_\eps
\le {\mathcal{H}}_\eps(\xi_\star)\le \frac{C}{d_\star^{n+4s}}=
\frac{C\eps^{n+4s}}{\delta_\star^{n+4s}}=c_1\eps^{n+4s},$$
for a suitable $c_1>0$.
On the other hand, by Corollary \ref{H BELOW},
$$ \min_{\partial\Omega_{\eps,\delta}} {\mathcal{H}}_\eps
\ge \frac{c\eps^{n+4s}}{\delta^{n+4s}},$$
which implies the desired result for $\delta$ appropriately small.
\end{proof}

\section{Estimates on $\bar u_\xi$
and first approximation of the solution}\label{S3:u}

Now we make some estimates on the function $\bar u_\xi$ introduced
in \eqref{EQ bar u}, by using  in particular
the auxiliary function $\Pi_\eps$ in \eqref{EQ Pi}.
For this, we define, for any $\xi\in\Omega_{\eps}$,
\begin{equation}\label{lambda}
\Lambda_{\xi}(x):=\int_{\R^n\setminus \Omega_\eps} w_\xi^p(y) \Gamma(x-y)\,dy.
\end{equation}
We have the following estimate for $\Lambda_{\xi}$:
\begin{lem}\label{79}
Let $x$, $\xi\in\Omega_\eps$.
Assume that $d:=\dist(\xi,\partial\Omega_\eps)\ge1$.
Then
$$ 0\le \Lambda_{\xi}(x) \le \frac{C}{d^{(n+2s)p}},$$
where $C>0$ depends on $n$, $p$, $s$ and $\Omega$.
\end{lem}

\begin{proof} If $y\in\R^n\setminus \Omega_\eps$ then $|y-\xi|\ge \dist(\xi,\partial\Omega_\eps)\ge1$ therefore, by \eqref{decay w},
$$ |w_\xi(y)|=|w(y-\xi)|\le C|y-\xi|^{-(n+2s)} \le Cd^{-(n+2s)}.$$
As a consequence of this, and recalling \eqref{EQ3}, we deduce that
\begin{equation*}
\int_{\R^n\setminus \Omega_\eps} w_\xi^p(y) \Gamma(x-y)\,dy
\le \Big(Cd^{-(n+2s)}\Big)^p\int_{\R^n\setminus \Omega_\eps}\Gamma(x-y)\,dy\le
\Big(Cd^{-(n+2s)}\Big)^p.
\qedhere\end{equation*}
\end{proof}

\begin{lem}\label{lem 5.2}
Let $x$, $\xi\in\Omega_\eps$.
Assume that $d:=\dist(\xi,\partial\Omega_\eps)\ge 1$.
Then
\begin{equation}\label{24}
\bar u_\xi(x)=w_\xi(x)-
\Lambda_{\xi}(x)-
\Pi_\eps(x,\xi)\end{equation}
and
\begin{equation}\label{25}
0\le
w_\xi(x)-\bar u_\xi(x)-
\Pi_\eps(x,\xi) \le \frac{C}{d^{(n+2s)p}},\end{equation}
for a suitable $C>0$ that
depends on $n$, $p$, $s$ and $\Omega$.
\end{lem}

\begin{proof} First of all, notice that $w=w^p * \Gamma$, since
they both satisfy \eqref{EQ w}, thanks to \eqref{EQ Gamma},
and uniqueness holds. As a consequence
\begin{equation}\label{76}
w_\xi(x)=w(x-\xi)=\int_{\R^n} w^p(x-\xi-y) \Gamma(y)\,dy=
\int_{\R^n} w^p_\xi(y) \Gamma(x-y)\,dy.\end{equation}
Similarly, recalling \eqref{EQ bar u}, \eqref{EQ G}
and the symmetry of $G_\eps$, we see that
\begin{eqnarray*}
\bar u_\xi(x) &=& \int_{\Omega_\eps} \bar u_\xi(z) \delta_x(z)\,dz \\
&=& \int_{\Omega_\eps} \bar u_\xi(z) ( (-\Delta)^s+1) G_\eps(z,x)\,dz \\
&=& \int_{\Omega_\eps} w_\xi^p(z) G_\eps(x,z)\,dz
\\ &=& \int_{\Omega_\eps} w_\xi^p(z) \Gamma(x-z)\,dz-
\int_{\Omega_\eps} w_\xi^p(z) H_\eps(x,z)\,dz\\
&=&\int_{\R^n} w_\xi^p(z) \Gamma(x-z)\,dz-
\int_{\R^n\setminus \Omega_\eps} w_\xi^p(z) \Gamma(x-z)\,dz-
\int_{\Omega_\eps} w_\xi^p(z) H_\eps(x,z)\,dz.
\end{eqnarray*}
This, \eqref{EQ Pi}, \eqref{lambda} and \eqref{76} imply \eqref{24},
which, together with Lemma \ref{79}, implies \eqref{25}.
\end{proof}

\begin{lem}\label{lem1}
Let $\xi\in\Omega_\eps$. Assume that $d:=\dist(\xi,\partial\Omega_\eps)\ge 2$. Then
$$ \int_{\Omega_{\eps}} w_{\xi}^{p-1}(x)\, \Lambda_{\xi}(x)\, \Pi_{\eps}(x,\xi)\, dx \le
\frac{C}{d^{(n+2s)p+2s}}, $$
for a suitable $C>0$ that depends on $n$, $p$, $s$ and $\Omega$.
\end{lem}

\begin{proof}
First of all, we notice that for $y\in\R^n\setminus\Omega_{\eps}$ we have
$|y-\xi|\ge d>1$, and therefore, thanks to~\eqref{decay w},
$$ |w_{\xi}(y)|=|w(y-\xi)|\le C|y-\xi|^{-(n+2s)}\le C\, d^{-(n+2s)}. $$
Hence, recalling \eqref{lambda},
\begin{equation}\begin{split}\label{lam}
&\int_{\Omega_\eps}w_{\xi}^{p-1}(x)\Lambda_\xi(x)\, \Pi_\eps(x,\xi)\, dx \\ =& \int_{\Omega_\eps}w_{\xi}^{p-1}(x)\left(\int_{\R^n\setminus\Omega_\eps}w_\xi^p(y)\Gamma(x-y)\, dy\right)\Pi_\eps(x,\xi)\, dx \\
\le&\, C d^{-(n+2s)p}\int_{\Omega_\eps}dx \int_{\R^n\setminus\Omega_\eps } dy \, w_\xi^{p-1}(x)\Gamma(x-y)\Pi_{\eps}(x,\xi) \\
\le&\,  C d^{-(n+2s)p}\int_{\left\lbrace|x-\xi|\le d/4\right\rbrace}dx \int_{\R^n\setminus\Omega_\eps } dy \, w_\xi^{p-1}(x)\Gamma(x-y)\Pi_{\eps}(x,\xi) \\
&\qquad + C d^{-(n+2s)p}\int_{\left\lbrace |x-\xi|>d/4\right\rbrace}dx \int_{\R^n\setminus\Omega_\eps } dy \, w_\xi^{p-1}(x)\Gamma(x-y)\Pi_{\eps}(x,\xi) \\
=:& I_1+I_2.
\end{split}\end{equation}
Now, thanks to \eqref{25}, we have that $\Pi_\eps(x,\xi)\le w_{\xi}(x)$,
and so
\begin{equation}\label{dopo}
w_{\xi}^{p-1}(x)\, \Pi_\eps(x,\xi)\le w_{\xi}^{p}(x).
\end{equation}
Therefore, $I_1$ can be estimated as follows:
\begin{eqnarray*}
I_1&\le& C d^{-(n+2s)p} \int_{\left\lbrace|x-\xi|\le d/4\right\rbrace}dx \int_{\R^n\setminus\Omega_\eps} dy \, w_\xi^{p}(x)\, \Gamma(x-y) \\
&\le& C d^{-(n+2s)p} \int_{\left\lbrace|x-\xi|\le d/4\right\rbrace}dx \int_{\R^n\setminus B_{d/2}(\xi)} dy \, w_\xi^{p}(x)\, \Gamma(x-y).
\end{eqnarray*}
We notice that, in the above domain,
$$ |x-y|\ge|y-\xi|-|x-\xi|\ge\frac{d}{2}-\frac{d}{4}=\frac{d}{4}, $$
hence
$$ \Gamma(x-y)\le \frac{C}{|x-y|^{n+2s}}.$$
Now, we can compute in polar coordinates the following integral
$$ \int_{\R^n\setminus B_{d/2}(\xi)} \frac{1}{|x-y|^{n+2s}}dy\le \frac{C}{d^{2s}}, $$
up to renaming the constant $C$.
This and the fact that $w_{\xi}^p$ is integrable give
\begin{equation}\begin{split}\label{I1}
I_1&\le C_1 d^{-(n+2s)p} \int_{\left\lbrace|x-\xi|\le d/4\right\rbrace} dx \int_{\R^n\setminus B_{d/2}(\xi)} dy \, \frac{w_\xi^{p}(x)}{|x-y|^{n+2s}}\\
&\le C_2 d^{-(n+2s)p} d^{-2s}\int_{\left\lbrace|x-\xi|\le d/4\right\rbrace}w_\xi^{p}(x)\, dx\\
&\le C_3 d^{-(n+2s)p} d^{-2s},
\end{split}\end{equation}
for suitable $C_1,C_2,C_3>0$.
Now, if $|x-\xi|>d/4$ then, thanks to \eqref{decay w},
$$ |w_{\xi}(x)|=|w(x-\xi)|\le C|x-\xi|^{-(n+2s)}. $$
This together with \eqref{dopo} and \eqref{EQ3} implies that
\begin{equation}\begin{split}\label{I2}
I_2&\le C d^{-(n+2s)p}\int_{\left\lbrace |x-\xi|>d/4\right\rbrace}dx \int_{\R^n\setminus\Omega_\eps } dy \, w_{\xi}^{p}(x)\Gamma(x-y) \\ &\le
C' d^{-(n+2s)p}\int_{\left\lbrace |x-\xi|>d/4\right\rbrace}dx \int_{\R^n\setminus\Omega_\eps} dy \, \frac{\Gamma(x-y)}{|x-\xi|^{(n+2s)p}} \\ &\le
C'' d^{-(n+2s)p} d^{-(n+2s)p+n},
\end{split}\end{equation}
for suitable $C',C''>0$,
where in the last inequality we have computed the integral in~$dx$
in polar coordinates and used~\eqref{EQ3}.
Putting together \eqref{I1} and \eqref{I2} and recalling~\eqref{lam}
we get the desired estimate.
\end{proof}

\begin{lem}\label{lem2}
Let $\xi\in\Omega_\eps$ and $\tilde p:=\min\{p,2\}$. Assume that $d:=\dist(\xi,\partial\Omega_\eps)\ge 2$. Then
$$ \int_{\Omega_{\eps}} w_{\xi}^{p-1}(x)\, \Pi_{\eps}^2(x,\xi)\, dx \le
\frac{C}{d^{\tilde p(n+2s)+2s}}, $$
for a suitable $C>0$ that depends on $n$, $p$, $s$ and $\Omega$.
\end{lem}

\begin{proof} First we observe that
\begin{equation}\label{p til}
\begin{split}
{\mbox{if $p\ge2$ then }}& 2(n+4s)\ge \tilde p(n+4s)=\tilde p(n+2s+2s)>\tilde p(n+2s)+2s,\\
{\mbox{if $1<p<2$ then }}& p(n+4s)\ge \tilde p(n+4s)=\tilde p(n+2s+2s)>\tilde p(n+2s)+2s,\\
&(n+2s)(p+1)-n=p (n+2s)+2s\ge \tilde p (n+2s)+2s.
\end{split}
\end{equation}
Now, we can write the integral that we want to estimate as
\begin{equation}\begin{split}\label{pi5}
&\int_{\Omega_{\eps}} w_{\xi}^{p-1}(x)\, \Pi_{\eps}^2(x,\xi)\, dx \\
=& \int_{\left\lbrace|x-\xi|\le d/8\right\rbrace}w_{\xi}^{p-1}(x)\, \Pi_{\eps}^2(x,\xi)\, dx+ \int_{\left\lbrace|x-\xi|>d/8 \right\rbrace} w_{\xi}^{p-1}(x)\, \Pi_{\eps}^2(x,\xi)\, dx \\
=:& I_1 + I_2.
\end{split}\end{equation}
If $p\ge2$, to estimate $I_1$ we use Lemma~\ref{EST PI} together with the fact that
$w_{\xi}^{p-1}$ is integrable to get
\begin{equation}\label{pi3}
I_1\le \frac{C}{d^{2(n+4s)}}.
\end{equation}
If $1<p<2$, we notice that, thanks to \eqref{25},
$\Pi_\eps(x,\xi)\le w_\xi(x)$ and so
$$ w_{\xi}^{p-1}(x)\, \Pi_{\eps}^2(x,\xi)= w_{\xi}^{p-1}(x)\, \Pi_{\eps}^{2-p}(x,\xi)\, \Pi_{\eps}^p(x,\xi)\le w_{\xi}^{p-1}(x)\, w_{\xi}^{2-p}(x)\, \Pi_{\eps}^p(x,\xi) =w_{\xi}(x)\, \Pi_{\eps}^p(x,\xi). $$
Therefore, using again Lemma~\ref{EST PI} and the fact that $w_\xi$ is integrable, we obtain
\begin{equation}\begin{split}\label{pi4}
I_1&\le \int_{\left\lbrace|x-\xi|\le d/8\right\rbrace}w_{\xi}(x)\, \Pi_{\eps}^p(x,\xi)\, dx\\
&\le \frac{C}{d^{p(n+4s)}}\int_{\left\lbrace|x-\xi|\le d/8\right\rbrace}w_{\xi}(x)\, dx\\
&\le \frac{C}{d^{p(n+4s)}}.
\end{split}\end{equation}
To estimate $I_2$, we use \eqref{25} to obtain that $\Pi_{\eps}(x,\xi)\le w_{\xi}(x)$,
and so $w_{\xi}^{p-1}(x)\Pi_\eps^2(x,\xi)\le w_\xi^{p+1}(x)$. This implies that
$$ I_2\le \int_{\left\lbrace|x-\xi|>d/8 \right\rbrace} w_{\xi}^{p+1}(x)\, dx. $$
Since $|x-\xi|>d/8$, thanks to \eqref{decay w}, we have that
$$ |w_\xi(x)|=|w(x-\xi)|\le C|x-\xi|^{-(n+2s)}. $$
Therefore, computing the integral in polar coordinates,
\begin{equation}\label{pi2}
I_2 \le \int_{\left\lbrace|x-\xi|>d/8 \right\rbrace}\frac{C}{|x-\xi|^{(n+2s)(p+1)}}dx
\le \frac{C}{d^{(n+2s)(p+1)-n}}.
\end{equation}
Putting together \eqref{pi3}, \eqref{pi4} and \eqref{pi2} and recalling~\eqref{pi5},
we obtain the result (one can use \eqref{p til} to
obtain a simpler common exponent).
\end{proof}

\begin{lem}\label{lem3}
Let $\delta\in(0,1)$. Let $\xi\in\Omega_\eps$ such
that $d:=\dist(\xi,\partial\Omega_\eps)\ge \delta/\eps$. Then
$$ \int_{\Omega_{\eps}} w_{\xi}^{p-1}(x)\, \Lambda_{\xi}^2(x)\, dx \le
C\, \eps^{2p(n+2s)-n}, $$
for a suitable $C>0$ that depends on $n$, $p$, $s$, $\delta$ and $\Omega$.
\end{lem}

\begin{proof}
We use Lemma \ref{79} and the fact that $\Omega$ is bounded to obtain that
\begin{eqnarray*}
\int_{\Omega_{\eps}} w_{\xi}^{p-1}(x)\, \Lambda_{\xi}^2(x)\, dx &\le& \frac{C}{d^{2p(n+2s)}}\int_{\Omega_{\eps}} w_{\xi}^{p-1}(x)\, dx \\
&\le&\frac{C'}{d^{2p(n+2s)}}|\Omega_{\eps}|\le \frac{C''}{d^{2p(n+2s)}\eps^n}
\le\frac{C''\, \eps^{2p(n+2s)}}{\delta^{2p(n+2s)}\eps^n},
\end{eqnarray*}
for suitable $C',C''>0$. This implies the desired estimate.
\end{proof}

\section{Energy estimates and functional expansion in~$\bar u_\xi$}\label{S:energy}

In this section we make some estimates for the energy functional \eqref{energy functional}.
For this, we consider the functional associated to problem~\eqref{EQ w}:
\begin{equation}\label{entire}
I(u)=\frac{1}{2}\int_{\R^n}\left((-\Delta)^su(x)\, u(x)+u^2(x)\right) dx-\frac{1}{p+1}\int_{\R^n}u^{p+1}(x)\, dx, \qquad u\in H^s(\R^n).
\end{equation}

\begin{thm}\label{TH expansion}
Fix $\delta\in(0,1)$ and $\xi\in\Omega_{\eps}$ such that $d:=\dist(\xi,\partial\Omega_{\eps})\ge\delta/\eps$.
Then, we have
\begin{equation}\label{estI}
I_{\eps}(\bar u_\xi)= I(w)+\frac{1}{2}{\mathcal{H}}_\eps(\xi)+o(\eps^{n+4s}),
\end{equation}
as $\eps\to0$,
where $I$ is given by~\eqref{entire}, $w$ is the solution to~\eqref{EQ w}
and ${\mathcal{H}}_\eps(\xi)$ is defined in \eqref{def H cal},
as long as $\delta$ is sufficiently small.
\end{thm}

The following simple observation will be used often in the sequel:

\begin{lem}\label{w use}
Let $\delta\ge1$ and $q>1$. Then
$$ \int_{\R^n\setminus B_\delta(\xi)} w_\xi^q(z)\,dz
\le \frac{C}{\delta^{n(q-1)+2sq}},$$
for some $C>0$.
\end{lem}

\begin{proof} First of all we observe that
$$ n-1-(n+2s)q < n-1-(n+2s)=-1-2s<-1$$
and therefore
\begin{equation}\label{Ro}
\int_\delta^{+\infty} \rho^{n-1-(n+2s)q}\,d\rho=
\frac{\delta^{n-(n+2s)q}}{(n+2s)q-n}.
\end{equation}
Now, we use \eqref{decay w} to see that
$$\int_{\R^n\setminus B_\delta(\xi)} w_\xi^q(z)\,dz \le
\int_{\R^n\setminus B_\delta(\xi)} \frac{C}{|x-\xi|^{(n+2s)q}}\,dz
= C'\int_\delta^{+\infty} \rho^{n-1-(n+2s)q}\,d\rho$$
for some $C'>0$.
This and \eqref{Ro} imply the desired result.
\end{proof}

\begin{cor}\label{Cor:J1}
Let $\xi\in\Omega_\eps$, with $d:=\dist(\xi,\partial\Omega_\eps)\ge1$.
Then
$$ \int_{\R^n\setminus\Omega_\eps} w_\xi^{p+1}(z)\,dz\le\frac{C}{d^{np+2s(p+1)}},$$
for some $C>0$.
\end{cor}

\begin{proof}
Notice that $(\R^n\setminus\Omega_\eps)\subseteq(\R^n\setminus B_d(\xi))$
and exploit Lemma \ref{w use}.
\end{proof}

\begin{proof}[Proof of Theorem~\ref{TH expansion}]
Using \eqref{EQ bar u} and \eqref{24}, we have
\begin{equation}\begin{split}\label{esp1}
I_\eps(\bar u_\xi)=&\, \frac{1}{2}\int_{\Omega_\eps}\left((-\Delta)^s\bar u_\xi(x)+\bar u_\xi(x)\right)\bar u_\xi(x)\, dx-\frac{1}{p+1}\int_{\Omega_\eps}\bar u_\xi^{p+1}(x)\, dx\\
=&\, \frac{1}{2}\int_{\Omega_\eps}w_\xi^p(x)\, \bar u_\xi(x)\, dx -\frac{1}{p+1}\int_{\Omega_\eps}\bar u_\xi^{p+1}(x)\, dx\\
=&\, \frac{1}{2}\int_{\Omega_\eps}w_\xi^p(x)\left(w_\xi(x)-\Lambda_\xi(x)-\Pi_\eps(x,\xi)\right)dx\\&\qquad -\frac{1}{p+1}\int_{\Omega_\eps}\left(w_\xi(x)-\Lambda_\xi(x)-\Pi_\eps(x,\xi)\right)^{p+1}dx\\
=&\, \left(\frac{1}{2}-\frac{1}{p+1}\right)\int_{\Omega_\eps}w_\xi^{p+1}(x)\, dx -\frac{1}{2}\int_{\Omega_\eps}w_\xi^p(x)\left(\Lambda_\xi(x)+\Pi_\eps(x,\xi)\right)dx\\
&\qquad +\frac{1}{p+1}\int_{\Omega_\eps}\left[w_\xi^{p+1}(x)-\left(w_\xi(x)-\Lambda_\xi(x)-\Pi_\eps(x,\xi)\right)^{p+1}\right]dx.
\end{split}\end{equation}
We notice that the first term in the right hand side of~\eqref{esp1} can be written as
\begin{eqnarray*}
&&\left(\frac{1}{2}-\frac{1}{p+1}\right)\int_{\Omega_\eps}w_\xi^{p+1}(x)\, dx \\
&=&\left(\frac{1}{2}-\frac{1}{p+1}\right)\int_{\R^n}w_\xi^{p+1}(x)\, dx -\left(\frac{1}{2}-\frac{1}{p+1}\right)\int_{\R^n\setminus\Omega_\eps}w_\xi^{p+1}(x)\, dx\\
&=& I(w)-\left(\frac{1}{2}-\frac{1}{p+1}\right)\int_{\R^n\setminus\Omega_\eps}w_\xi^{p+1}(x)\, dx,
\end{eqnarray*}
since $w$ is a solution to~\eqref{EQ w}.
Therefore,
\begin{equation}\begin{split}\label{Ieps}
I_\eps(\bar u_\xi)&= I(w)-\left(\frac{1}{2}-\frac{1}{p+1}\right)\int_{\R^n\setminus\Omega_\eps}w_\xi^{p+1}(x)\, dx\\
&\qquad -\frac{1}{2}\int_{\Omega_\eps}w_\xi^p(x)\left(\Lambda_\xi(x)+\Pi_\eps(x,\xi)\right)dx\\
&\qquad +\frac{1}{p+1}\int_{\Omega_\eps}\left[w_\xi^{p+1}(x)-\left(w_\xi(x)-\Lambda_\xi(x)-\Pi_\eps(x,\xi)\right)^{p+1}\right]dx\\
&=:I(w)-\left(\frac{1}{2}-\frac{1}{p+1}\right)J_1-\frac{1}{2}J_2+\frac{1}{p+1}J_3,
\end{split}\end{equation}
where
\begin{eqnarray*}
&& J_1:=\int_{\R^n\setminus\Omega_\eps}w_\xi^{p+1}(x)\, dx,\\
&& J_2:=\int_{\Omega_\eps}w_\xi^p(x)\left(\Lambda_\xi(x)+\Pi_\eps(x,\xi)\right)dx\\
{\mbox{and }}&& J_3:=\int_{\Omega_\eps}\left[w_\xi^{p+1}(x)-\left(w_\xi(x)-\Lambda_\xi(x)-\Pi_\eps(x,\xi)\right)^{p+1}\right]dx
.\end{eqnarray*}
Now, we estimate separately $J_1$, $J_2$ and $J_3$.
Thanks to Corollary \ref{Cor:J1}, we have that
\begin{equation}\label{J1final}
J_1=\int_{\R^n\setminus\Omega_\eps}w_\xi^{p+1}(x)\, dx\le \frac{C}{d^{np+2s(p+1)}}\le \frac{C}{\delta^{np+2s(p+1)}}\eps^{np+2s(p+1)}.
\end{equation}

Concerning $J_2$, we write it as
$$ J_2= J_{21}+J_{22}, $$
where
\begin{equation}\begin{split}\label{J21}
J_{21}&:=\int_{\Omega_\eps}w_\xi^p(x)\, \Lambda_\xi(x)\, dx,\\
J_{22}&:=\int_{\Omega_\eps}w_\xi^p(x)\, \Pi_\eps(x,\xi)\, dx.
\end{split}\end{equation}
Recalling the definition of $\Lambda_\xi$ in~\eqref{lambda} and the estimate in~\eqref{decay w}, we have that
\begin{equation}\begin{split}\label{tre}
J_{21}=&\int_{\Omega_\eps}dx\int_{\R^n\setminus\Omega_\eps}dy\, w_\xi^p(x)\, w_\xi^p(y)\, \Gamma(x-y)\\
\le&\, C\int_{\Omega_\eps}dx\int_{\R^n\setminus\Omega_\eps}dy\, w_\xi^p(x)\frac{\Gamma(x-y)}{|y-\xi|^{(n+2s)p}}\\
\le&\, \frac{C}{d^{(n+2s)p}}\int_{\Omega_\eps}dx\int_{\R^n\setminus\Omega_\eps}dy\, w_\xi^p(x)\, \Gamma(x-y)\\
=&\, \frac{C}{d^{(n+2s)p}}\left(\int_{\Omega_\eps}dx\int_{{\R^n\setminus\Omega_\eps}\atop{\{|x-y|\le d/2\}}}dy\, w_\xi^p(x)\, \Gamma(x-y) +\int_{\Omega_\eps}dx\int_{{\R^n\setminus\Omega_\eps}\atop{\{|x-y|>d/2\}}}dy\, w_\xi^p(x)\, \Gamma(x-y)\right).
\end{split}\end{equation}
We notice that, if $x\in\Omega_\eps$ and $y\in\R^n\setminus\Omega_\eps$ with $|x-y|\le d/2$,
then
$$ |x-\xi|\ge |y-\xi|-|x-y|\ge d-\frac{d}{2}=\frac{d}{2}. $$
Therefore, using \eqref{decay w}, \eqref{EQ3} and the fact that $\Omega$ is bounded, we have
\begin{equation}\begin{split}\label{uno}
\int_{\Omega_\eps}dx\int_{{\R^n\setminus\Omega_\eps}\atop{\{|x-y|\le d/2\}}}dy\, w_\xi^p(x)\, \Gamma(x-y) &\le C' \int_{\Omega_\eps}dx\int_{{\R^n\setminus\Omega_\eps}\atop{\{|x-y|\le d/2\}}}dy\, \frac{\Gamma(x-y)}{|x-\xi|^{(n+2s)p}}\\
&\le C'\int_{\Omega_\eps}dx\int_{\R^n}d\tilde y\frac{\Gamma(\tilde y)}{|x-\xi|^{(n+2s)p}}\\
&\le C''(1/d)^{(n+2s)p}|\Omega_\eps|\\
&\le \frac{C'''}{d^{(n+2s)p}\eps^{n}}\\
&\le \frac{C'''}{\delta^{(n+2s)p}}\eps^{(n+2s)p-n},
\end{split}\end{equation}
for suitable constants $C',C'',C'''>0$.
Moreover, if $|x-y|>d/2$, we use \eqref{GA} to get
\begin{equation}\begin{split}\label{due}
\int_{\Omega_\eps}dx\int_{{\R^n\setminus\Omega_\eps}\atop{\{|x-y|>d/2\}}}dy\, w_\xi^p(x)\, \Gamma(x-y)&\le C \int_{\R^n}dx\int_{{\R^n\setminus\Omega_\eps}\atop{\{|x-y|>d/2\}}}dy\, \frac{w_\xi^p(x)}{|x-y|^{n+2s}}\\&\le \tilde C d^{-2s}\le \frac{\tilde C}{\delta^{2s}}\epsilon^{2s},
\end{split}\end{equation}
for some $\tilde C>0$.
Putting together \eqref{uno} and \eqref{due} and recalling \eqref{tre}, we obtain
\begin{equation}\begin{split}\label{J21final}
J_{21}&\le \frac{C}{d^{(n+2s)p}}\left(\frac{C'''}{\delta^{(n+2s)p}}\eps^{(n+2s)p-n} +\frac{\tilde C}{\delta^{2s}}\epsilon^{2s}\right)\\&\le \frac{C}{\delta^{(n+2s)p}}\eps^{(n+2s)p}\left(\frac{C'''}{\delta^{(n+2s)p}}\eps^{(n+2s)p-n} +\frac{\tilde C}{\delta^{2s}}\epsilon^{2s}\right) \le \hat C \eps^{np+2s(p+1)},
\end{split}\end{equation}
for suitable $\hat C>0$. Therefore,
\begin{equation}\label{J2final}
J_2 = J_{22}+ o(\eps^{n+4s})=\int_{\Omega_\eps}w_\xi^p(x)\, \Pi_\eps(x,\xi)\, dx+o(\eps^{n+4s}).
\end{equation}

To estimate $J_3$ we expand $w_\xi^{p+1}(x)$ in the following way
$$ w^{p+1}_\xi(x)= \bar u_\xi^{p+1}(x)+(p+1)w_\xi^p(x)(w_\xi(x)-\bar u_\xi(x))+ c_p\alpha_\xi^{p-1}(x)(w_\xi(x)-\bar u_\xi(x))^2, $$
where $0\le\bar u_\xi\le\alpha_\xi\le w_\xi$ and $c_p$ is a positive constant
depending only on $p$. Therefore, recalling \eqref{24} and \eqref{J21},
\begin{equation}\begin{split}\label{J3}
J_3=&\,\int_{\Omega_\eps}\left[w_\xi^{p+1}(x)-\left(w_\xi(x)-\Lambda_\xi(x)-\Pi_\eps(x,\xi)\right)^{p+1}\right]dx \\
=&\,\int_{\Omega_\eps}\left[w_\xi^{p+1}(x)-\bar u_\xi^{p+1}(x)\right]dx \\
=&\, (p+1) \int_{\Omega_\eps}w^p_\xi(x)\left(\Lambda_\xi(x)+\Pi_\eps(x,\xi)\right) dx +
c_p\int_{\Omega_\eps}\alpha_\xi^{p-1}(x)\left(\Lambda_\xi(x)+\Pi_\eps(x,\xi)\right)^2 dx \\
=&\, (p+1) (J_{21}+J_{22}) + c_p\int_{\Omega_\eps}\alpha_\xi^{p-1}(x)\left(\Lambda_\xi(x)+\Pi_\eps(x,\xi)\right)^2 dx.
\end{split}\end{equation}
Since $\alpha_\xi(x)\le w_\xi(x)$, we have that
\begin{eqnarray*}
&&\int_{\Omega_\eps}\alpha_\xi^{p-1}(x)\left(\Lambda_\xi(x)+\Pi_\eps(x,\xi)\right)^2 dx\\
&\le&\int_{\Omega_\eps}w_\xi^{p-1}(x)\left(\Lambda_\xi(x)+\Pi_\eps(x,\xi)\right)^2 dx \\
&=& \int_{\Omega_\eps}w_\xi^{p-1}(x)\, \Lambda^2_\xi(x)\, dx
+\int_{\Omega_\eps}w_\xi^{p-1}(x)\, \Pi^2_\eps(x,\xi)\,  dx + 2\int_{\Omega_\eps} w_\xi^{p-1}(x)\, \Lambda_\xi(x)\, \Pi_\eps(x,\xi)\, dx.
\end{eqnarray*}
Hence, from Lemmata \ref{lem1}, \ref{lem2} and \ref{lem3} (together with the fact that
$d\ge\delta/\eps$) we deduce that
$$ \int_{\Omega_\eps}\alpha_\xi^{p-1}(x)\left(\Lambda_\xi(x)+\Pi_\eps(x,\xi)\right)^2 dx\le C_{\delta}(\eps^{2p(n+2s)-n}+ \eps^{(n+2s)\tilde p+2s}+ \eps^{(n+2s)p+2s}),
$$
for some $C_{\delta}$, where $\tilde p=\min\{p,2\}$. The last estimate, \eqref{J21final} and~\eqref{J3} give
\begin{equation}\label{J3final}
J_3 =(p+1)J_{22}+ o(\epsilon^{n+4s}) =(p+1)\int_{\Omega_\eps}w_\xi^p(x)\, \Pi_\eps(x,\xi)\, dx+ o(\epsilon^{n+4s}).
\end{equation}
Putting together \eqref{J1final}, \eqref{J2final} and \eqref{J3final} and
using \eqref{Ieps}, we get
$$ I_{\eps}(\bar u_\xi)=I(w)+ \frac{1}{2}\int_{\Omega_\eps}w_\xi^p(x)\, \Pi_\eps(x,\xi)\, dx +o(\eps^{n+4s}). $$
Thus, recalling the definitions of $\Pi_{\eps}$ and ${\mathcal{H}}_\eps$ in \eqref{EQ Pi} and
\eqref{def H cal} respectively, we obtain \eqref{estI}.
\end{proof}

\section{Decay of the ground state $w$}\label{S:a1}

In this section we recall some basic (though not optimal)
decay properties of the ground state and of its derivatives.

For this, we start with a general convolution result:

\begin{lem}\label{DF4}
Let~$a$, $b>n$, $C_a$, $C_b>0$ and~$f$, $g\in L^\infty(\R^n)$, with
$${\mbox{$|f(x)|\le C_a\,(1+|x|)^{-a}$
and~$|g(x)|\le C_b\,(1+|x|)^{-b}$.}}$$ Then there exists~$C>0$ such that
$$ \big|(f*g)(x)\big|\le C\,(1+|x|)^{-c},$$
with~$c:=\min\{a,\,b\}$.
\end{lem}

\begin{proof} Fix~$x\in\R^n$ and~$r:=|x|/2$, and observe that if~$y\in
B_{r}(x)$ then
$$ |y|\ge |x|-|x-y|\ge |x|-r=\frac{|x|}{2}.$$
As a consequence
\begin{equation}\label{Dq1}\begin{split}
& \int_{B_r(x)} \frac{C_a}{(1+|x-y|)^a}\,
\frac{C_b}{(1+|y|)^b}\,dy
\le \int_{B_r(x)} \frac{C_a}{(1+|x-y|)^a}\,
\frac{C_b}{(1+(|x|/2))^b}\,dy\\
&\qquad\le \frac{C}{(1+|x|)^b}\,
\int_{\R^n}
\frac{1}{(1+|x-y|)^a}\,dy
\le \frac{C}{(1+|x|)^b},
\end{split}\end{equation}
up to renaming constants. On the other hand,
if~$y\in \R^n\setminus B_r(x)$ then~$|x-y|\ge r=|x|/2$, thus
\begin{equation}\label{Dq2}\begin{split}
& \int_{\R^n\setminus B_r(x)} \frac{C_a}{(1+|x-y|)^a}\,
\frac{C_b}{(1+|y|)^b}\,dy
\le \int_{\R^n\setminus B_r(x)} \frac{C_a}{(1+(|x|/2))^a}\,
\frac{C_b}{(1+|y|)^b}\,dy\\
&\qquad\le \frac{C}{(1+|x|)^a}\,
\int_{\R^n}
\frac{1}{(1+|y|)^b}\,dy
\le \frac{C}{(1+|x|)^a}.
\end{split}\end{equation}
%Also,
%$$ \big|(f*g)(x)\big| \le \int_{\R^n} |f(x-y)|\,|g(y)|\,dy
%\le\int_{\R^n} \frac{C_a}{(1+|x-y|)^a}\,
%\frac{C_b}{(1+|y|)^b}\,dy.$$
Putting together \eqref{Dq1} and~\eqref{Dq2}
we obtain the desired result.
\end{proof}

Now we fix~$\xi\in\Omega_\epsilon$ and we define,
for any~$i\in\{1,\ldots,n\}$,
\begin{equation}\label{Zj}
Z_i:=\frac{\partial w_\xi}{\partial x_i},
\end{equation}
where~$w_\xi$ is the ground state solution centered at~$\xi$.
Moreover, we denote by~$\mathcal Z$ the linear space spanned by the functions~$Z_i$.

We prove first the following lemmata:

\begin{lem}\label{:decay Z}
There exists a positive constant~$C$ such that, for any~$i=1,\ldots,n$,
$$ |Z_i|\le C|x-\xi|^{-\nu_1}, {\mbox{ for any }}|x-\xi|\ge 1,$$
where~$\nu_1:=\min\{(n+2s+1),\,p(n+2s) \}$.
\end{lem}

\begin{proof} Given~$R>0$, we take~$\Gamma_{1,R}
\in C^\infty(\R^n)$, with~$0\le\Gamma_{1,R}\le\Gamma$ in~$\R^n$ and
$\Gamma_{1,R}=\Gamma$
outside~$B_R$, and we define~$\Gamma_{2,R}:=\Gamma-\Gamma_{1,R}$.
We use~\eqref{EQ Gamma} to write
\begin{equation}\label{w 0}
w=\Gamma*w^p=\Gamma_{1,R}*w^p+\Gamma_{2,R}*w^p.\end{equation}
We assume, up to translation, that $\xi=0$.
Then, our goal is to prove that, for any~$k\in\N$ we have that
\begin{equation}\label{TOP}
\left|
\frac{\partial w}{\partial x_i}(x)\right|\leq C_k ( 1+|x|)^{-\nu(k)},
\end{equation}
where
$$ \nu(k):=\min\{(n+2s+1),\ p(n+2s), \ k(p-1)(n+2s) \}=\min\{
\nu_1, \ k(p-1)(n+2s) \},$$
for some~$C_k>0$.
Indeed, the desired claim would follows from~\eqref{TOP}
simply by taking the smallest~$k$ for which~$k(p-1)>p$.

To prove~\eqref{TOP} we perform an inductive argument.
So, we first check~\eqref{TOP} when~$k=0$.
For this, we use the fact that~$w\in L^\infty(\R^n)$
and that~$\Gamma\in L^1(\R^n)$ to find~$R>0$ sufficiently
small that
$$ \int_{B_R}\Gamma(y)\,dy\le \frac{1}{2p\,\|w\|_{L^\infty(\R^n)}}.$$
This fixes~$R$ once and for all for the
proof of~\eqref{TOP} when~$k=0$.
Hence, we use the sign of~$\Gamma_{1,R}$ and the fact that~$\Gamma_{2,R}=0$
outside $B_R$ to obtain that
\begin{equation}\label{ w 2}
\int_{\R^n}\Gamma_{2,R}(y)\,dy=\int_{B_R}\Gamma_{2,R}(y)\,dy
\le \int_{B_R}\Gamma_{1,R}(y)+\Gamma_{2,R}(y)\,dy\le
\frac{1}{2p\,\|w\|_{L^\infty(\R^n)}}.\end{equation}
Then, for any~$t\in(0,1)$,
we define~$D_t w(x):= \big(w(x+te_i)-w(x)\big)/t$
and we infer from~\eqref{w 0} that
\begin{equation}\label{ w 3}
D_t w=(D_t\Gamma_{1,R})*w^p+\Gamma_{2,R}* (D_t w^p).\end{equation}
Also, from formula~(3.2) of~\cite{FQT} we know that
\begin{equation}\label{dd51}
|\nabla \Gamma(x)|\le C\, |x|^{-(n+2s+1)},\quad {\mbox{ for any }}
|x|\geq 1.
\end{equation}
As a consequence, if~$|x|>2$ and~$\eta\in B_1(x)$,
we have that
$$ |\eta|\ge |x|-|x-\eta| \ge \frac{|x|}{2}>1,$$
hence
$$ |\Gamma(x+te_1)-\Gamma(x)|\le
t\,\sup_{\eta\in B_1(x)} |\nabla\Gamma(\eta)|\le
C\,t\, |x|^{-(n+2s+1)},$$
up to renaming~$C$. This gives that~$|D_t\Gamma(x)|\le
C\, (1+|x|)^{-(n+2s+1)}$, and so~$|D_t\Gamma_{1,R}(x)|\le
C\, (1+|x|)^{-(n+2s+1)}$. Accordingly,
we have that
\begin{equation}\label{ w w 51}
|(D_t\Gamma_{1,R} )* w^p|\le \|w\|_{L^\infty(\R^n)}^p
\,\int_{\R^n} |D_t\Gamma_{1,R}(y)|\,dy\le C.\end{equation}
Also
$$ |w^p(x+te_i) - w^p(x)|\le p \,\|w\|_{L^\infty(\R^n)}^{p-1}
\,|w(x+te_i) - w(x)|.$$
This says that
$$ |D_t w^p(x)|\le p \,\|w\|_{L^\infty(\R^n)}^{p-1}\,|D_t w(x)|.$$
Moreover
$$ |D_t w(x)|\le \frac{2\,\|w\|_{L^\infty(\R^n)}}{t},$$
hence we can define
$$ M(t):=\sup_{x\in\R^n} |D_t w(x)|,$$
so we obtain that
$$ |D_t w^p(x)|\le p \,\|w\|_{L^\infty(\R^n)}^{p-1}\,M(t),$$
for every~$x\in\R^n$, and thus
\begin{eqnarray*} && |\Gamma_{2,R}* (D_t w^p)(x)|\le\int_{\R^n}
\Gamma_{2,R}(y)\,|D_t w^p(x-y)|\,dy\\
&&\qquad\le p \,\|w\|_{L^\infty(\R^n)}^{p-1}\,M(t)
\int_{\R^n}
\Gamma_{2,R}(y)\,dy \le
\frac{M(t)}{2},\end{eqnarray*}
thanks to~\eqref{ w 2}. Using this and~\eqref{ w w 51}
into~\eqref{ w 3}, we conclude that
$$ D_t w\le C+\frac{M(t)}{2}.$$
By taking the supremum, we obtain that
$$ M(t)\le C+\frac{M(t)}{2},$$
and this gives, up to renaming~$C$, that~$M(t)\le C$.
By sending~$t\searrow 0$, we complete the proof of~\eqref{TOP}
when~$k=0$.

Now we suppose that~\eqref{TOP} holds true for some~$k$
and we prove it for~$k+1$. The proof is indeed
similar to the case~$k=0$: here
we take~$R:=1$
and use the short notation~$\Gamma_1:=\Gamma_{1,R}$ and~$\Gamma_2:=
\Gamma_{2,R}$.
By~\eqref{TOP} for~$k=0$ and the regularity
theory (applied to the equation for~$D_t w$), we know that~$w\in C^1(\R)$,
hence we can differentiate~\eqref{w 0} and obtain that
\begin{equation}\label{k->1}
\frac{\partial w}{\partial x_i} =
\frac{\partial \Gamma_1}{\partial x_i} *w^p + \Gamma_2 *
\left(p w^{p-1}  \frac{\partial w}{\partial x_i}\right).
\end{equation}
So, we use \eqref{decay w}, \eqref{dd51} and Lemma~\ref{DF4}
to obtain
\begin{equation}\label{k->2}
\left|\frac{\partial \Gamma_1}{\partial x_i} *w^p (x)\right|
\le C(1+|x|)^{-\min\{(n+2s+1),\,p(n+2s) \}}.
\end{equation}
Moreover, we notice that
\begin{eqnarray*}&& (p-1)(n+2s)+\nu(k)=
\min\{(p-1)(n+2s)+
\nu_1, \ (k+1)(p-1)(n+2s) \} \\&&\qquad\ge \min\{
\nu_1, \ (k+1)(p-1)(n+2s) \}=\nu(k+1).\end{eqnarray*}
Hence, using~\eqref{TOP} for~$k$ and~\eqref{decay w}
we see that
\begin{equation}\label{p meno 1}
\left| p w^{p-1}  \frac{\partial w}{\partial x_i}(x)\right|\le
C( 1+|x|)^{-(p-1)(n+2s)-\nu(k)}
\le C( 1+|x|)^{-\nu(k+1)},
\end{equation}
up to renaming constants (possibly depending on~$p$).
Now, we observe that
\begin{equation}\label{x meno y}
{\mbox{if~$x\in\R^n$ and~$|y|<1$ then~$1+|x-y|\ge \displaystyle\frac13(1+|x|)$.}}
\end{equation}
Indeed, if~$|x|\ge 2$ and~$|y|<1$, then
$$ |x-y|\ge |x|-|y|\ge\frac{|x|}{2},$$
which implies~\eqref{x meno y} in this case.
If instead~$|x|<2$ and~$|y|<1$, we have that
$$ 1+|x|<3<3(1+|x-y|),$$
and this finishes the proof of~\eqref{x meno y}.

Therefore, since~$\Gamma_2$ vanishes outside~$B_1$,
using~\eqref{p meno 1} and~\eqref{x meno y}, we have
\begin{eqnarray*}
&& \left|
\Gamma_2 * \left(p w^{p-1}  \frac{\partial w}{\partial x_i}\right)(x)
\right| \leq C\,\int_{B_1} \frac{\Gamma_2(y)}{
(1+|x-y|)^{\nu(k+1)} }\,dy
\\ && \qquad \le
C\,\int_{B_1} \frac{\Gamma_2(y)}{
(1+|x|)^{\nu(k+1)} }\,dy
\le\frac{C}{(1+|x|)^{\nu(k+1)} }
\,\int_{\R^n} \Gamma(y)\,dy = \frac{C}{(1+|x|)^{\nu(k+1)} }.\end{eqnarray*}
This and~\eqref{k->2} establish~\eqref{TOP} for~$k+1$,
thus completing the inductive argument.
\end{proof}

\begin{lem}\label{lem:decay DZ}
There exists a positive constant~$C$ such that, for any~$i=1,\ldots,n$,
$$ |\nabla Z_i|\le C|x-\xi|^{-\nu_2}, {\mbox{ for any }}|x-\xi|\ge 1,$$
where~$\nu_2:=\min\{(n+2s+2),\,p(n+2s) \}$.
\end{lem}

\begin{proof} {F}rom formula~(3.2) of~\cite{FQT} we know that
\begin{equation}\label{dd51:2}
|D^2\Gamma(x)|\le C |x|^{-(n+2s+2)} , \quad |x|\geq 1.
\end{equation}
Hence the proof of Lemma~\ref{lem:decay DZ}
follows as the one of Lemma~\ref{:decay Z}
by using~\eqref{dd51:2} instead of~\eqref{dd51}.
\end{proof}

\begin{lem}\label{lem:decay D2 Z}
For any~$k\in\N$
there exists a positive constant~$C_k$ such that, for any~$i=1,\ldots,n$,
$$ |D^k Z_i|\le C_k |x-\xi|^{-n}, {\mbox{ for any }}|x-\xi|\ge 1.$$
\end{lem}

\begin{proof} {F}rom Lemma~C.1(ii) of~\cite{FLS}, we have that
\begin{equation}
\label{dd51:3}
|D^{k+1}\Gamma(x)|\le C_k |x|^{-n} , \quad |x|\geq 1.
\end{equation}
The proof of Lemma~\ref{lem:decay D2 Z}
follows as the one of Lemma~\ref{:decay Z}
by using~\eqref{dd51:3} instead of~\eqref{dd51}.
\end{proof}

We notice that
\begin{equation}\label{uguali}
\int_{\R^n}Z_i^2\,dx=\int_{\R^n}Z_j^2\, dx {\mbox{ for any }}i,j=1,\ldots,n.
\end{equation}
We set
\begin{equation}\label{alfaj}
\alpha:=\int_{\R^n}Z_1^2\,dx,
\end{equation}
and so, thanks to~\eqref{uguali}, we observe that
\begin{equation}\label{ug alfa}
\int_{\R^n}Z_i^2\,dx=\alpha {\mbox{ for any }}i=1,\ldots,n.
\end{equation}

\begin{lem}\label{lem:orto}
The~$Z_i$'s satisfy the following condition
\begin{equation}\label{J1}
\int_{\R^n}Z_i\,Z_j\,dx=\alpha\,\delta_{ij}.\end{equation}
Also, if~$\tau_o\in L^\infty([0,+\infty))$, $\tau(x):=
\tau_o(|x-\xi|)$ for any~$x\in\R^n$ and~$\tilde Z_i:=\tau Z_i$, then
\begin{equation}\label{J2}
\int_{\R^n}\tilde Z_i\,Z_j\,dx=\tilde\alpha\,\delta_{ij},\end{equation}
where\footnote{In particular,
we note that, if~$\tau_o$ has a sign and does not vanish identically
then~$\tilde\alpha\ne0$ (and we will often implicitly
assume that this is so in the sequel).}
\begin{equation}\label{alpha:J3}
\tilde\alpha:=\int_{\R^n}\tilde Z_1\, Z_1\,dx
.\end{equation}
\end{lem}

\begin{proof}
We first observe that the function~$w$ is radial (see, for instance,~\cite{FQT}) and therefore, recalling the definition of~$Z_i$ in~\eqref{Zj}, we have that
$$ Z_i =\frac{\partial w}{\partial x_i}(x-\xi)= w'_\xi(|x-\xi|)\frac{x_i-\xi_i}{|x-\xi|}.$$
Hence, using the change of variable~$y=x-\xi$,
for any~$i,j=1,\ldots,n$, we have
\begin{equation}\label{J5}\begin{split}
\int_{\R^n}\tilde Z_i\,Z_j\,dx \,&=  \int_{\R^n}\tau_o(|x-\xi|)\,
|w'(|x-\xi|)|^2\frac{(x_i-\xi_i)(x_j-\xi_j)}{|x-\xi|^2}\,dx\\
&= \int_{\R^n}
\tau_o(|y|)\,
|w'(|y|)|^2\frac{y_i\,y_j}{|y|^2}\,dy.\end{split}\end{equation}
Therefore, if~$i\ne j$,
$$ \int_{\R^n}\tilde Z_i\,Z_j\,dx
= \int_{\R^{n-1}}
y_j\left(\int_{\R}\tau_o(|y|)\,
|w'(|y|)|^2\frac{y_i}{|y|^2}\,dy_i\right)dy'=0,$$
since the function~$\tau_o(|y|)\,
|w'(|y|)|^2\frac{y_i}{|y|^2}$ is odd.
This proves~\eqref{J2} when~$i\ne j$. On the other hand, if~$i=j$,
formula~\eqref{J5} becomes
\begin{equation*}
\int_{\R^n}\tilde Z_i\,Z_i\,dx = \int_{\R^n}
\tau_o(|y|)\,
|w'(|y|)|^2\frac{y_i^2}{|y|^2}\,dy.\end{equation*}
We observe that the latter integral is invariant under rotation,
hence
$$ \int_{\R^n}\tilde Z_i\,Z_i\,dx = \int_{\R^n}
\tau_o(|y|)\,
|w'(|y|)|^2\frac{y_1^2}{|y|^2}\,dy=\tilde\alpha.$$
This establishes~\eqref{J2} also when~$i=j$.
Then, \eqref{J1} follows from~\eqref{J2} by choosing~$\tau_o:=1$
and comparing~\eqref{ug alfa} and~\eqref{alpha:J3}.
\end{proof}

\begin{cor}\label{cor:orto}
The~$Z_i$'s satisfy the following condition
$$
\int_{\Omega_\epsilon}Z_i\,Z_j\,dx = \alpha\delta_{ij}+O(\epsilon^{\nu}),
$$ with~$\nu>n+4s$.
\end{cor}

\begin{proof}
From Lemma~\ref{lem:orto},we have that
\begin{eqnarray*}
\int_{\Omega_\epsilon}Z_i\,Z_j\,dx &=& \int_{\R^n}Z_i\,Z_j\,dx-\int_{\R^n\setminus\Omega_\epsilon}Z_i\,Z_j\,dx\\
&=& \alpha\delta_{ij}-\int_{\R^n\setminus\Omega_\epsilon}Z_i\,Z_j\,dx.
\end{eqnarray*}
Moreover, from Lemma~\ref{:decay Z}, we obtain that
$$ \int_{\R^n\setminus\Omega_\epsilon}Z_i\,Z_j\,dx\le C\int_{\R^n\setminus\Omega_\epsilon}\frac{1}{|x-\xi|^{2\nu_1}}\,dx \le C\,\epsilon^{2\nu_1-n},$$
which implies the desired result.
\end{proof}

\section{Some regularity estimates}\label{S:a2}

Here we perform some uniform estimates on the solutions
of our differential equations.
For this, we introduce some notation:
given~$\xi\in\Omega_\epsilon$ with
\begin{equation}\label{dist bordo}
\dist(\xi,\partial\Omega_\epsilon)\ge\frac{c}{\epsilon}, {\mbox{ for some }}c\in(0,1),
\end{equation}
and~$\frac{n}{2}<\mu <n+2s$,
we define, for any~$x\in\R^n$,
\begin{equation}\label{rho}
\rho_\xi(x):=\frac{1}{(1+|x-\xi|)^\mu}.
\end{equation}
Moreover, we set
$$ \|\psi\|_{\star,\xi}:=\|\rho_\xi^{-1}\psi\|_{L^{\infty}(\R^n)}. $$

\begin{lem}\label{lem:100}
Let~$g\in L^2(\R^n)\cap L^\infty(\R^n)$ and let~$\psi\in H^s(\R^n)$ be a
solution to the problem
\begin{equation}\label{eq:100}
\left\{
\begin{matrix}
(-\Delta)^s \psi +\psi +g =0 & {\mbox{ in $\Omega_\epsilon$,}}\\
\psi =0 & {\mbox{ in }}\R^n\setminus\Omega_\epsilon.
\end{matrix}
\right.
\end{equation}

Then, there exists a positive constant~$C$ such that
$$ \|\psi\|_{L^{\infty}(\R^n)}+\sup_{x\ne y}\frac{|\psi(x)-\psi(y)|}{|x-y|^s}\le C\left( \|g\|_{L^\infty(\R^n)}+\|g\|_{L^2(\R^n)}\right).$$
\end{lem}

\begin{proof}
From Theorem~$8.2$ in~\cite{DFV} we have that~$\psi\in L^{\infty}(\R^n)$ and
there exists a constant~$C>0$ such that
\begin{equation}\label{norma 1}
\|\psi\|_{L^{\infty}(\R^n)}\le C\,\left(\|g\|_{L^{\infty}(\R^n)}+\|\psi\|_{L^2(\R^n)}\right).
\end{equation}
Now, we show that
\begin{equation}\label{norma 2}
\|\psi\|_{L^2(\R^n)}\le \|g\|_{L^2(\R^n)}.
\end{equation}
Indeed, we multiply the equation in~\eqref{eq:100} by~$\psi$ and we integrate
over~$\Omega_\epsilon$, obtaining that
\begin{equation}\label{mkiopipj}
\int_{\Omega_\epsilon}(-\Delta)^s\psi\,\psi+\psi^2+g\,\psi\,dx=0.
\end{equation}
We notice that, thanks to formula~$(1.5)$ in~\cite{SerraRos-Dirichlet-regularity},
$$ \int_{\Omega_\epsilon}(-\Delta)^s\psi\,\psi\,dx= \int_{\Omega_\eps}\left|(-\Delta)^s\psi\right|^2\,dx\ge 0.$$
Hence, from~\eqref{mkiopipj} we have
$$ \int_{\Omega_\epsilon}\psi^2\,dx\le \int_{\Omega_\epsilon}-g\,\psi\,dx.$$
So, using H\"older inequality, we get
$$ \int_{\Omega_\epsilon}\psi^2\,dx\le \left(\int_{\Omega_\epsilon}g^2\,dx\right)^{1/2}\left(\int_{\Omega_\epsilon}\psi^2\,dx\right)^{1/2}, $$
and therefore, dividing by~$\left(\int_{\Omega_\epsilon}\psi^2\,dx\right)^{1/2}$,
we obtain~\eqref{norma 2}.

From~\eqref{norma 1} and~\eqref{norma 2}, we have that
$$\|\psi\|_{L^{\infty}(\R^n)}\le C\left(\|g\|_{L^{\infty}(\R^n)}+\|g\|_{L^2(\R^n)}\right).$$

Now, since both~$\psi$ and~$g$ are bounded, from the regularity results
in~\cite{Silvestre} we have that~$\psi$ is~$C^{\alpha}$ in the interior
of~$\Omega_\epsilon$, for some~$\alpha\in(0,2s)$.

It remains to prove that~$\psi$ is~$C^\alpha$ near the boundary of~$\Omega_\epsilon$. For this, we fix a point~$p\in\partial\Omega_\epsilon$
and we look at the equation in the ball~$B_1(p)$.

We notice that~$|(-\Delta)^s \psi|$ is bounded,
since both~$\psi$ and~$g$ are in~$L^{\infty}(\R^n)$,
and therefore we can apply Proposition~$3.5$ in~\cite{SRO},
obtaining that, for any~$x,y\in B_1(p)\cap\Omega_\epsilon$,
\begin{equation}\label{C alfa}
\frac{\psi(x)}{d^s(x)} - \frac{\psi(y)}{d^s(y)} \le C_1\left(\|\psi\|_{L^{\infty}(\R^n)}+\|g\|_{L^\infty(\R^n)}\right),
\end{equation}
where~$d(x):=\dist(x,\partial\Omega_\epsilon)$.
In particular, we can fix~$y\in B_1(p)\cap\Omega_\epsilon$,
such that~$d(y)=1/2$. Since~$\psi$ is bounded, from~\eqref{C alfa} we have that
$$ \frac{\psi(x)}{d^s(x)}\le C_2\left(\|\psi\|_{L^{\infty}(\R^n)}+\|g\|_{L^\infty(\R^n)}\right), $$
which gives that
$$ \psi(x)\le C_2\left(\|\psi\|_{L^{\infty}(\R^n)}+\|g\|_{L^\infty(\R^n)}\right) d^s(x).$$
This implies that~$\psi$ is~$C^s$ also near the boundary
and concludes the proof of the lemma.
\end{proof}

\begin{lem}\label{LL}
Let~$\xi\in\Omega_\eps$, ${\mathcal{B}}$ be a bounded subset of~$\R^n$,
and~$R_0>0$ be such that
\begin{equation}\label{R0}
B_{R_0}(\xi)\supseteq {\mathcal{B}}.\end{equation}
Let~${\mathcal{W}}\in L^\infty(\R^n)$
be such that
\begin{equation}\label{m}
m:=\inf_{\R^n\setminus{\mathcal{B}}} {\mathcal{W}} >0.
\end{equation}
Let also~$g\in L^2(\R^n)$, with~$\|g\|_{\star,\xi}<+\infty$,
and let~$\psi\in H^s(\R^n)$ be a solution to
\begin{equation*}
\left\{
\begin{matrix}
(-\Delta)^s \psi +{\mathcal{W}}\psi +g =0 & {\mbox{ in $\Omega_\epsilon$,}}\\
\psi =0 & {\mbox{ in }}\R^n\setminus\Omega_\epsilon.
\end{matrix}
\right.
\end{equation*}
Then, there exists a positive constant~$C$, possibly
depending on~$m$, $R_0$ and~$\|{\mathcal{W}}\|_{L^\infty(\R^n)}$
(and also on~$n$, $s$,
and $\Omega$),
such that\footnote{In~\eqref{x psi}
we use \label{v foo}
the standard convention that~$\|
\psi\|_{L^\infty({\mathcal{B}})}:=0$ when~${\mathcal{B}}:=\varnothing$
(equivalently, if~${\mathcal{B}}=\varnothing$,
the term~$\|
\psi\|_{L^\infty({\mathcal{B}})}$ can be neglected in the proof
of Lemma~\ref{LL}, since, in this case,~$G$ and~$g$
are the same from~\eqref{var} on).}
\begin{equation}\label{x psi}
\|\psi\|_{\star,\xi} \le C\left(\|\psi\|_{L^\infty({\mathcal{B}})}
+ \|g\|_{\star,\xi}\right).\end{equation}
\end{lem}

\begin{proof}
We define
\begin{equation}\label{var}\begin{split}
& W:=m\chi_{ {\mathcal{B}} }+
{\mathcal{W}}\,\chi_{\R^n\setminus {\mathcal{B}} }\\
{\mbox{and }} \ &
G:= (m-{\mathcal{W}})\,\chi_{ {\mathcal{B}} }\,\psi-g.\end{split}\end{equation}
We observe that
\begin{equation}\label{dop}\begin{split}
\| G\|_{\star,\xi} \,&\le
\sup_{x\in {\mathcal{B}} }
(1+|x-\xi|)^\mu\, (m+{\mathcal{W}}(x))
\,\psi(x)
+\| g\|_{\star,\xi}\\
&\le 2\,(1+R_0)^\mu \,\|{\mathcal{W}}\|_{L^\infty(\R^n)}\,
\|\psi\|_{L^\infty({\mathcal{B}})}+\| g\|_{\star,\xi}
\\ &\le C_{0}\, \|\psi\|_{L^\infty({\mathcal{B}})}+\|g\|_{\star,\xi},\end{split}
\end{equation}
for a suitable~$C_{0}>0$ possibly depending
on~$R_0$ and~$\|{\mathcal{W}}\|_{L^\infty(\R^n)}$
(notice that~\eqref{R0} was used here).
Also~$\psi$ is a solution of
\begin{equation}\label{t}\begin{split}
& (-\Delta)^s \psi +W\psi =
(W-{\mathcal{W}})\psi-g\\
&\qquad =(W-{\mathcal{W}}\chi_{\R^n\setminus{\mathcal{B}}}
-{\mathcal{W}}\chi_{{\mathcal{B}}})\psi-g\\
&\qquad =(m\chi_{{\mathcal{B}}}
-{\mathcal{W}}\chi_{{\mathcal{B}}})\psi-g
\\ &\qquad= G\end{split}\end{equation}
and, in virtue of~\eqref{m},
\begin{equation}\label{d2}
W\ge
m\chi_{ {\mathcal{B}} }+
m\chi_{\R^n\setminus {\mathcal{B}} }=m.
\end{equation}
We take~$\rho_0:=(1+|x|)^{-\mu}$ and
$\eta\in H^s(\R^n)$ to be a solution of
\begin{equation}\label{e}
(-\Delta)^s \eta+m \eta = \rho_0.\end{equation}
We refer to formula~(2.4) in~\cite{DDW} for the existence
of such solution and to Lemma~2.2 there for the following
estimate:
\begin{equation}\label{L22}
\sup_{x\in\R^n} (1+|x|)^{\mu} \eta(x) \le
C_1 \sup_{x\in\R^n} (1+|x|)^{\mu}\rho_0(x)=C_1,
\end{equation}
for some~$C_1>0$, possibly depending on~$m$. Also, by Lemma~2.4 in~\cite{DDW},
we have that~$\eta\ge0$, and so, recalling~\eqref{d2},
we obtain that
\begin{equation}\label{p}
\left(W(x)-m\right)\,\eta(x-\xi)\ge0.
\end{equation}
Now we define~$\eta_\xi(x):=\eta(x-\xi)$,
\begin{equation}\label{L23}
C_\star := \|G\|_{\star,\xi}  \end{equation}
and~$\omega:=C_\star\eta_\xi\pm \psi$.
We remark that the quantity~$C_\star$ plays a different role
from the other constants~$C_{0}$, $C_1$ and~$C_2$:
indeed, while~$C_{0}$, $C_1$ and~$C_2$
depend only
on~$m$, $R_0$ and~$\|{\mathcal{W}}\|_{L^\infty(\R^n)}$
(as well on~$n$, $s$
and $\Omega$),
the quantity~$C_\star$ also depend on~$G$, and this
will be made explicit at the end of the proof.

Notice also that~$\rho_0(x-\xi)=\rho_\xi(x)$, due to the definition
in~\eqref{rho}, and
\begin{equation}\label{PP} \begin{split}
C_\star\rho_\xi(x)\pm G(x)\,&\ge \rho_\xi(x) \Big( C_\star-\rho_\xi^{-1}(x)|G(x)|\Big)\\
&\ge\rho_\xi(x)
\Big( C_\star -\|G\|_{\star,\xi}\Big)\\ &\ge0.\end{split}
\end{equation}
Thus we infer that
\begin{equation}\label{PL22}\begin{split}
(-\Delta)^s \omega+W\omega \,&=
C_\star\Big((-\Delta)^s \eta_\xi+W\eta_\xi\Big)\pm \Big((-\Delta)^s \psi+W\psi
\Big)\\ &=C_\star\rho_\xi+
C_\star\,\left(W-m\right)\,\eta_\xi \pm G\\ &\ge 0,
\end{split}
\end{equation}
in~$\Omega_\eps$,
thanks to~\eqref{t}, \eqref{e}, \eqref{p} and~\eqref{PP}.
Furthermore, in~$\R^n\setminus\Omega_\eps$ we have
that~$\omega=C_\star\eta_\xi\ge0$. As a consequence of this, \eqref{PL22}
and the maximum principle (see e.g. Lemma~6
in~\cite{SerVal}), we conclude that~$\omega\ge0$
in the whole of~$\R^n$.

Accordingly, for any~$x\in\R^n$
\begin{eqnarray*}
\mp\rho_\xi^{-1}(x)\psi(x) &=& \rho_\xi^{-1}(x)\Big(C_\star\eta_\xi(x)-\omega(x)\Big)
\\ &\le& C_\star\rho_\xi^{-1}(x)\eta_\xi(x) \\
&\le& C_\star \sup_{y\in\R^n} \rho_\xi^{-1}(y)\eta_\xi(y) \\
&=& C_\star \sup_{y\in\R^n} \rho_\xi^{-1}(y+\xi)\eta_\xi(y+\xi)\\
&=& C_\star \sup_{y\in\R^n} (1+|y|)^{\mu} \eta(y) \\
&\le& C_1\,C_\star,\end{eqnarray*}
where~\eqref{L22} was used in the last step.
Hence, recalling~\eqref{L23} and~\eqref{dop},
$$ |\rho_\xi^{-1}(x)\psi(x)|\le C_1\,\|G\|_{\star,\xi}\le
C_1 \Big( C_{0} \,\|\psi\|_{L^\infty({\mathcal{B}})} +\|g\|_{\star,\xi} \Big),$$
which implies~\eqref{x psi}.
\end{proof}

As a consequence of Lemma~\ref{LL}, we obtain the following two corollaries:

\begin{cor}\label{lem:200}
Let~$g\in L^2(\R^n)$, with~$\|g\|_{\star,\xi}<+\infty$,
and let~$\psi\in H^s(\R^n)$ be a solution to
\begin{equation*}
\left\{
\begin{matrix}
(-\Delta)^s \psi +\psi -pw_\xi^{p-1}\psi +g =0 & {\mbox{ in $\Omega_\epsilon$,}}\\
\psi =0 & {\mbox{ in }}\R^n\setminus\Omega_\epsilon.
\end{matrix}
\right.
\end{equation*}
Then, there exist positive constants~$C$ and~$R$
such that
\begin{equation}\label{x psi-2}
\|\psi\|_{\star,\xi} \le C\left(\|\psi\|_{L^\infty(B_R(\xi))}
+ \|g\|_{\star,\xi}\right).\end{equation}
\end{cor}

\begin{proof} We apply Lemma~\ref{LL} with~${\mathcal{W}}:=
1-pw_\xi^{p-1}$ and~${\mathcal{B}}:=B_R(\xi)$ (notice that,
with this notation~\eqref{x psi-2} would follow from~\eqref{x psi}).
So, we only need to check that~\eqref{m} holds true with
a suitable choice of~$R$. For this,
we use that~$w$ decays at infinity
(recall~\eqref{decay w}), hence
we can fix~$R$ large enough such that
$$ pw^{p-1}(x)\le\frac12 \ {\mbox{for every }} x\in \R^n\setminus B_R.$$
accordingly~${\mathcal{W}}\ge 1-(1/2)=1/2$, which
establishes~\eqref{m} with~$m:=1/2$.
\end{proof}

\begin{cor}\label{cor:300}
Let~$g\in L^2(\R^n)$, with~$\|g\|_{\star,\xi}<+\infty$,
and let~$\psi\in H^s(\R^n)$ be a solution to
\begin{equation*}
\left\{
\begin{matrix}
(-\Delta)^s \psi +\psi +g =0 & {\mbox{ in $\Omega_\epsilon$,}}\\
\psi =0 & {\mbox{ in }}\R^n\setminus\Omega_\epsilon.
\end{matrix}
\right.
\end{equation*}
Then, there exists a positive constant~$C$ such that
$$ \|\psi\|_{\star,\xi} \le C \|g\|_{\star,\xi}.$$
\end{cor}

\begin{proof} We use for this Lemma~\ref{LL}
with~${\mathcal{W}}:=1$ and~${\mathcal{B}}:=\varnothing$
(recall footnote~\ref{v foo}).
\end{proof}

\section{The Lyapunov-Schmidt reduction}\label{S:a3}

In this section we deal with
the linear theory associated to the scaled problem~\eqref{EQ eps}.
For this, we introduce the functional space
$$ \Psi:=\left\{\psi\in H^{s}(\R^n) {\mbox{ s.t. }}\psi=0 {\mbox{ in }}\R^n\setminus\Omega_\epsilon {\mbox{ and }}\int_{\Omega_\epsilon}\psi\,Z_i\,dx=0 {\mbox{ for any }}i=1,\ldots,n\right\},$$
where the~$Z_i$'s are introduced in~\eqref{Zj}. We remark that the condition
$$ \int_{\Omega_\epsilon}\psi\,Z_i\,dx=0 {\mbox{ for any }}i=1,\ldots,n$$
means that~$\psi$ is orthogonal to the space~$\mathcal Z$
(that is the space spanned by~$Z_i$)
with respect to the scalar product in~$L^2(\Omega_\eps)$.

We look for solution to~\eqref{EQ eps} of the form
\begin{equation}\label{ruybc}
u=u_\xi:=\bar u_\xi+\psi,
\end{equation}
where~$\bar u_\xi$ is the solution to~\eqref{EQ bar u} and~$\psi$ is a small function (for~$\epsilon$ sufficiently small)
which belongs to~$\Psi$.

Inserting~$u$ (given in~\eqref{ruybc}) into~\eqref{EQ eps} and recalling
that~$\bar u_\xi$ is a solution to~\eqref{EQ bar u}, we have that,
in order to obtain a solution to~\eqref{EQ eps},~$\psi$ must satisfy
\begin{equation}\label{EQ psi}
(-\Delta)^s\psi+\psi-pw_\xi^{p-1}\psi = E(\psi)+ N(\psi) {\mbox{ in }}\Omega_\epsilon,
\end{equation}
where\footnote{As a matter of fact, one should write the positive
parts in~\eqref{N psi}, namely set~$
E(\psi):= (\bar u_\xi+\psi)_+^p-(w_\xi+\psi)_+^p$ and~$N(\psi):=
(w_\xi+\psi)_+^p-w_\xi^p-pw_\xi^{p-1}\psi$, but, a posteriori, this
is the same by maximum principle. So we preferred, with a slight
abuse of notation, to drop the positive parts for simplicity of notation.}
\begin{equation}\begin{split}\label{N psi}
& E(\psi):= (\bar u_\xi+\psi)^p-(w_\xi+\psi)^p\\
{\mbox{ and }} & N(\psi):=(w_\xi+\psi)^p-w_\xi^p-pw_\xi^{p-1}\psi.
\end{split}\end{equation}

Instead of solving~\eqref{EQ psi}, we will consider a projected version of the problem. Namely we will look for a
solution~$\psi\in H^{s}(\R^n)$ of the equation
\begin{equation}\label{EQ proj}
(-\Delta)^s\psi+\psi-pw_\xi^{p-1}\psi =E(\psi)+ N(\psi)+\sum_{i=1}^n c_i\,Z_i {\mbox{ in $\Omega_\epsilon$,}}
\end{equation}
for some coefficients~$c_i\in\R$,~$i\in\{1,\ldots,n\}$.
Moreover, we require that~$\psi$ satisfies the conditions
\begin{equation}\label{zero}
\psi =0 {\mbox{ in }}\R^n\setminus\Omega_\epsilon,
\end{equation}
and
\begin{equation}\label{orto}
\int_{\Omega_\epsilon}\psi\,Z_i\,dx=0  {\mbox{ for any }}i=1,\ldots,n.
\end{equation}
We will prove that problem~\eqref{EQ proj}-\eqref{orto} admits
a unique solution, which is small if~$\epsilon$ is sufficiently small,
and then we will show that the coefficients~$c_i$ are equal to zero for every~$i\in\{1,\ldots,n\}$ for a suitable~$\xi$.
This will give us a solution~$\psi\in\Psi$ to~\eqref{EQ psi}, and therefore
a solution~$u$ of~\eqref{EQ eps}, thanks to the definition in~\eqref{ruybc}.

\subsection{Linear theory}
In this subsection we develop a general theory that will give us
the existence result for the linear problem~\eqref{EQ proj}-\eqref{orto}.

\begin{thm}\label{linear}
Let~$g\in L^2(\R^n)$ with~$\|g\|_{\star,\xi}<+\infty$. If~$\epsilon>0$ is sufficiently small,
there exist a unique~$\psi\in\Psi$ and numbers~$c_i\in\R$, for any~$i\in\{1,\ldots,n\}$, such that
\begin{equation}\label{EQ g}
(-\Delta)^s \psi +\psi -pw_\xi^{p-1}\psi +g =\sum_{i=1}^n c_i\,Z_i  {\mbox{ in $\Omega_\epsilon$.}}
\end{equation}
Moreover, there exists a constant~$C>0$ such that
\begin{equation}\label{est fin}
\|\psi\|_{\star,\xi}\le C\|g\|_{\star,\xi}.
\end{equation}
\end{thm}

Before proving Theorem~\ref{linear} we need some preliminary lemmata.
In the next lemma we show that we can uniquely determine the
coefficients~$c_i$ in~\eqref{EQ g} in terms of~$\psi$ and~$g$.
Actually, we will show that the estimate on the ~$c_i$'s holds in a more general
case, that is we do not need the orthogonality condition in~\eqref{orto}.

\begin{lem}\label{lem:ci}
Let~$g\in L^2(\R^n)$ with~$\|g\|_{\star,\xi}<+\infty$.
Suppose that~$\psi\in H^s(\R^n)$ satisfies
\begin{equation}\label{9.7bis}
\left\{
\begin{matrix}
(-\Delta)^s \psi +\psi -pw_\xi^{p-1}\psi +g =\displaystyle\sum_{i=1}^n c_i\,Z_i   & {\mbox{ in $\Omega_\epsilon$,}}\\
\psi =0 & {\mbox{ in }}\R^n\setminus\Omega_\epsilon,
\end{matrix}
\right.
\end{equation}
for some~$c_i\in\R$,~$i=1,\ldots,n$.

Then, for~$\eps>0$ sufficiently small and
for any~$i\in\{1,\ldots,n\}$, the coefficient~$c_i$ is given by
\begin{equation}\label{ci}
c_i=\frac{1}{\alpha}\int_{\R^n}g\,Z_i\,dx + f_i,
\end{equation}
where~$\alpha$ is defined in~\eqref{alfaj},
for suitable~$f_i\in\R$ that satisfies
\begin{equation}\label{ci2}
|f_i|\le C\,\epsilon^{n/2}\left(\|\psi\|_{L^{2}(\R^n)}+\|g\|_{L^2(\R^n)}\right), \end{equation}
for some positive constant~$C$.
\end{lem}

\begin{proof}
We start with some considerations in Fourier
space on a function~$T\in C^2(\R^n)\cap H^2(\R^n)$.
First of all, for any~$j\in\{1,\dots,n\}$
$$ \| \partial^2_j T\|^2_{L^2(\R^n)}=
\| {\mathcal{F}} ( \partial^2_j T)\|^2_{L^2(\R^n)}
= \| \xi_j^2 \hat T\|_{L^2(\R^n)}^2 =
\int_{\R^n} \xi_j^4 \,|\hat T(\xi)|^2\,d\xi.$$
Moreover, by convexity,
$$ |\xi|^4 =\left( \sum_{j=1}^n \xi_j^2\right)^2 \le 2
\sum_{j=1}^n \xi_j^4$$
and therefore
$$ 2\| D^2 T\|_{L^2(\R^n)}^2=
\sum_{j=1}^n 2\| \partial^2_j T\|^2_{L^2(\R^n)}
\ge \int_{\R^n} |\xi|^4 \,|\hat T(\xi)|^2\,d\xi.$$
As a consequence
\begin{equation}\label{7.10}
\begin{split}
& \| (-\Delta)^s T\|_{L^2(\R^n)}^2 =\| {\mathcal{F}}
((-\Delta)^s T)\|_{L^2(\R^n)}^2
=\| |\xi|^{2s}\hat T\|_{L^2(\R^n)}^2\\
&\qquad=\int_{\R^n} |\xi|^{4s} \,|\hat T(\xi)|^2\,d\xi
\le \int_{\R^n} (1+|\xi|^{4}) \,|\hat T(\xi)|^2\,d\xi
\\ &\qquad\le \|T\|_{L^2(\R^n)}^2+2\| D^2 T\|_{L^2(\R^n)}^2
\le C \| T\|_{H^2(\R^n)},
\end{split}
\end{equation}
for some~$C>0$.

Now, without loss of generality, we may suppose
that
\begin{equation}\label{Sub}
B_{c/\eps}(\xi)\subseteq\Omega_\eps,\end{equation} for some~$c>0$.
Fix~$\eps>0$, and let~$\tau_\eps\in C^\infty(\R^n,[0,1])$,
with~$\tau_\eps=1$ in~$B_{(c/\eps)-1}(\xi)$,
$\tau_\eps=0$ outside~$B_{c/\eps}(\xi)$
and~$|\nabla \tau_\eps|\le C$. We set~$T_{\eps,j}:= Z_j \tau_\eps$.
Hence, from~\eqref{7.10} and Lemmata~\ref{lem:decay DZ} and~\ref{lem:decay D2 Z},
\begin{equation}\label{7.11}
\| (-\Delta)^s T_{\eps,j}\|_{L^2(\R^n)}^2\le C,\end{equation}
for some~$C>0$, independent of~$\eps$ and~$j$.

Moreover, the function~$T_{\eps,j}$
belongs to~$H^s(\R^n)$ and vanishes outside~$B_{c/\eps}(\xi)$,
and so in particular outside~$\Omega_\eps$, thanks to~\eqref{Sub}.

Thus (see e.g. formula~(1.5) in \cite{SerraRos-Dirichlet-regularity})
\begin{equation}\label{joxa}
\int_{\Omega_\eps} (-\Delta)^s \psi \, T_{\eps,j}\,dx
= \int_{\Omega_\eps} (-\Delta)^{s/2}\psi \, (-\Delta)^{s/2} T_{\eps,j}\,dx
= \int_{\Omega_\eps} \psi \, (-\Delta)^s T_{\eps,j}\,dx.
\end{equation}
As a consequence, recalling~\eqref{7.11}
\begin{equation}\label{prima est}
\left|\int_{\Omega_\eps} (-\Delta)^s \psi \, T_{\eps,j}\,dx\right|
\le \|\psi\|_{L^2(\Omega_\eps)} \,\|(-\Delta)^s T_{\eps,j}\|_{L^2(\Omega_\eps)}
\le C\,\|\psi\|_{L^2(\R^n)}.
\end{equation}

Now, we fix~$j\in\{1,\ldots,n\}$, we multiply the equation in~\eqref{9.7bis}
by~$T_{\eps,j}$ and we integrate over~$\Omega_\epsilon$.
We obtain
\begin{equation}\label{ZZ}
\sum_{i=1}^n c_i \int_{\Omega_\epsilon}Z_i\,T_{\eps,j}\,dx=
\int_{\Omega_\epsilon}T_{\eps,j}\left((-\Delta)^s\psi + \psi -pw_\xi^{p-1}\psi +g\right)dx.
\end{equation}

Now, we observe that, thanks to~\eqref{joxa}, we can write
\begin{equation}\label{hence}
\int_{\Omega_\eps}(-\Delta)^s\psi\,T_{\eps,j}\,dx
=\int_{\Omega_\eps}\psi(-\Delta)^sT_{\eps,j}\,dx=
\int_{\Omega_\epsilon} \psi(-\Delta)^s\left(T_{\eps,j}-Z_j\right)\, dx +\int_{\Omega_\eps}\psi(-\Delta)^sZ_{j}\,dx.
\end{equation}
Using H\"older inequality and~\eqref{7.10}, we have that
\begin{equation}\begin{split}\label{scarica}
\left|\int_{\Omega_\eps}\psi(-\Delta)^s\left(T_{\eps,j}-Z_j\right)\,dx\right| &\le \, \|\psi\|_{L^2(\R^n)}\|(-\Delta)^s\left(T_{\eps,j}-Z_j\right)\|_{L^2(\R^n)}\\
&\le \, C\|\psi\|_{L^2(\R^n)}\|T_{\eps,j}-Z_j\|_{H^2(\R^n)}.
\end{split}\end{equation}

Let us estimate the~$H^2$-norm of~$T_{\eps,j}-Z_j$.
First we have that
$$ \| T_{\eps,j}-Z_j \|_{L^2(\R^n)}^2 = \int_{\R^n}(\tau_\epsilon -1)^2\,Z_j^2 \, dx \le \int_{B^c_{(c/\eps)-1}(\xi)} Z_j^2\,dx,$$
since~$\tau_\eps=1$ in~$B_{(c/\eps)-1}(\xi)$ and takes values in~$(0,1)$. Hence, from Lemma~\ref{:decay Z} we deduce that
$$ \| T_{\eps,j}-Z_j \|_{L^2(\R^n)}^2\le C\int_{B_{(c/\eps)-1}^c(\xi)}
\frac{1}{|x-\xi|^{2\nu_1}}\,dx\le C\,\epsilon^{n},
$$
up to renaming~$C$. Therefore,
\begin{equation}\label{L2 norm}
\| T_{\eps,j}-Z_j \|_{L^2(\R^n)}\le C\,\epsilon^{n/2}.
\end{equation}

Moreover, we have that
\begin{eqnarray*}
\| \nabla\left(T_{\eps,j}-Z_j\right)\|_{L^2(\R^n)}^2 &=& \int_{\R^n}|(\tau_\eps-1)\nabla Z_j+\nabla\tau_\eps Z_j|^2\,dx \\
&=& \int_{\R^n}(\tau_\eps-1)^2|\nabla Z_j|^2 +|\nabla\tau_\eps|^2Z_j^2+2(\tau_\eps-1)Z_j\nabla Z_j\cdot\nabla\tau_\eps\,dx.
\end{eqnarray*}
Using the fact that both~$\tau_\eps-1$ and~$\nabla\tau_\eps$ have support
outside~$B_{(c/\eps)-1}(\xi)$ and Lemmata~\ref{:decay Z} and~\ref{lem:decay DZ},
we obtain that
\begin{equation}\label{L2 norm D}
\| \nabla\left(T_{\eps,j}-Z_j\right)\|_{L^2(\R^n)}\le C\,\epsilon^{n/2}.
\end{equation}

Finally, using again the fact that~$\tau_\eps-1$,~$\nabla\tau_\eps$ and~$D^2\tau_\eps$ have support outside~$B_{(c/\eps)-1}(\xi)$ and
Lemmata~\ref{:decay Z},~\ref{lem:decay DZ} and~\ref{lem:decay D2 Z}, we obtain that
$$ \| D^2\left(T_{\eps,j}-Z_j\right)\|_{L^2(\R^n)}\le C\,\epsilon^{n/2}.$$
Using this,~\eqref{L2 norm} and~\eqref{L2 norm D} we have that
$$ \|T_{\eps,j}-Z_j\|_{H^2(\R^n)}\le C\,\epsilon^{n/2},$$
and so from~\eqref{scarica} we obtain
\begin{equation}\label{2.29bis}
\left|\int_{\Omega_\eps}\psi(-\Delta)^s\left(T_{\eps,j}-Z_j\right)\,dx\right|\le C\,\epsilon^{n/2}\|\psi\|_{L^2(\R^n)}.\end{equation}

Now, using~\eqref{hence}, we have that
\begin{eqnarray*}
&&\int_{\Omega_\epsilon}T_{\eps,j}\left((-\Delta)^s\psi +\psi-pw_\xi^{p-1}\psi\right)dx \\
&&\qquad =  \int_{\Omega_\eps}\psi(-\Delta)^sZ_{j}+ T_{\eps,j}\,\psi-pw_\xi^{p-1}\psi\,T_{\eps,j}\,dx +\int_{\Omega_\eps}\psi(-\Delta)^s(T_{\eps,j}-Z_j)\,dx.
\end{eqnarray*}
Since~$w_\xi$ is a solution to~\eqref{EQ w}, we have that~$Z_j$ solves
$$ (-\Delta)^s Z_j+Z_j= pw_\xi^{p-1}Z_j,$$
and this implies that
\begin{eqnarray*}
&&\int_{\Omega_\epsilon}T_{\eps,j}\left((-\Delta)^s\psi +\psi-pw_\xi^{p-1}\psi\right)dx \\
&&\qquad = \int_{\Omega_\eps}\psi(T_{\eps,j}-Z_j)-pw_\xi^{p-1}\psi(T_{\eps,j}-Z_j)\,dx +\int_{\Omega_\eps}\psi(-\Delta)^s(T_{\eps,j}-Z_j)\,dx.
\end{eqnarray*}
Hence, using the fact that~$w_\xi$ is bounded (see~\eqref{decay w})
and H\"older inequality, we have that
\begin{equation}\begin{split}\label{first p}
&\left|\int_{\Omega_\epsilon}T_{\eps,j}\left((-\Delta)^s\psi +\psi-pw_\xi^{p-1}\psi\right)dx\right| \\
&\qquad \le C\Big( \|\psi\|_{L^2(\R^n)}\|T_{\eps,j}-Z_j\|_{L^2(\R^n)}+
\left|\int_{\Omega_\eps}\psi\,(-\Delta)^s(T_{\eps,j}-Z_j)\,dx\right|\Big)\\
&\qquad \le C\,\eps^{n/2}\|\psi\|_{L^2(\R^n)},
\end{split}\end{equation}
where we have used~\eqref{L2 norm} and~\eqref{2.29bis} in the last step.

Now, we can write
\begin{equation}\label{lakjsgfd}
\int_{\Omega_\eps}T_{\eps,j}\,g\,dx = \int_{\Omega_\eps}(T_{\eps,j}-Z_j)g\,dx + \int_{\Omega_\eps}Z_j\,g\,dx=\int_{\Omega_\eps}(T_{\eps,j}-Z_j)g\,dx + \int_{\R^n}Z_j\,g\,dx-\int_{\R^n\setminus\Omega_\eps}Z_j\,g\,dx.
\end{equation}
Using H\"older inequality and~\eqref{L2 norm}, we can estimate
$$
\left|\int_{\Omega_\eps}(T_{\eps,j}-Z_j)g\,dx\right|\le \|T_{\eps,j}-Z_j\|_{L^2(\R^n)}\|g\|_{L^2(\R^n)}\le C\eps^{n/2}\|g\|_{L^2(\R^n)}.$$
Moreover, from H\"older inequality and Lemma~\ref{:decay Z}
(and recalling that~$\dist(\xi,\partial\Omega_\eps)\ge c/\eps$), we obtain that
$$ \left|\int_{\R^n\setminus\Omega_\eps}Z_j\,g\,dx\right|\le
\|g\|_{L^2(\R^n)}\left(\int_{\R^n\setminus\Omega_\eps}\frac{C}{|x-\xi|^{
2\nu_1}}\,dx\right)^{1/2}\le C\,\eps^{n/2}\|g\|_{L^2(\R^n)}.$$
The last two estimates and~\eqref{lakjsgfd} imply that
$$ \int_{\Omega_\eps}T_{\eps,j}\,g\,dx =\int_{\R^n}Z_j\,g\,dx+\tilde{f}_j,$$
where
$$ |\tilde{f}_j|\le C\,\epsilon^{n/2}\|g\|_{L^2(\R^n)}.$$
{F}rom this,~\eqref{ZZ} and~\eqref{first p}, we have that
\begin{equation}\label{ZZZZZ}
\sum_{i=1}^n c_i \int_{\Omega_\epsilon}Z_i\,T_{\eps,j}\,dx=\int_{\R^n}g\,Z_j\,dx+\bar{f}_j,
\end{equation}
where
\begin{equation}\label{bar f}
|\bar{f}_j|\le C\,\epsilon^{n/2}\left(
\|\psi\|_{L^{2}(\R^n)}+ \|g\|_{L^2(\R^n)}\right)
\end{equation}
up to renaming the constants.

On the other hand, we can write
\begin{equation}\label{barbis}
\int_{\Omega_\eps}Z_i\,T_{\eps,j}\,dx=\int_{\Omega_\eps}Z_i\left(T_{\eps,j}-Z_j\right)\,dx+\int_{\Omega_\eps}Z_i\,Z_{j}\,dx.
\end{equation}
From H\"older inequality,~\eqref{L2 norm} and Lemma~\ref{:decay Z}, we have that
$$ \left|\int_{\Omega_\eps}Z_i\left(T_{\eps,j}-Z_j\right)\,dx\right|\le \left(\int_{\Omega_\eps}Z_i^2\,dx\right)^{1/2}\|T_{\eps,j}-Z_j\|_{L^2(\R^n)}\le C\,\epsilon^{n/2}.$$
Using this and Corollary~\ref{cor:orto} in~\eqref{barbis}, we obtain that
\begin{equation}\label{qwwqdeqweqweqw}
\int_{\Omega_\eps}Z_i\,T_{\eps,j}\,dx=\alpha\,\delta_{ij}+O(\epsilon^{n/2}).
\end{equation}
So, we consider the matrix~$A\in{\mbox{Mat}}(n\times n)$ defined as
\begin{equation}
A_{ji}:=\int_{\Omega_\eps}Z_i\,T_{\eps,j}\,dx.\end{equation}
Thanks to~\eqref{qwwqdeqweqweqw}, the matrix~$\alpha^{-1}A$ is a perturbation
of the identity and so it is invertible for~$\epsilon$ sufficiently small,
with inverse equal to the identity plus smaller order term
of size~$\eps^{n/2}$. Hence, the matrix~$A$
is invertible too, with inverse
\begin{equation}\label{a-1}
(A^{-1})_{ji}= \alpha^{-1}\delta_{ij}+O(\eps^{n/2}).\end{equation}
So we consider the vector~$d=(d_1,\dots,d_n)$ defined
by
\begin{equation}\label{d-1}
d_j:= \int_{\R^n}g\,Z_j\,dx+\bar{f}_j.\end{equation}
We observe that
$$ \left|\int_{\R^n}g\,Z_j\,dx\right| \le
\|g\|_{L^2(\R^n)}\|Z_j\|_{L^2(\R^n)}\le C\,\|g\|_{L^2(\R^n)},$$
thanks Lemma~\ref{:decay Z}.
As a consequence, recalling~\eqref{bar f}, we obtain that
\begin{equation}\label{d-2} |d|\le C\,\left(
\|\psi\|_{L^{2}(\R^n)}+ \|g\|_{L^2(\R^n)}\right),\end{equation}
up to renaming~$C$.

With the setting above,~\eqref{ZZZZZ} reads
$$ \sum_{i=1}^n c_i\, A_{ji}=\int_{\R^n}g\,Z_j\,dx+\bar{f}_j=d_j$$
that is, in matrix notation,~$Ac=d$. We can invert such relation using~\eqref{a-1} and write
$$ c=A^{-1} d=\alpha^{-1} d+ f^\sharp $$
with
\begin{equation}\label{ci3}
|f^\sharp|\le C\,\epsilon^{n/2}\,|d|\le
C\,\epsilon^{n/2}\left(
\|\psi\|_{L^{2}(\R^n)}+ \|g\|_{L^2(\R^n)}\right),\end{equation}
in virtue of~\eqref{d-2}.
So, using~\eqref{d-1},
$$ c_i = \alpha^{-1} d_i +f^\sharp_i =
\alpha^{-1}
\int_{\R^n}g\,Z_i\,dx+\alpha^{-1}\bar{f}_i
+f^\sharp_i.$$
This proves~\eqref{ci} with
$$ f_i:=\alpha^{-1}\bar{f}_i +f^\sharp_i,$$
and then~\eqref{ci2} follows from~\eqref{bar f}
and~\eqref{ci3}.
\end{proof}

Now, we show that solutions to~\eqref{EQ g} satisfy an a priori estimate.
We remark that the result in the following lemma is different from the one in Corollary~\ref{cor:300}, since here also a combination of~$Z_i$, for~$i=1,\ldots,n$, appears in the equation satisfied by~$\psi$.

\begin{lem}\label{lem:500}
Let~$g\in L^2(\R^n)$ with~$\|g\|_{\star,\xi}<+\infty$.
Let~$\psi\in\Psi$ be a solution to~\eqref{EQ g} for some
coefficients~$c_i\in\R$, $i=1,\ldots,n$ and for~$\epsilon$ sufficiently small.

Then,
$$ \|\psi\|_{\star,\xi}\le C\|g\|_{\star,\xi}.$$
\end{lem}

\begin{proof}
Suppose by contradiction that there exists a sequence~$\epsilon_j\searrow 0$
as~$j\rightarrow +\infty$
such that, for any~$j\in\N$, the function~$\psi_j$ satisfies
\begin{equation}\label{eq gj}
\left\{
\begin{matrix}
(-\Delta)^s \psi_j +\psi_j -pw_{\xi_j}^{p-1}\psi_j +g_j =\displaystyle\sum_{i=1}^n c_{i}^j\,Z_i^j & {\mbox{ in $\Omega_{\epsilon_j}$,}}\\
\psi_j =0 & {\mbox{ in }}\R^n\setminus\Omega_{\epsilon_j},\\
\displaystyle \int_{\Omega_{\epsilon_j}}\psi_j\,Z_i^j\,dx=0 & {\mbox{ for any }}i=1,\ldots,n,
\end{matrix}
\right.
\end{equation}
for suitable~$g_j\in L^2(\R^n)$ and~$\xi_j\in\Omega_{\epsilon_j}$,
where
$$ Z_i^j:=\frac{\partial w_{\xi_j}}{\partial x_i}.$$
Moreover,
\begin{eqnarray}
&& \|\psi_j\|_{\star,\xi_j} =1 {\mbox{ for any }}j\in\N \label{bound}\\
{\mbox{and }} && \|g_j\|_{\star,\xi_j}\searrow 0 {\mbox{ as }}j\rightarrow +\infty.\label{bound2}
\end{eqnarray}
Notice that the fact that the equation in~\eqref{eq gj} is linear
with respect to~$\psi_j$, $g_j$ and~$Z_i^j$ allows us to take
the sequences~$\psi_j$ and~$g_j$ as in~\eqref{bound} and~\eqref{bound2}.

We claim that, for any given~$R>0$,
\begin{equation}\label{claim lim}
\|\psi_j\|_{L^\infty(B_R(\xi_j))}\rightarrow 0 {\mbox{ as }}j\rightarrow +\infty.
\end{equation}
For this, we argue by contradiction and we assume that
there exists~$\delta>0$ and~$j_0\in\N$ such that,
for any~$j\ge j_0$, we have that~$\|\psi_j\|_{L^\infty(B_R(\xi_j))}\ge\delta$.

Thanks to Lemmata~\ref{lem:ci} and~\ref{:decay Z},
we have that
$$ |c_i^j|\le \frac{C_1}{\alpha_i}\|g_j\|_{\star,\xi_j} +C_2\,\epsilon_j^{n/2},$$
for suitable positive constants~$C_1$ and~$C_2$. Hence, from~\eqref{bound2}
we obtain that
\begin{equation}\label{bound3}
c_i^j\searrow 0 {\mbox{ as $j\rightarrow +\infty$ for any~$i\in\{1,\ldots,n\}$.}}
\end{equation}

Now, from Lemma~\ref{lem:100} we have that
\begin{equation}\begin{split}\label{norma Cs}
&\sup_{x\neq y}\frac{|\psi_j(x)-\psi_j(y)|}{|x-y|^s}\\
&\quad \le
C\left(\left\|g_j +\sum_{i=1}^nc_i^j\,Z_i^j+pw_{\xi_j}^{p-1}\psi_j\right\|_{L^\infty(\R^n)}+ \left\|g_j +\sum_{i=1}^nc_i^j\,Z_i^j+pw_{\xi_j}^{p-1}\psi_j\right\|_{L^2(\R^n)}\right).
\end{split}\end{equation}
We observe that
$$ \left\|g_j +\sum_{i=1}^nc_i^j\,Z_i^j+pw_{\xi_j}^{p-1}\psi_j\right\|_{L^\infty(\R^n)}\le C\left(\|g\|_{\star,\xi_j} + \sum_{i=1}^n|c^j_i|+\|\psi_j\|_{\star,\xi_j}\right),$$
thanks to the decay of~$Z^j_i$ in Lemma~\ref{:decay Z} and the
fact that~$w_{\xi_j}^{p-1}$ is bounded
(recall~\eqref{decay w}). So, from~\eqref{bound2},~\eqref{bound3} and~\eqref{bound}, we obtain that
\begin{equation}\label{norma infinito}
\left\|g_j +\sum_{i=1}^nc_i^j\,Z_i^j+pw_{\xi_j}^{p-1}\psi_j\right\|_{L^\infty(\R^n)}\le C,
\end{equation}
for a suitable constant~$C>0$ independent of~$j$.

We claim that
\begin{equation}\label{norma  star}
\left\|g_j +\sum_{i=1}^nc_i^j\,Z_i^j+pw_{\xi_j}^{p-1}\psi_j\right\|_{L^2(\R^n)}\le C,\end{equation}
where~$C>0$ does not depend on~$j$. Indeed,
\begin{eqnarray*}
\|g_j\|_{L^2(\R^n)}&=& \left( \int_{\R^n}g_j^2\,dx\right)^{1/2}\le \|g_j\|_{\star,\xi_j} \left(\int_{\R^n}\rho_{\xi_j}^2\,dx\right)^{1/2}\\
&\le &\|g_j\|_{\star,\xi_j}\left(\int_{\R^n}\frac{1}{(1+|x-\xi_j|)^{2\mu}}\,dx\right)^{1/2}\le C\|g_j\|_{\star,\xi_j}\le C,
\end{eqnarray*}
since~$2\mu>n$ and~\eqref{bound2} holds. Moreover,
$$
\left\|\sum_{i=1}^nc_i^j\,Z_i^j\right\|_{L^2(\R^n)}\le
\sum_{i=1}^n|c_i^j|\|Z_i^j\|_{L^2(\R^n)}\le C,
$$
thanks to~\eqref{bound3} and Lemma~\ref{:decay Z}.
Finally, using~\eqref{decay w}, the fact that~$2\mu>n$ and~\eqref{bound}, we have that
\begin{eqnarray*}
\left\|pw_{\xi_j}^{p-1}\psi_j\right\|_{L^2(\R^n)} &\le &
p\|\psi_j\|_{\star,\xi_j} \left( \int_{\R^n} w_{\xi_j}^{2(p-1)} \rho_{\xi_j}^2\,dx\right)^{1/2}\\
&\le & C\|\psi_j\|_{\star,\xi_j} \left(\int_{\R^n}
\frac{1}{(1+|x-\xi_j|)^{2(p-1)(n+2s)+2\mu}}\,dx\right)^{1/2}\le C.
\end{eqnarray*}
Putting together the above estimates, we obtain~\eqref{norma  star}.

Hence, from~\eqref{norma Cs},~\eqref{norma infinito} and~\eqref{norma  star},
we have that the~$\psi_j$'s are equicontinuous.

For any~$j=1,\ldots,n$, we define the function
$$ \tilde\psi_j(x):=\psi_j(x+\xi_j),$$
and the set
$$ \tilde\Omega_j:=\{x=y-\xi_j {\mbox{ with }}y\in\Omega_{\epsilon_j}\}.$$
We notice that~$\tilde\psi_j$ satisfies
\begin{equation}\label{eq psi tilde}
(-\Delta)^s\tilde\psi_j+\tilde\psi_j-pw^{p-1}\tilde\psi_j+\tilde g_j=\sum_{i=1}^n c^j_i\,\tilde Z_i {\mbox{ in }}\tilde\Omega_j,
\end{equation}
where~$\tilde g_j(x):=g(x+\xi_j)$ and~$\tilde Z_i:=\frac{\partial w}{\partial x_i}$. Moreover
\begin{equation}\label{zero fuori}
\tilde\psi_j=0 {\mbox{ in }}\R^n\setminus\tilde\Omega_j.
\end{equation}

Now, thanks to~\eqref{dist bordo} we have
that~$B_{c/\epsilon_j}(\xi_j)\subset\Omega_{\epsilon_j}$.
Hence,~$B_{c/\epsilon_j}\subset\tilde\Omega_j$, which means that~$\tilde\Omega_j$
converges to~$\R^n$ when~$j\rightarrow +\infty$.

Furthermore, we have that
\begin{eqnarray}
&& \|\tilde\psi_j\|_{L^\infty(B_R)}\ge \delta \nonumber\\
{\mbox{and }}&& \|(1+|x|)^\mu \tilde\psi_j\|_{L^\infty(\R^n)}=1. \label{9.34bis}
\end{eqnarray}

Now, since the~$\psi_j$'s are equicontinuous, the~$\tilde\psi_j$'s are equicontinuous too, and therefore there exists a function~$\bar\psi$ such that, up to subsequences,~$\tilde\psi_j$ converge to~$\bar\psi$ uniformly on compact sets.

The function~$\bar\psi\in L^2(\R^n)$. Indeed, by Fatou's Theorem and~\eqref{bound} and recalling that~$2\mu>n$, we have
$$ \int_{\R^n}\bar\psi^2\,dx\le\liminf_{j\rightarrow +\infty}\int_{\R^n}\psi_j^2\,dx \le\liminf_{j\rightarrow +\infty}\|\psi_j\|_{\star,\xi_j}^2\int_{\R^n}\frac{1}{(1+|x-\xi_j|)^{2\mu}}\,dx\le C. $$
Moreover, $\bar\psi$ satisfies the conditions
\begin{eqnarray}
&& \|\bar\psi\|_{L^\infty(B_R)}\ge\delta \label{bar psi}\\
{\mbox{ and }} && \|(1+|x|)^\mu \bar\psi\|_{L^\infty(\R^n)}\le 1. \label{unif}
\end{eqnarray}

We prove that~$\bar\psi$ solves the equation
\begin{equation}\label{eq strong}
(-\Delta)^s\bar\psi +\bar\psi = pw^{p-1}\bar\psi {\mbox{ in }}\R^n.
\end{equation}
Indeed, we multiply the equation in~\eqref{eq psi tilde} by a function~$\eta\in C^{\infty}_0(\tilde\Omega_j)$  and we integrate over~$\R^n$.
We notice that both~$\eta$ and~$\tilde\psi_j$ are equal to zero outside~$\tilde\Omega_j$ (recall~\eqref{zero fuori}), and therefore we can use
formula~$(1.5)$ in~\cite{SerraRos-Dirichlet-regularity} and we get
\begin{equation}\label{eq forte}
\int_{\R^n} \left( (-\Delta)^s\eta+\eta-pw^{p-1}\eta\right) \tilde\psi_j\,dx+ \int_{\R^n}\tilde g_j\,\eta\,dx = \sum_{i=1}^n c^j_i\int_{\R^n}\tilde Z_i\,\eta\,dx. \end{equation}
Now, we have that
$$ \|\tilde g_j\|_{\star,0}=\|g_j\|_{\star,\xi_j}\searrow 0 {\mbox{ as }}j\rightarrow+\infty.$$
Moreover,
$$ \left|\int_{\R^n}\tilde g_j\,\eta\,dx\right| \le \|\tilde g_j\|_{\star,0}\int_{\R^n}\rho_0\,\eta\le C\|\tilde g_j\|_{\star,0},$$
since~$2\mu>n$, which implies that
\begin{equation}\label{g zero}
\int_{\R^n}\tilde g_j\,\eta\,dx \rightarrow 0 {\mbox { as }}j\rightarrow +\infty.
\end{equation}

Also,
$$ \left|\sum_{i=1}^n c^j_i\int_{\R^n}\tilde Z_i\,\eta\,dx\right|\le C \sum_{i=1}^n |c^j_i|,$$
and so, thanks to~\eqref{bound3}, we obtain that
\begin{equation}\label{c zero}
\sum_{i=1}^n c^j_i\int_{\R^n}\tilde Z_i\,\eta\,dx\rightarrow 0 {\mbox{ as }}j\rightarrow +\infty.
\end{equation}

Finally, we fix~$r>0$ and we estimate
\begin{equation}\begin{split}\label{spezza}
&\left| \int_{\R^n} \left( (-\Delta)^s\eta+\eta-pw^{p-1}\eta\right) \tilde\psi_j\,dx-\int_{\R^n} \left( (-\Delta)^s\eta+\eta-pw^{p-1}\eta\right) \bar\psi\,dx\right| \\
&\qquad \le \int_{\R^n} |(-\Delta)^s\eta+\eta-pw^{p-1}\eta |\,
|\tilde\psi_j-\bar\psi|\,dx\\
&\qquad =\int_{B_r} |(-\Delta)^s\eta+\eta-pw^{p-1}\eta |\,
|\tilde\psi_j-\bar\psi|\,dx +\int_{\R^n\setminus B_r} |(-\Delta)^s\eta+\eta-pw^{p-1}\eta |\,
|\tilde\psi_j-\bar\psi|\,dx.
\end{split}\end{equation}
We define the function
$$ \tilde \eta :=(-\Delta)^s\eta+\eta-pw^{p-1}\eta $$
and we notice that it satisfies the following decay
\begin{equation}\label{decay eta}
|\tilde\eta(x)|\le\frac{C_1}{(1+|x|)^{n+2s}}.
\end{equation}
Hence,
\begin{eqnarray*}
\int_{B_r} |(-\Delta)^s\eta+\eta-pw^{p-1}\eta |\,
|\tilde\psi_j-\bar\psi|\,dx &\le & C_1 \|\tilde\psi_j-\bar\psi\|_{L^{\infty}(B_r)}\int_{B_r}\frac{1}{(1+|x|)^{n+2s}}\\
&\le & C_2\|\tilde\psi_j-\bar\psi\|_{L^{\infty}(B_r)},
\end{eqnarray*}
which implies that
\begin{equation}\label{spezza 1}
\int_{B_r} |(-\Delta)^s\eta+\eta-pw^{p-1}\eta |\,
|\tilde\psi_j-\bar\psi|\,dx\searrow 0 {\mbox{ as }}j\rightarrow +\infty,
\end{equation}
due to the uniform convergence of~$\tilde\psi_j$ to~$\bar\psi$
on compact sets.
On the other hand, from~\eqref{9.34bis},~\eqref{unif} and~\eqref{decay eta}
we have that
\begin{eqnarray*}
\int_{\R^n\setminus B_r} |(-\Delta)^s\eta+\eta-pw^{p-1}\eta |\,
|\tilde\psi_j-\bar\psi|\,dx &\le & 2\,C_1 \int_{\R^n\setminus B_r}\frac{1}{(1+|x|)^{n+2s}}\,dx\\
&\le & 2\,C_1 \int_{\R^n\setminus B_r}\frac{1}{|x|^{n+2s}}\,dx\\
&\le & C_3 r^{-2s}.
\end{eqnarray*}
Hence, sending~$r\rightarrow +\infty$, we obtain that
\begin{equation}\label{spezza 2}
\int_{\R^n\setminus B_r} |(-\Delta)^s\eta+\eta-pw^{p-1}\eta |\,
|\tilde\psi_j-\bar\psi|\,dx\searrow 0.
\end{equation}
Putting together~\eqref{spezza},~\eqref{spezza 1} and~\eqref{spezza 2}, we obtain that
$$ \int_{\R^n} \left( (-\Delta)^s\eta+\eta-pw^{p-1}\eta\right) \tilde\psi_j\,dx\rightarrow \int_{\R^n} \left( (-\Delta)^s\eta+\eta-pw^{p-1}\eta\right) \bar\psi\,dx {\mbox{ as }}j\rightarrow +\infty.$$
This,~\eqref{g zero},~\eqref{c zero} and~\eqref{eq forte} imply that
$$ \int_{\R^n}\left( (-\Delta)^s\eta+\eta-pw^{p-1}\eta\right) \bar\psi\,dx=0$$
for any~$\eta\in C^{\infty}_0(\R^n)$. This means that~$\bar\psi$ is a weak solution to~\eqref{eq strong}, and so a strong solution, thanks to~\cite{SerVal}.

Hence, recalling the nondegeneracy result in~\cite{FLS}, we have that
\begin{equation}\label{comb}
\bar\psi=\sum_{i=1}^n\beta_i\frac{\partial w}{\partial x_i},
\end{equation}
for some coefficients~$\beta_i\in\R$.

On the other hand, the orthogonality
condition in~\eqref{eq gj} passes to the limit, that is
\begin{equation}\label{ortogo}
\int_{\R^n}\bar\psi\,\tilde Z_i\,dx =0, {\mbox{ for any }}i=1,\ldots,n.
\end{equation}
Indeed, we fix~$r>0$ and we compute
\begin{equation}\label{alwpqoqw}
\int_{\R^n} (\bar\psi-\tilde\psi_j)\tilde Z_i\,dx =\int_{B_r}(\bar\psi-\tilde\psi_j)\tilde Z_i\,dx+\int_{\R^n\setminus B_r}(\bar\psi-\tilde\psi_j)\tilde Z_i\,dx.
\end{equation}
Concerning the first term on the right-hand side, we use the uniform convergence of~$\tilde\psi_j$ to~$\bar\psi$ on compact sets together with the fact that~$\tilde Z_i$ is bounded to obtain that
$$ \int_{B_r}(\bar\psi-\tilde\psi_j)\tilde Z_i\,dx\rightarrow 0 {\mbox{ as }}j\rightarrow +\infty.$$
As for the second term, we use~\eqref{9.34bis},~\eqref{unif} and
Lemma~\ref{:decay Z} and we get
$$ \int_{\R^n\setminus B_r}(\bar\psi-\tilde\psi_j)\tilde Z_i\,dx\le
C\int_{\R^n\setminus B_r}\frac{1}{|x|^{n+2s}}\,dx\le \bar C\,r^{-2s},$$
which tends to zero as~$r\rightarrow +\infty$.
Using the above two formulas into~\eqref{alwpqoqw} we obtain that
$$ 0=\int_{\R^n}\tilde\psi_j\tilde Z_i\,dx \rightarrow \int_{\R^n}\bar\psi\tilde Z_i\,dx, $$
which implies~\eqref{ortogo}.

Therefore, recalling also~\eqref{Zj},~\eqref{comb} and~\eqref{ortogo} imply that~$\bar\psi\equiv 0$,
thus contradicting~\eqref{bar psi}.
This proves~\eqref{claim lim}.

Now, from Corollary~\ref{lem:200} (notice that we can take~$R$ sufficiently big in order to apply the corollary) we have that
\begin{eqnarray*}
\|\psi_j\|_{\star,\xi_j} &\le & C\left(\|\psi_j\|_{L^\infty(B_R(\xi_j))}+\left\|g_j+\sum_{i=1}^nc_i^j\,Z_i^j\right\|_{\star,\xi_j}\right)\\
&\le & C\left(\|\psi_j\|_{L^\infty(B_R(\xi_j))}+\|g_j\|_{\star,\xi_j}+\left\|\sum_{i=1}^n|c_i^j|\,Z_i^j\right\|_{\star,\xi_j}\right)\\
&= & C\left(\|\psi_j\|_{L^\infty(B_R(\xi_j))}+\|g_j\|_{\star,\xi_j}+\left\|\sum_{i=1}^n|c_i^j|\,\rho_{\xi_j}^{-1}\,Z_i^j\right\|_{L^\infty(\R^n)}\right)\\
&\le & C\left(\|\psi_j\|_{L^\infty(B_R(\xi_j))}+\|g_j\|_{\star,\xi_j}+\sum_{i=1}^n|c_i^j|\right)
\end{eqnarray*}
up to renaming~$C$, where we have used the decay of~$Z^j_i$ (see Lemma~\ref{:decay Z}) and the fact that~$\mu<n+2s$. Therefore,~\eqref{bound2},~\eqref{claim lim} and~\eqref{bound3} imply that
$$ \|\psi_j\|_{\star,\xi_j}\rightarrow 0 {\mbox{ as }}j\rightarrow +\infty, $$
which contradicts~\eqref{bound} and concludes the proof.
\end{proof}

Now we consider an auxiliary problem: we look for a solution~$\psi\in\Psi$
of
\begin{equation}\label{eq aux}
(-\Delta)^s \psi +\psi +g =\sum_{i=1}^n c_{i}\,Z_i  {\mbox{ in $\Omega_\epsilon$,}}
\end{equation}
and we prove the following:

\begin{pro}\label{prop:400}
Let~$g\in L^2(\R^n)$ with~$\|g\|_{\star,\xi}<+\infty$.
Then, there exists a unique~$\psi\in\Psi$
solution to~\eqref{eq aux}.

Furthermore, there exists a constant~$C>0$ such that
\begin{equation}\label{lkjhsirhljtr}
\|\psi\|_{\star,\xi}\le C\|g\|_{\star,\xi}.
\end{equation}
\end{pro}

\begin{proof}
We first prove the existence of a solution to~\eqref{eq aux}.

First of all, we notice that formula~$(1.5)$ in~\cite{SerraRos-Dirichlet-regularity} implies that, for any~$\psi,\varphi\in\Psi$,
$$ \int_{\R^n}(-\Delta)^s\psi\,\varphi\,dx= \int_{\R^n}(-\Delta)^{s/2}\psi\,(-\Delta)^{s/2}\varphi\,dx= \int_{\R^n}\psi\,(-\Delta)^s\varphi\,dx.$$
Now, given~$g\in L^2(\R^n)$, we look for a solution~$\psi\in\Psi$
of the problem
\begin{equation}\label{eq weak}
\int_{\R^n}(-\Delta)^{s/2}\psi\, (-\Delta)^{s/2}\varphi\,dx+\int_{\R^n}\psi\,\varphi\,dx +\int_{\R^n}g\,\varphi\,dx=0,
\end{equation}
for any~$\varphi\in\Psi$. Subsequently we will show that~$\psi$ is a solution to
the original problem~\eqref{eq aux}.

We observe that
$$ \left<\psi,\varphi\right>:=\int_{\R^n}(-\Delta)^{s/2}\psi\, (-\Delta)^{s/2}\varphi\,dx+\int_{\R^n}\psi\,\varphi\,dx $$
defines an inner product in~$\Psi$, and that
$$ F(\varphi):=-\int_{\R^n}g\,\varphi\,dx $$
is a linear and continuous functional on~$\Psi$.
Hence, from Riesz's Theorem we have that there exists a unique
function~$\psi\in\Psi$ which solves~\eqref{eq weak}.

We claim that
\begin{equation}\label{strong}
{\mbox{$\psi$ is a strong solution to~\eqref{eq aux}.}}
\end{equation}
For this, we take
a radial cutoff~$\tau\in C^\infty_0(\Omega_\eps)$
of the form~$\tau(x)=\tau_o(|x-\xi|)$, for some smooth and compactly
supported real function, and we use Lemma~\ref{lem:orto}.
So, for any~$\phi\in H^{s}(\R^n)$ such that~$\phi=0$
outside~$\Omega_\epsilon$, we define
$$ \tilde\phi:=\phi-\sum_{i=1}^n \lambda_i(\phi)\,\tilde Z_i,$$
where
\begin{equation}\label{J98} \lambda_i(\phi):= \tilde\alpha^{-1}
\, \int_{\R^n} \phi\,Z_i\,dx,\end{equation}
and~$\tilde Z_i$ and~$\tilde\alpha$
are as in Lemma~\ref{lem:orto}.
We remark that~$\tilde Z_i$ vanishes outside~$\Omega_\eps$,
hence so does~$\tilde\phi$. Furthermore, for any~$j\in\{1,\dots,n\}$,
\begin{eqnarray*}
\int_{\Omega_\eps} \tilde\phi\,Z_j\,dx &=& \int_{\R^n} \tilde\phi\,Z_j\,dx
\\ &=&\int_{\R^n} \phi\,Z_j\,dx -
\sum_{i=1}^n \lambda_i(\phi)\,\int_{\R^n}\tilde Z_i\,Z_j\,dx\\
&=& \int_{\R^n} \phi\,Z_j\,dx -
\sum_{i=1}^n \lambda_i(\phi)\,\tilde\alpha\,\delta_{ij} \\
&=& \int_{\R^n} \phi\,Z_j\,dx -
\lambda_j(\phi)\,\tilde\alpha
\\ &=& 0,
\end{eqnarray*}
thanks to Lemma~\ref{lem:orto} and~\eqref{J98}.
This shows that~$\tilde\phi\in\Psi$.

As a consequence, we can use~$\tilde\phi$ as a test function in~\eqref{eq weak}
and conclude that
\begin{eqnarray*}
&&\int_{\R^n}(-\Delta)^{s/2}\psi\, (-\Delta)^{s/2}\left(\phi-
\sum_{i=1}^n \lambda_i(\phi)\,\tilde Z_i
\right)\,dx\\
&&\qquad +\int_{\R^n}\psi\left(\phi-
\sum_{i=1}^n \lambda_i(\phi)\,\tilde Z_i
\right)\,dx +\int_{\R^n}g\left(\phi-\sum_{i=1}^n \lambda_i(\phi)\,\tilde Z_i
\right)\,dx=0,
\end{eqnarray*}
that is
\begin{equation}\begin{split}\label{oqwpgn}
&\int_{\R^n}(-\Delta)^{s/2}\psi\, (-\Delta)^{s/2}\phi\,dx
+\int_{\R^n}\psi\,\phi\,dx +\int_{\R^n}g\,\phi\,dx\\
&\qquad\qquad =\int_{\R^n}\left(\psi+g\right)
\sum_{i=1}^n \lambda_i(\phi)\,\tilde Z_i
\,dx+\int_{\R^n}(-\Delta)^s\psi\,
\sum_{i=1}^n \lambda_i(\phi)\,\tilde Z_i \,dx\\
&\qquad\qquad =\sum_{i=1}^n \lambda_i(\phi) \int_{\R^n}\left(\psi+g
+(-\Delta)^s\psi\right)\,\tilde Z_i \,dx.
\end{split}
\end{equation}
Now we define
$$ b_i:= \tilde\alpha^{-1}\,
\int_{\R^n}\left(\psi+g
+(-\Delta)^s\psi\right)\,\tilde Z_i \,dx,$$
we recall~\eqref{J98} and we write~\eqref{oqwpgn}
as
\begin{eqnarray*}
&& \int_{\R^n}(-\Delta)^{s/2}\psi\, (-\Delta)^{s/2}\phi\,dx
+\int_{\R^n}\psi\,\phi\,dx +\int_{\R^n}g\,\phi\,dx\\
&&\qquad\qquad =\sum_{i=1}^n \lambda_i(\phi) \,\tilde\alpha\, b_i\\
&&\qquad\qquad =\sum_{i=1}^n b_i \,\int_{\R^n} \phi\,Z_i\,dx.
\end{eqnarray*}
Since~$\phi$ is any test function, this means that~$\psi$
is a solution of
$$ (-\Delta)^s\psi+\psi+g = \sum_{i=1}^n b_i\,Z_i.$$
in a weak sense, and therefore in a strong sense, thanks to~\cite{SerVal},
thus proving~\eqref{strong}.

Now, we prove the uniqueness of the solution to~\eqref{eq aux}.
For this, suppose by contradiction that there exist~$\psi_1$ and~$\psi_2$
in~$\Psi$ that solve~\eqref{eq aux}. We set
$$ \tilde\psi:=\psi_1-\psi_2,$$
and we observe that~$\tilde\psi\in\Psi$ and solves
\begin{equation}\label{eq differenza}
(-\Delta)^s \tilde\psi +\tilde\psi =\sum_{i=1}^n a_{i}\,Z_i  {\mbox{ in $\Omega_\epsilon$,}}
\end{equation}
for suitable coefficients~$a_i\in\R$,~$i=1,\ldots,n$.

We multiply the equation in~\eqref{eq differenza} by~$\tilde\psi$ and we integrate
over~$\Omega_\eps$, obtaining that
$$ \int_{\Omega_\eps}(-\Delta)^s\tilde\psi\,\tilde\psi +\tilde\psi^2\,dx=0,$$
since~$\tilde\psi\in\Psi$ (and so it is orthogonal to~$Z_i$ in~$L^2(\Omega_\eps)$ for any~$i=1,\ldots,n$).
Since~$\tilde\psi=0$ outside~$\Omega_\eps$, we can apply formula~$(1.5)$ in~\cite{SerraRos-Dirichlet-regularity} and we obtain that
$$ \int_{\R^n}\left|(-\Delta)^{s/2}\tilde\psi\right|^2 +\tilde\psi^2\,dx=0,$$
that is
$$ \|\tilde\psi\|_{H^s(\R^n)}=0,$$
which implies that~$\tilde\psi=0$.
Thus~$\psi_1=\psi_2$ and this concludes the proof of the uniqueness.

It remains to establish~\eqref{lkjhsirhljtr}. Thanks to~\eqref{eq aux}
and Corollary~\ref{cor:300}, we have that
\begin{equation}\label{pqwwq3275238}
\|\psi\|_{\star,\xi} \le C\,\left\|g-\sum_{i=1}^nc_i\,Z_i\right\|_{\star,\xi}
\le C\left(\|g\|_{\star,\xi} +\left\|\sum_{i=1}^nc_i\,Z_i\right\|_{\star,\xi}\right).
\end{equation}
First, we observe that for any~$i=1,\ldots,n$
$$ \|Z_i\|_{\star,\xi}=\sup_{\R^n}|\rho_\xi^{-1}\,Z_i|\le C_1, $$
due to Lemma~\ref{:decay Z} and the fact that~$\mu<n+2s$ (recall also~\eqref{rho}). Hence,
\begin{equation}\label{9.65bis}
\left\|\sum_{i=1}^nc_i\,Z_i\right\|_{\star,\xi}\le \sum_{i=1}^n|c_i|\,\|Z_i\|_{\star,\xi}\le C_1 \sum_{i=1}^n|c_i|. \end{equation}
Now, we claim that
\begin{equation}\label{ljytpqsa}
\left\|\sum_{i=1}^nc_i\,Z_i\right\|_{\star,\xi} \le \, C_2 \left(\|\psi\|_{L^2(\R^n)}+ \|g\|_{L^2(\R^n)}\right).
\end{equation}
Indeed, we recall Lemma~\ref{lem:orto},
we multiply equation~\eqref{eq aux} by~$\tilde Z_j$, for~$j\in\{1,\ldots,n\}$,
and we integrate over~$\R^n$, obtaining that
\begin{equation}\label{dopoooo}
\int_{\R^n}(-\Delta)^s\psi\, \tilde Z_j +\psi\,\tilde Z_j +g\,\tilde Z_j\,dx=\tilde\alpha c_j,
\end{equation}
where~$\tilde Z_j$ and~$\tilde\alpha$ are as in Lemma~\ref{lem:orto}.
Thanks to formula~$(1.5)$ in~\cite{SerraRos-Dirichlet-regularity}, we have that
$$ \left|\int_{\R^n}(-\Delta)^s\psi\, \tilde Z_j\,dx\right|=\left|\int_{\R^n}\psi\, (-\Delta)^s\tilde Z_j\,dx\right|\le
\|(-\Delta)^s\tilde Z_j\|_{L^2(\R^n)} \|\psi\|_{L^2(\R^n)},$$
where we have used H\"older inequality.
Therefore, this and~\eqref{dopoooo} give that
$$ \tilde\alpha |c_j|\le
\|(-\Delta)^s\tilde Z_j\|_{L^2(\R^n)} \|\psi\|_{L^2(\R^n)}+ \|\tilde Z_j
\|_{L^2(\R^n)} \|\psi\|_{L^2(\R^n)}+ \|\tilde Z_j\|_{L^2(\R^n)} \|g\|_{L^2(\R^n)},$$
which, together with~\eqref{9.65bis}, implies~\eqref{ljytpqsa}, since both~$\|(-\Delta)^s\tilde Z_j\|_{L^2(\R^n)}$ and~$\|\tilde Z_j\|_{L^2(\R^n)}$ are bounded
(recall Lemma~\ref{lem:decay D2 Z}).

Now, we observe that
\begin{equation}\label{mcvnmxcb}
\|\psi\|_{L^2(\R^n)}\le \|g\|_{L^2(\R^n)}.
\end{equation}
Indeed, we multiply equation~\eqref{eq aux} by~$\psi$ and we integrate
over~$\Omega_\epsilon$: we obtain
$$ \int_{\Omega_\eps}(-\Delta)^s\psi\,\psi+\psi^2+g\,\psi\,dx=0,$$
since~$\psi\in\Psi$. We notice that the first term in the above formula is
quadratic, and so, using H\"older inequality, we have that
$$ \int_{\Omega_\eps}\psi^2\,dx\le \int_{\Omega_\eps}(-g)\psi\,dx\le
\|g\|_{L^2(\R^n)}\|\psi\|_{L^2(\R^n)}, $$
which implies~\eqref{mcvnmxcb}.

Therefore, from~\eqref{ljytpqsa} and~\eqref{mcvnmxcb}, we deduce that
$$ \left\|\sum_{i=1}^nc_i\,Z_i\right\|_{\star,\xi}\le 2\,C_3\,\|g\|_{L^2(\R^n)}.$$
Moreover,
$$  \|g\|_{L^2(\R^n)}\le \|g\|_{\star,\xi}\left(\int_{\R^n}\rho_\xi^{2}\,dx\right)^{1/2}=\|g\|_{\star,\xi}\left(\int_{\R^n}\frac{1}{(1+|x-\xi|)^{2\mu}}\,dx\right)^{1/2}\le C\|g\|_{\star,\xi}, $$
since~$2\mu>n$.
The above two formulas give that
$$  \left\|\sum_{i=1}^nc_i\,Z_i\right\|_{\star,\xi}\le C_4\,\|g\|_{\star,\xi}.$$
This and~\eqref{pqwwq3275238} shows~\eqref{lkjhsirhljtr} and conclude the proof.
\end{proof}

Now, for any~$g\in L^2(\R^n)$ with~$\|g\|_{\star,\xi}<+\infty$,
we denote by~$\mathcal A[g]$
the unique solution to~\eqref{eq aux}. We notice that Proposition~\ref{prop:400}
implies that the operator~$\mathcal A$ is well defined and that
$$ \|\mathcal A[g]\|_{\star,\xi}\le C \|g\|_{\star,\xi}.$$
We also remark that~$\mathcal A$ is a linear operator.

We consider the Banach space
\begin{equation}\label{Y star}
Y_\star:=\{\psi :\R^n\rightarrow\R {\mbox{ s.t. }}\|\psi\|_{\star,\xi}<+\infty\}
\end{equation}
endowed with the norm~$\|\cdot\|_{\star,\xi}$.

With this notation we can prove the main theorem of the linear theory, i.e. Theorem~\ref{linear}.

\begin{proof}[Proof of Theorem~\ref{linear}]
We notice that solving~\eqref{EQ g} is equivalent to find a function~$\psi\in\Psi$ such that
\begin{equation}\label{id meno comp}
\psi-\mathcal A[-pw_\xi^{p-1}\psi]=\mathcal A[g].
\end{equation}
For this, we set
\begin{equation}\label{math B}
\mathcal B[\psi]:= \mathcal A[-pw_\xi^{p-1}\psi].
\end{equation}
Recalling the definition of~$Y_\star$ given in~\eqref{Y star},
we observe that
\begin{equation}\label{LLLLLLLL}
{\mbox{if~$\psi\in Y_\star$ then~$\mathcal B[\psi]\in Y_\star$.}}
\end{equation}
Indeed, from Proposition~\ref{prop:400} we deduce that~$\mathcal B[\psi]\in\Psi$ solves~\eqref{eq aux} with~$g:=-pw_\xi^{p-1}\psi$,
and so
$$ \|\mathcal B[\psi]\|_{\star,\xi}\le C\|-pw_\xi^{p-1}\psi\|_{\star,\xi}\le \tilde{C}\|\psi\|_{\star,\xi}, $$
for some~$\tilde{C}>0$ (recall that~$w_\xi$ is bounded thanks to~\eqref{decay w}), which proves~\eqref{LLLLLLLL}.

We claim that
\begin{equation}\label{contraction set}
{\mbox{$\mathcal B$ defines a compact operator in~$Y_\star$ with respect to the norm~$\|\cdot\|_{\star,\xi}$.}}
\end{equation}
Indeed, let~$(\psi_j)_j$ a bounded sequence in~$Y_\star$ with respect to the norm~$\|\cdot\|_{\star,\xi}$.
Then, thanks to Lemma~\ref{lem:100}, the fact
that~$w_\xi^{p-1}$ and~$Z^j_i$ are bounded and~$w_\xi^{p-1}\rho_\xi$ and~$Z^j_i$ belong to~$L^2(\R^n)$ and  Lemma~\ref{lem:ci}, we have that
\begin{eqnarray*}
&&\sup_{x\neq y}\frac{|\mathcal B[\psi_j](x)-\mathcal B[\psi_j](y)|}{|x-y|^s}\\
&&\le C_1 \left(\left\|-pw_\xi^{p-1}\psi_j+\sum_{i=1}^nc_i^j\,Z_i^j\right\|_{L^{\infty}(\R^n)}+\left\|-pw_\xi^{p-1}\psi_j+\sum_{i=1}^nc_i^j\,Z_i^j\right\|_{L^2(\R^n)}\right)\\
&&\le  C_2\left(\|\psi_j\|_{L^\infty(\R^n)}+ \sum_{i=1}^n|c_i^j|\|Z_i^j\|_{L^\infty(\R^n)} +\|\psi_j\|_{\star,\xi}\,
\|w_\xi^{p-1}\rho_\xi\|_{L^2(\R^n)}+\sum_{i=1}^n|c_i^j|\|Z^j_i\|_{L^2(\R^n)}
\right)\\
&&\le  C_3\left(\|\psi_j\|_{\star,\xi}+ \sum_{i=1}^n|c_i^j|\right)\\
&&\le  C_4,
\end{eqnarray*}
for suitable positive constants~$C_1, C_2, C_3$ and~$C_4$.
This gives the equicontinuity of the sequence~$\mathcal B[\psi_j]$,
and so it converges to a function~$\bar b$ uniformly on compact sets.
Hence, for any~$R>0$, we have
\begin{equation}\label{compact}
\|\mathcal B[\psi_j]-\bar b\|_{L^\infty(B_R(\xi))}\rightarrow 0 {\mbox{ as }}j\rightarrow +\infty.
\end{equation}

On the other hand, for any~$x\in\R^n\setminus B_R(\xi)$, we have the following estimate
\begin{eqnarray*}
|w_\xi^{p-1}(x)\psi_j(x)| &\le & \|\psi_j\|_{\star,\xi} |w_\xi^{p-1}(x)\rho_\xi(x)| \\
&\le & C_5\|\psi_j\|_{\star,\xi} \left|\frac{1}{(1+|x-\xi|)^{(n+2s)(p-1)}}\rho_\xi(x)\right|\\
&\le & C_5 \|\psi_j\|_{\star,\xi} \rho_\xi^{1+\frac{(n+2s)(p-1)}{\mu}}(x),
\end{eqnarray*}
for some~$C_5>0$, where we have used the decay of~$w_\xi$ in~\eqref{decay w}
and the expression of~$\rho_\xi$ given in~\eqref{rho}.
This implies that
$$ \sup_{x\in\R^n\setminus B_R(\xi)}|\rho_\xi^{-1}w_\xi^{p-1}\psi_j|\le C_5\|\psi_j\|_{\star,\xi}\,\sup_{x\in\R^n\setminus B_R(\xi)} \rho_\xi^{\sigma}(x),$$
where
\begin{equation}\label{08079767453}
\sigma:=\frac{(n+2s)(p-1)}{\mu}>0.
\end{equation}
Hence, since~$\psi_j$ is a uniformly bounded sequence with respect to
the norm~$\|\cdot\|_{\star,\xi}$, we obtain that
\begin{equation}\label{stima1}
\sup_{x\in\R^n\setminus B_R(\xi)}|\rho_\xi^{-1}\mathcal B[\psi_j]|\le C_6 \,\sup_{x\in\R^n\setminus B_R(\xi)}\rho_\xi^{\sigma}(x).
\end{equation}
It follows that
\begin{equation}\label{stima2}
\sup_{x\in\R^n\setminus B_R(\xi)}|\rho_\xi^{-1}\bar b|\le C_6\,\sup_{x\in\R^n\setminus B_R(\xi)} \rho_\xi^{\sigma}(x).
\end{equation}

We observe that
\begin{eqnarray*}
\sup_{x\in\R^n}\left|\rho_\xi^{-1}\left(\mathcal B[\psi_j]-\bar b\right)\right| &=&
\sup_{x\in\R^n} \left|\rho_\xi^{-1}\left(\mathcal B[\psi_j]-\bar b\right)\chi_{B_R(\xi)} +\rho_\xi^{-1}\left(\mathcal B[\psi_j]-\bar b\right)\chi_{\R^n\setminus B_R(\xi)} \right|\\
&\le & \sup_{x\in\R^n} \left(\left|\rho_\xi^{-1}\left(\mathcal B[\psi_j]-\bar b\right)\chi_{B_R(\xi)}\right| +\left|\rho_\xi^{-1}\left(\mathcal B[\psi_j]-\bar b\right)\chi_{\R^n\setminus B_R(\xi)} \right|\right)\\
&\le &  \sup_{x\in\R^n} \left|\rho_\xi^{-1}\left(\mathcal B[\psi_j]-\bar b\right)\chi_{B_R(\xi)}\right| +\sup_{x\in\R^n}\left|\rho_\xi^{-1}\left(\mathcal B[\psi_j]-\bar b\right)\chi_{\R^n\setminus B_R(\xi)} \right|\\
&= & \sup_{x\in B_R(\xi)} \left|\rho_\xi^{-1}\left(\mathcal B[\psi_j]-\bar b\right)\right| +\sup_{x\in\R^n\setminus B_R(\xi)}\left|\rho_\xi^{-1}\left(\mathcal B[\psi_j]-\bar b\right) \right|.
\end{eqnarray*}
Therefore, we obtain that
\begin{equation}\begin{split}\label{iew023nh10}
\|\mathcal B[\psi_j]-\bar b\|_{\star,\xi}&=\, \sup_{x\in\R^n}\left|\rho_\xi^{-1}\left(\mathcal B[\psi_j]-\bar b\right)\right|\\
&\le\, \sup_{x\in B_R(\xi)}\left|\rho_\xi^{-1}\left(\mathcal B[\psi_j]-\bar b\right)\right| +\sup_{x\in\R^n\setminus B_R(\xi)}\left|\rho_\xi^{-1}\left(\mathcal B[\psi_j]-\bar b\right) \right| \\
&\le\, \sup_{x\in B_R(\xi)}\left|\rho_\xi^{-1}\left(\mathcal B[\psi_j]-\bar b\right)\right| + C_7 \sup_{x\in\R^n\setminus B_R(\xi)}\rho_\xi^{\sigma}(x),
\end{split}\end{equation}
where we have also used~\eqref{stima1} and~\eqref{stima2}.
Concerning the first term in the right-hand side, we have
\begin{eqnarray*}
\sup_{x\in B_R(\xi)}\left|\rho_\xi^{-1}\left(\mathcal B[\psi_j]-\bar b\right)\right| &= &\sup_{x\in B_R(\xi)}\left|(1+|x-\xi|)^\mu\left(\mathcal B[\psi_j]-\bar b\right)\right|\\
&\le & (1+R)^\mu \|\mathcal B[\psi_j]-\bar b\|_{L^\infty(B_R(\xi))}.
\end{eqnarray*}
Therefore, sending~$j\rightarrow +\infty$ and recalling~\eqref{compact},
we obtain that
\begin{equation}\label{235668}
\sup_{x\in B_R(\xi)}\left|\rho_\xi^{-1}\left(\mathcal B[\psi_j]-\bar b\right)\right| \rightarrow 0 {\mbox{ as }}j\rightarrow +\infty.
\end{equation}
Now, we send~$R\rightarrow +\infty$ and, recalling~\eqref{08079767453}, we get
\begin{equation}\label{09890898-4}
\sup_{x\in\R^n\setminus B_R(\xi)}\rho_\xi^{\sigma}(x)\rightarrow 0 {\mbox{ as }}R\rightarrow +\infty.
\end{equation}
Putting together~\eqref{iew023nh10},~\eqref{235668} and~\eqref{09890898-4},
we obtain that
$$ \|\mathcal B[\psi_j]-\bar b\|_{\star,\xi}\rightarrow 0 {\mbox{ as }}j\rightarrow +\infty,$$
and this shows~\eqref{contraction set}.

From~\eqref{lkjhsirhljtr} in Proposition~\ref{prop:400}, we deduce that
if~$g=0$ then~$\psi=\mathcal A[g]=0$ is the unique solution to~\eqref{eq aux},
and so by Fredholm's alternative we obtain that, for any~$g\in Y_\star$,
there exists a unique~$\psi$ that solves~\eqref{id meno comp}
(recall~\eqref{math B} and~\eqref{contraction set}).
This gives existence and uniqueness of the solution to~\eqref{EQ g},
while estimate~\eqref{est fin} follows from Lemma~\ref{lem:500}.
This concludes the proof of Theorem~\ref{linear}.
\end{proof}

In the next proposition we deal with the differentiability of the solution~$\psi$
to~\eqref{EQ g} with respect to the parameter~$\xi$
(we recall Theorem~\ref{linear} for the existence and uniqueness of such solution).

For this, we denote by~$\mathcal T_\xi$ the operator that associates to
any~$g\in L^2(\R^n)$ with~$\|g\|_{\star,\xi}<+\infty$
the solution to~\eqref{EQ g}, that is
\begin{equation}\label{mathcal T}
{\mbox{$\psi:=\mathcal T_\xi[g]$ is the unique solution to~\eqref{EQ g} in~$Y_\star$,}}
\end{equation}
where~$Y_\star$ is given in~\eqref{Y star}.

We notice that, thanks to Theorem~\ref{linear},~$\mathcal T_\xi$ is a linear and continuous operator from~$Y_\star$ to~$Y_\star$ endowed with the
norm~$\|\cdot\|_{\star,\xi}$,
and we will write~$\mathcal T_\xi\in\mathcal L(Y_\star)$.

\begin{pro}\label{prop:der}
The map~$\xi\in\Omega_\epsilon\mapsto\mathcal T_\xi$ is
continuously differentiable. Moreover
there exists a positive constant~$C$ such that
\begin{equation}\label{cv45}
\left\|\frac{\partial\mathcal T_\xi[g]}{\partial\xi}\right\|_{\star,\xi} \le C\left(\|g\|_{\star,\xi} +\left\|\frac{\partial g}{\partial\xi}\right\|_{\star,\xi}
\right).\end{equation}
\end{pro}

\begin{proof}
First, let us prove \eqref{cv45} assuming the differentiability of $\xi\mapsto\mathcal T_\xi$.
Given~$\xi\in\Omega_\epsilon$,~$|t|<1$ with~$t\neq 0$ and a function~$f$, we denote by~$\xi_j^t:=\xi+te_j$, and by
$$ D^t_jf:=\frac{f(\xi_j^t)-f(\xi)}{t},$$
for any~$j=1,\ldots,n$.

Also, we set
\begin{equation}\label{fi tj}
\varphi^t_j:=D^t_j\psi {\mbox{ and }} d^t_{i,j}:=D^t_j c_i.
\end{equation}
Using the fact that~$\psi$ is a solution to~\eqref{EQ g}, we have
that~$\varphi^t_j$ solves
\begin{equation}\label{sdfyfyss2}
(-\Delta)^s\varphi^t_j+\varphi^t_j-pw_\xi^{p-1}\varphi^t_j= p (D^t_jw_\xi^{p-1})\psi-D^t_jg+\sum_{i=1}^n c_i\, D^t_jZ_i +\sum_{i=1}^n d^t_{i,j}\, Z_i {\mbox{ in }}\Omega_\epsilon.\end{equation}
Moreover, we have that~$\varphi^t_j\in H^s(\R^n)$ and~$\varphi^t_j=0$
outside~$\Omega_\epsilon$.

Now, for the fixed index~$j$, for any~$i\in\{1,\dots,n\}$ we define
\begin{equation}\label{sdfyfyss}
\lambda_i(\varphi^t_j):=\tilde\alpha^{-1} \int_{\R^n} \varphi^t_j Z_i\,dx,\end{equation}
where~$\tilde\alpha$ is defined in~\eqref{alpha:J3}, and
\begin{equation}\label{sdfyfyss3}
\tilde\varphi^t_j:= \varphi^t_j-\sum_{i=1}^n
\lambda_i(\varphi^t_j)\,\tilde Z_i,\end{equation}
where~$\tilde Z_i$ are the ones in Lemma~\ref{lem:orto}.
We remark that~$\varphi^t_j$ and~$\tilde Z_i$
vanish outside~$\Omega_\eps$ by construction. Hence~$\tilde\varphi^t_j$
vanishes outside~$\Omega_\eps$ as well.
Moreover,
\begin{eqnarray*}
\int_{\R^n} \tilde\varphi^t_j \,Z_k\,dx&=&
\int_{\R^n} \varphi^t_j\,Z_k\,dx -\sum_{i=1}^n
\lambda_i(\varphi^t_j)\,\int_{\R^n}\tilde Z_i \,Z_k\,dx
\\ &=& \int_{\R^n} \varphi^t_j\,Z_k\,dx -
\sum_{i=1}^n \lambda_i(\varphi^t_j)\,
\tilde\alpha\,\delta_{ik}
\\ &=& \int_{\R^n} \varphi^t_j\,Z_k\,dx -
\lambda_k(\varphi^t_j)\,
\tilde\alpha\\
&=&0,\end{eqnarray*}
thanks to Lemma~\ref{lem:orto} and~\eqref{sdfyfyss}. This yields
that
\begin{equation}\label{PPPsiPP}
\tilde\varphi^t_j\in \Psi.\end{equation}
By plugging~\eqref{sdfyfyss3}
into~\eqref{sdfyfyss2}, we obtain that
\begin{equation}\label{AEX1}
(-\Delta)^s \tilde\varphi^t_j+\tilde\varphi^t_j-pw_\xi^{p-1}\tilde\varphi^t_j
=\tilde g_j+\sum_{i=1}^n d^t_{i,j} Z_i,
\end{equation}
where
\begin{equation}\label{AEX2}\begin{split}
\tilde g_j\,&:=
-(-\Delta)^s \sum_{i=1}^n\lambda_i(\tilde\varphi^t_j)\,\tilde Z_i
-\sum_{i=1}^n\lambda_i(\tilde\varphi^t_j)\,\tilde Z_i\\
&\quad +pw_\xi^{p-1} \sum_{i=1}^n\lambda_i(\tilde\varphi^t_j)\,\tilde Z_i
\\&\quad +p(D^t_j w_\xi^{p-1})\psi-D^t_j g+ \sum_{i=1}^n c_i\,D^t_jZ_i.
\end{split}\end{equation}
{F}rom~\eqref{AEX1}, \eqref{PPPsiPP}
and Lemma~\ref{lem:500}, we obtain
that
\begin{equation}\label{AEX3}
\|\tilde\varphi^t_j\|_{\star,\xi}\le C\,\|\tilde g_j\|_{\star,\xi}.
\end{equation}
Now we observe that
\begin{equation}\label{AEX4}
\left\|\sum_{i=1}^n \lambda_i(\tilde\varphi^t_j)\,\tilde Z_i
\right\|_{\star,\xi}\le C\,\|g\|_{\star,\xi}.
\end{equation}
To prove this, we notice that the orthogonality
condition~$\psi\in\Psi$ implies that
$$\int_{\Omega_\epsilon}\varphi^t_j\,Z_k\,dx=
-\int_{\Omega_\epsilon}\psi\,D^t_j Z_k\,dx,$$
for any~$k\in\{1,\dots,n\}$.
Hence, recalling~\eqref{sdfyfyss3}, \eqref{PPPsiPP}
and Lemma~\ref{lem:orto},
\begin{eqnarray*}
-\int_{\Omega_\epsilon}\psi\,D^t_j Z_k\,dx &=&
\int_{\Omega_\epsilon}
\left(\tilde\varphi^t_j+
\sum_{i=1}^n \lambda_i(\tilde\varphi^t_j)\,\tilde Z_i\right)\,Z_k\,dx
\\ &=& \sum_{i=1}^n \lambda_i(\tilde\varphi^t_j)\,
\int_{\Omega_\epsilon} \tilde Z_i\,Z_k\,dx
\\ &=& \sum_{i=1}^n \lambda_i(\tilde\varphi^t_j)\,
\tilde\alpha \,\delta_{ik}
\\ &=&\lambda_k(\tilde\varphi^t_j)\,\tilde\alpha.
\end{eqnarray*}
Therefore
\begin{eqnarray*}
|\lambda_k(\tilde\varphi^t_j)| &=&|\tilde\alpha^{-1}|
\,\left|\int_{\Omega_\epsilon}\psi\,D^t_j Z_k
\,dx\right|\\ &\le&
|\tilde\alpha^{-1}|
\,\int_{\Omega_\epsilon}\rho_\xi^{-1}|\psi|\,\rho_\xi\,|D^t_j Z_k|\,dx
\\ &\le& |\tilde\alpha^{-1}|
\,\| \psi\|_{\star,\xi}\,
\int_{\R^n} \rho_\xi\,|D^t_j Z_k|\,dx
\\ &\le& C\, \| \psi\|_{\star,\xi},
\end{eqnarray*}
thanks to Lemma~\ref{lem:decay DZ}. Using this and Lemma~\ref{:decay Z},
and possibly renaming the constants,
we obtain that
\begin{eqnarray*}
&& \left\|
\sum_{i=1}^n \lambda_i(\tilde\varphi^t_j)\, \tilde Z_i
\right\|_{\star,\xi} \le
\sum_{i=1}^n |\lambda_i(\tilde\varphi^t_j)|\; \|\tilde Z_i
\|_{\star,\xi}\\
&&\qquad\le C\,\sum_{i=1}^n |\lambda_i(\tilde\varphi^t_j)|
\le C\,\| \psi\|_{\star,\xi}
.\end{eqnarray*}
On the other hand, by Lemma~\ref{lem:500}, we have that~$\|\psi\|_{\star,\xi}
\le C\,\|g\|_{\star,\xi}$, so the above estimate
implies~\eqref{AEX4}, as desired.

Now we claim that
\begin{equation}\label{AEX5}
\left\|(-\Delta)^s
\sum_{i=1}^n \lambda_i(\tilde\varphi^t_j)\, \tilde Z_i
\right\|_{\star,\xi}\le C\,\|g\|_{\star,\xi}.
\end{equation}
Indeed, $\tilde Z_i$ is compactly supported
in a neighborhood of~$\xi$, hence~$(-\Delta)^s \tilde Z_i$
decays like~$|x-\xi|^{-n-2s}$ at infinity. Accordingly,
$\|(-\Delta)^s \tilde Z_i\|_{\star,\xi}$ is finite,
and then we obtain
\begin{eqnarray*}
\left\|(-\Delta)^s
\sum_{i=1}^n \lambda_i(\tilde\varphi^t_j)\, \tilde Z_i
\right\|_{\star,\xi} &=&
\left\|
\sum_{i=1}^n \lambda_i(\tilde\varphi^t_j)\, (-\Delta)^s\tilde Z_i
\right\|_{\star,\xi}\\
&\le&
\sum_{i=1}^n |\lambda_i(\tilde\varphi^t_j)|\,
\left\|(-\Delta)^s\tilde Z_i\right\|_{\star,\xi}\\
&\le& C\, \sum_{i=1}^n |\lambda_i(\tilde\varphi^t_j)|
\\ &\le& C\,\|g\|_{\star,\xi},
\end{eqnarray*}
due to~\eqref{AEX4}, and this establishes~\eqref{AEX5}.

Now we claim that
\begin{equation}\label{AEX0}
|D^t_j w_\xi^{p-1}|\le C,
\end{equation}
with~$C$ independent of~$t$. Indeed
\begin{eqnarray*}
D^t_j w_\xi^{p-1}(x) &=& \frac1t \Big(
w^{p-1}(x-\xi-te_j)-w^{p-1}(x-\xi)\Big) \\
&=& \frac1t \int_0^t \frac{d}{d\tau}
w^{p-1}(x-\xi-\tau e_j)\,d\tau \\
&=& \frac{p-1}{t} \int_0^t w^{p-2}(x-\xi-\tau e_j) \frac{d}{d\tau}
w(x-\xi-\tau e_j)\,d\tau \\
&=& -\frac{p-1}{t} \int_0^t w^{p-2}(x-\xi-\tau e_j)
\nabla w(x-\xi-\tau e_j)\cdot e_j\,d\tau.
\end{eqnarray*}
Also, by formulas~(IV.2) and~(IV.6)
of~\cite{CMSJ}, we know that
\begin{equation}\label{d88du44uuuuqqqqq}
{\mbox{$w(x)$ is
bounded both from above and from below by a constant times }}\,\frac{1}{
1+|x|^{n+2s}}.\end{equation}
Thus, supposing without loss of generality that~$t>0$,
and recalling Lemma~\ref{:decay Z}, we have that
\begin{eqnarray*}
|D^t_j w_\xi^{p-1}(x)| &\le& \frac{p-1}{t}
\int_0^t w^{p-2}(x-\xi-\tau e_j)
|\nabla w(x-\xi-\tau e_j)|\,d\tau\\
&\le& \frac{C}{t}
\int_0^t \big(1+|x-\xi-\tau e_j|\big)^{-(p-2)(n+2s)}
\big(1+|x-\xi-\tau e_j|\big)^{-(n+2s)}\,d\tau \\
&=& \frac{C}{t}
\int_0^t \big(1+|x-\xi-\tau e_j|\big)^{-(p-1)(n+2s)}
\,d\tau\\
&\le& \frac{C}{t}
\int_0^t 1\,d\tau \\
&=& C,
\end{eqnarray*}
and this proves~\eqref{AEX0}.

{F}rom~\eqref{AEX0} and Lemma~\ref{lem:500}
we obtain that
\begin{equation}\label{AEX6}
\| (D^t_j w_\xi^{p-1})\psi\|_{\star,\xi}\le
C\,\|\psi\|_{\star,\xi}\le
C\,\|g\|_{\star,\xi}.
\end{equation}

Now we use Lemmata~\ref{lem:decay DZ},
\ref{lem:ci} and~\ref{lem:500} to see that
\begin{equation}\label{AEX7}\begin{split}
\left\| \sum_{i=1}^n c_i D^t_j Z_i\right\|_{\star,\xi}\,&\le
\sum_{i=1}^n |c_i|\, \left\|D^t_j Z_i\right\|_{\star,\xi}
\\ &\le C\,\sum_{i=1}^n |c_i|\\
&= C\,\sum_{i=1}^n\left|
\frac{1}{\alpha}\int_{\R^n}g\,Z_i\,dx + f_i\right| \\
&\le C\, \left( \|g\|_{L^2(\R^n)} +\sum_{i=1}^n |f_i|
\right) \\
&\le C\, \left( \|\psi\|_{L^2(\R^n)}+
\|g\|_{L^2(\R^n)} \right) \\
&\le C\, \left( \|\psi\|_{\star,\xi}+
\|g\|_{\star,\xi}\right)\\
&\le C\,\|g\|_{\star,\xi}.\end{split}
\end{equation}
By plugging~\eqref{AEX4}, \eqref{AEX5}, \eqref{AEX6}
and~\eqref{AEX7} into~\eqref{AEX2}
we obtain that
$$ \|\tilde g_j\|_{\star,\xi}\le C\,\left(\|g\|_{\star,\xi}+
\|D^t_j g\|_{\star,\xi}\right) .$$
Therefore, by~\eqref{AEX3},
$$ \|\tilde \varphi^t_j\|_{\star,\xi}\le
C\,\left(\|g\|_{\star,\xi}+
\|D^t_j g\|_{\star,\xi}\right).$$
This and~\eqref{AEX4} imply that
\begin{eqnarray*}
\|\varphi^t_j\|_{\star,\xi} &\le& \|\tilde \varphi^t_j\|_{\star,\xi}
+\left\|
\sum_{i=1}^n \lambda_i(\tilde\varphi^t_j)\, \tilde Z_i\right\|_{\star,\xi}
\\ &\le& C\,\left(\|g\|_{\star,\xi}+
\|D^t_j g\|_{\star,\xi}\right).\end{eqnarray*}
Hence we send~$t\searrow 0$ and we obtain
$$ \left\|\frac{\partial\psi}{\partial\xi}\right\|_{\star,\xi}\le C\,
\left(\|g\|_{\star,\xi} +\left\|\frac{\partial g}{\partial\xi}\right\|_{\star,\xi} \right),$$
which implies that
$$ \left\|\frac{\partial\mathcal T_\xi[g]}{\partial\xi}\right\|_{\star,\xi}\le C\,\left(\|g\|_{\star,\xi} +\left\|\frac{\partial g}{\partial\xi}\right\|_{\star,\xi}
\right).
$$
Using the previous computation and the implicit function theorem, a standard argument shows that $\xi\mapsto \mathcal T_\xi$ is continuously differentiable (see e.g. Section~$2.2.1$ in~\cite{AM}, and in particular Lemma~$2.11$ there, or~\cite{DDW} below formula~$(4.20)$).
\end{proof}

\subsection{The nonlinear projected problem}
In this subsection we solve the nonlinear projected problem
\begin{equation}\label{EQ nonlinear}
\left\{
\begin{matrix}
(-\Delta)^s \psi +\psi -pw_\xi^{p-1}\psi =E(\psi)+ N(\psi)+
\displaystyle\sum_{i=1}^n c_i\,Z_i & {\mbox{ in $\Omega_\epsilon$,}}\\
\psi =0 & {\mbox{ in }}\R^n\setminus\Omega_\epsilon,\\
\displaystyle\int_{\Omega_\epsilon}\psi\,Z_i\,dx=0 & {\mbox{ for any }}i=1,\ldots,n,
\end{matrix}
\right.
\end{equation}
where~$E(\psi)$ and~$N(\psi)$ are given in~\eqref{N psi}.

\begin{thm}\label{th:psi}
If~$\epsilon>0$ is sufficiently small, there exists a
unique~$\psi\in H^{s}(\R^n)$ solution to~\eqref{EQ nonlinear} for suitable
real coefficients~$c_i$, for~$i=1,\ldots,n$, and such that
there exists a positive constant~$C$ such that
\begin{equation}\label{small norm}
\|\psi\|_{\star,\xi}\le C\,\epsilon^{n+2s}.
\end{equation}
\end{thm}

Before proving Theorem~\ref{th:psi}, we show some estimates for the error terms~$E(\psi)$ and~$N(\psi)$.

\begin{lem}
There exists a positive constant~$C$ such that
\begin{equation}\label{u meno w}
|\bar u_\xi -w_\xi|\le C\, \epsilon^{n+2s}.
\end{equation}
\end{lem}

\begin{proof}
To prove~\eqref{u meno w}, we define~$\eta_\xi:=\bar u_\xi-w_\xi$,
and we observe that~$\eta_\xi$ satisfies
\begin{equation}\label{dentrooooo}
\left\{
\begin{matrix}
(-\Delta)^s \eta_\xi +\eta_\xi =0 & {\mbox{ in $\Omega_\epsilon$,}}\\
\eta_\xi = -w_\xi & {\mbox{ in }}\R^n\setminus\Omega_\epsilon,
\end{matrix}
\right.
\end{equation}
due to~\eqref{EQ w} and~\eqref{EQ bar u}.

We have
$$ |\eta_\xi|=|w_\xi|\le C\eps^{n+2s} {\mbox{ outside }}\Omega_\eps,$$
thanks to~\eqref{decay w}.
Hence, this together with~\eqref{dentrooooo} and the maximum principle give
$$ |\eta_\xi|\le C\eps^{n+2s} {\mbox{ in }}\R^n,$$
which implies the thesis (recall the definition of~$\eta_\xi$).
\end{proof}

Moreover, we can prove the following:
\begin{lem}
There exists a positive constant~$C$ such that
\begin{equation}\label{der u meno w}
\left|\frac{\partial\bar u_\xi}{\partial\xi}-
\frac{\partial w_\xi}{\partial\xi}\right|\le C\, \epsilon^{\nu_1},
\end{equation}
with~$\nu_1:=\min\{(n+2s+1),\,p(n+2s) \}$.
\end{lem}

\begin{proof}
We set~$\eta_\xi:=\bar u_\xi-w_\xi$. From~\eqref{EQ w} and~\eqref{EQ bar u}, we have that~$\eta_\xi$ solves
$$ (-\Delta)^s\eta_\xi +\eta_\xi=0 {\mbox{ in }}\Omega_\eps.$$
Therefore, the derivative of~$\eta_\xi$ with respect to~$\xi$ satisfies
\begin{equation}\label{eq der eta}
(-\Delta)^s\frac{\partial\eta_\xi}{\partial\xi} +\frac{\partial\eta_\xi}{\partial\xi} =0 {\mbox{ in }}\Omega_\eps.
\end{equation}

Moreover, since~$\bar u_\xi=0$ outside~$\Omega_\eps$, we have that
$$ \eta_\xi=\bar u_\xi-w_\xi=-w_\xi {\mbox{ in }}\R^n\setminus\Omega_\eps,$$
which implies
$$ \frac{\partial\eta_\xi}{\partial\xi}=-\frac{\partial w_\xi}{\partial\xi}=\frac{\partial w_\xi}{\partial x} {\mbox{ in }}\R^n\setminus\Omega_\eps.$$
Therefore, from Lemma~\ref{:decay Z} (recall also~\eqref{Zj}), we have that
$$ \left|\frac{\partial\eta_\xi}{\partial\xi}\right|\le C\epsilon^{\nu_1}
{\mbox{ outside }}\Omega_\eps.$$
From this,~\eqref{eq der eta} and the maximum principle we deduce that
$$ \left|\frac{\partial\eta_\xi}{\partial\xi}\right|\le C\epsilon^{\nu_1}
{\mbox{ in }}\R^n,$$
which gives the desired estimate (recall the definition of~$\eta_\xi$).
\end{proof}

In the next lemma we estimate the~$\star$-norm of the error term~$E(\psi)$.
For this, we recall the definition of the space~$Y_\star$ given in~\eqref{Y star}.

\begin{lem}\label{lem E}
Let~$\psi\in Y_\star$ with~$\|\psi\|_{\star,\xi}\le1$.
Then, there exists a positive constant~$\bar C$ such that
$$ \|E(\psi)\|_{\star,\xi}\le \bar C\,\epsilon^{n+2s}.$$
\end{lem}

\begin{proof}
Using~\eqref{u meno w} and Lemma~$2.1$ in~\cite{DMM} with~$a:=w_\xi+\psi$ and~$b:=\bar u_\xi-w_\xi$, we obtain that
\begin{eqnarray*}
|E(\psi)| &=& |(\bar u_\xi-w_\xi +w_\xi +\psi)^p-(w_\xi+\psi)^p|\\
&\le & C_1(w_\xi+\psi)^{p-1}|\bar u_\xi-w_\xi|\\
&\le & C_2\, \epsilon^{n+2s}(w_\xi+\psi)^{p-1}.
\end{eqnarray*}
Hence, since~$\|w_\xi\|_{\star,\xi}$ and~$\|\psi\|_{\star,\xi}$ are bounded, we have
$$ \|E(\psi)\|_{\star,\xi} \le C_3 \epsilon^{n+2s},$$
which gives the desired result.
\end{proof}

Now, we give a bound for the~$\star$-norm of the error term~$N(\psi)$.

\begin{lem}\label{lemma:N}
Let~$\psi\in Y_\star$. Then, there exists a positive constant~$C$ such that
$$ \|N(\psi)\|_{\star,\xi}\le C\left(\|\psi\|_{\star,\xi}^2+\|\psi\|_{\star,\xi}^p\right).$$
\end{lem}

\begin{proof}
We take~$\psi\in Y_\star$ and we estimate
\begin{eqnarray*}
|N(\psi)|&=&|(w_\xi+\psi)^p-w_\xi^p-pw_\xi^{p-1}\psi|\\
& \le & C\left(|\psi|^2+|\psi|^p\right),
\end{eqnarray*}
for some positive constant~$C$ (see, for instance, Corollary~$2.2$ in~\cite{DMM}, applied here with~$a:=w_\xi$ and~$b:=\psi$). Hence,
\begin{eqnarray*}
\rho_\xi^{-1}|N(\psi)|&\le & C\,\rho_\xi^{-1}\left(|\psi|^2+|\psi|^p\right)\\
&\le & C\left(\rho_\xi^{-2}|\psi|^2+\rho_\xi^{-p}|\psi|^p\right)\\
&\le & C\left( \|\psi\|^2_{\star,\xi} + \|\psi\|^p_{\star,\xi}\right),
\end{eqnarray*}
which implies the desired estimate.
\end{proof}

For further reference, we now recall
an estimate of elementary nature:

\begin{lem}
Fixed~$\kappa>0$, there exists a constant~$C_\kappa>0$
such that, for any~$a$, $b\in [0,\kappa]$ we have
\begin{equation}\label{E.1}
|a^{p-1} - b^{p-1}|\le C_\kappa |a-b|^{q},
\end{equation}
where
\begin{equation}\label{Eq}
q:=\min\{1,p-1\}.\end{equation}
\end{lem}

\begin{proof}
Fixed any~$\alpha\in(0,1)$, for any~$t>0$
we define
$$ h(t):=\frac{(t+1)^{\alpha}-1}{t^{\alpha}}.$$
Using de L'Hospital Rule we see that
$$ \lim_{t\searrow0} h(t)=
\lim_{t\searrow0} \frac{t^{1-\alpha}}{(t+1)^{1-\alpha}}=0,$$
hence we can extend~$h$ to a continuous function on~$[0,+\infty)$
with~$h(0):=0$.
Moreover
$$ \lim_{t\rightarrow+\infty} h(t)=1,$$
hence there exists
\begin{equation}\label{Z0}
M_0 := \sup_{t\in [0,+\infty)} h(t) <+\infty.\end{equation}
Now we prove~\eqref{E.1}.
For this, we may and do assume that~$a> b$.
If~$p\ge 2$, we have that
\begin{eqnarray*}
&& a^{p-1} - b^{p-1} =(p-1) \int_b^a \tau^{p-2}\,d\tau \le
(p-1) \, a^{p-2}\, (a-b)\\
&&\quad\le (p-1) \, \kappa^{p-2}\, (a-b),
\end{eqnarray*}
that is~\eqref{E.1} in this case. On the other hand,
if~$p\in(1,2)$ we take~$t:=(a/b)-1>0$ and~$\alpha:=p-1$, so
$$ M_0\ge
h(t) =
\frac{(a/b)^{p-1}-1}{((a/b)-1)^{p-1}} =
\frac{a^{p-1}-b^{p-1}}{(a-b)^{p-1}},$$
thanks to~\eqref{Z0}, and this establishes~\eqref{E.1} also
in this case.
\end{proof}

Now we are ready to complete the proof of
Theorem~\ref{th:psi}.

\begin{proof}[Proof of Theorem~\ref{th:psi}]
Recalling the definition of the operator~$\mathcal T_\xi$ in~\eqref{mathcal T},
we can write
$$ \psi=\mathcal T_\xi[E(\psi)+N(\psi)].$$

We will prove
Theorem~\ref{th:psi} by a contraction argument. To do this, we set
\begin{equation}\label{TTT}
\mathcal K_\xi(\psi):=\mathcal T_\xi[E(\psi)+N(\psi)].
\end{equation}
Moreover, we take a constant~$C_0>0$ and~$\epsilon>0$
small (we will specify the choice of~$C_0$ and~$\epsilon$ in~\eqref{choice fin}), and we define the set
$$ B:=\{\psi\in Y_\star {\mbox{ s.t }}\|\psi\|_{\star,\xi}\le C_0\,\epsilon^{n+2s}\},$$
where~$Y_\star$ is introduced in~\eqref{Y star}.

We claim that
\begin{equation}\label{contr arg}
{\mbox{$\mathcal K_\xi$ as in~\eqref{TTT} is a contraction mapping from~$B$ into itself with respect to the norm~$\|\cdot\|_{\star,\xi}$.}}
\end{equation}

First we prove that
\begin{equation}\label{set ok}
{\mbox{if~$\psi\in B$ then~$\mathcal K_\xi(\psi)\in B$.}}
\end{equation}
Indeed, if~$\psi\in B$, we have that
\begin{equation}\label{NNNNNNN}
\|N(\psi)\|_{\star,\xi}\le  C_1\left( \|\psi\|^2_{\star,\xi} + \|\psi\|^p_{\star,\xi}\right)
\end{equation}
thanks to Lemma~\ref{lemma:N}.

Now, thanks to~\eqref{est fin}, we have that
$$ \|\mathcal K_\xi(\psi)\|_{\star,\xi} =\|\mathcal T_\xi[E(\psi)+N(\psi)]\|_{\star,\xi} \le C\|E(\psi)+N(\psi)\|_{\star,\xi}. $$
This, Lemma~\ref{lem E} and~\eqref{NNNNNNN} give that
\begin{equation}\begin{split}\label{qqqqqqqqqqqq}
\|\mathcal K_\xi(\psi)\|_{\star,\xi} &\le \, C\left(\|E(\psi)\|_{\star,\xi} +\|N(\psi)\|_{\star,\xi}\right)\\
&\le \, C\left(\|E(\psi)\|_{\star,\xi} +C_1\left( \|\psi\|^2_{\star,\xi} + \|\psi\|^p_{\star,\xi}\right)\right)\\
&\le \, C \left( \bar C\, \epsilon^{n+2s} + C_1\,C_0^2\,\epsilon^{2(n+2s)} + C_1\, C_0^p\,\epsilon^{p(n+2s)} \right)\\
&= \, C_0\,\epsilon^{n+2s}\left(\frac{C\,\bar C}{C_0}+ C\,C_1\,C_0\epsilon^{n+2s}+ C\,C_1\, C_0^{p-1}\epsilon^{(p-1)(n+2s)}\right),
\end{split}\end{equation}
since~$\psi\in B$. We assume
\begin{equation}\label{choiceeeee}
C_0>2C\bar C
\end{equation}
and
\begin{equation}\label{choice2}
\epsilon<\epsilon_1:=\left\{\begin{matrix}
\left( \frac{1}{2C\,C_1(C_0+C_0^{p-1})}\right)^{1/(n+2s)} &{\mbox{ if }}p\ge 2,\\
\left( \frac{1}{2C\,C_1(C_0+C_0^{p-1})} \right)^{1/(p-1)(n+2s)}& {\mbox{ if }} 1<p<2. \end{matrix}\right.
\end{equation}
With this choice of~$C_0$ and ~$\epsilon$,~\eqref{qqqqqqqqqqqq} implies that
$$ \|\mathcal K_\xi(\psi)\|_{\star,\xi}\le C_0\,\epsilon^{n+2s},$$
which proves~\eqref{set ok}.

Now, we take~$\psi_1,\psi_2\in B$. Then,
\begin{eqnarray*}
|N(\psi_1)-N(\psi_2)| &=& |(w_\xi+\psi_1)^p-(w_\xi+\psi_2)^p-pw_\xi^{p-1}\left(\psi_1-\psi_2\right)|\\
&\le & C_2\left(|\psi_1|+|\psi_2|+|\psi_1|^{p-1}+|\psi_2|^{p-1}\right)|\psi_1-\psi_2|.
\end{eqnarray*}
This and the fact that~$\psi_1,\psi_2\in B$ give that
\begin{equation}\begin{split}\label{diff N}
\|N(\psi_1)-N(\psi_2)\|_{\star,\xi} &\le \, C_2\left(\|\psi_1\|_{\star,\xi} +\|\psi_2\|_{\star,\xi} +\|\psi_1\|_{\star,\xi}^{p-1}+\|\psi_2\|_{\star,\xi}^{p-1}\right)\|\psi_1-\psi_2\|_{\star,\xi} \\
&\le \, C_2 \left( 2C_0\,\epsilon^{n+2s}+2C_0^{p-1}\,\epsilon^{(p-1)(n+2s)}\right)\|\psi_1-\psi_2\|_{\star,\xi}\\
&\le \, 2\,C_2 \left(C_0+C_0^{p-1}\right)\epsilon^{q(n+2s)}\|\psi_1-\psi_2\|_{\star,\xi},
\end{split}\end{equation}
where~$q$ is defined in~\eqref{Eq}.

We claim that
\begin{equation}\label{E.0}
|E(\psi_1)-E(\psi_2)|\le C |\bar u_\xi-w_\xi|^q\,|\psi_1-\psi_2|,\end{equation}
where~$q$ is given in~\eqref{Eq}.

Fixed~$x\in\Omega_\eps$, given~$\tau$ in a bounded subset
of~$\R$ we consider
the function
$$ e(\tau):= (\bar u_\xi(x)+\tau)^p-(w_\xi(x)+\tau)^p.$$
We have that
$$ |e'(\tau)|= p\Big|
(\bar u_\xi(x)+\tau)^{p-1}-(w_\xi(x)+\tau)^{p-1}
\Big| \le C |\bar u_\xi-w_\xi|^q,$$
where we used \eqref{E.1} with~$a:=\bar u_\xi(x)+\tau$
and~$b:=w_\xi(x)+\tau$. This gives that
\begin{equation}\label{E.2}
|e(\tau_1)-e(\tau_2)|\le C |\bar u_\xi-w_\xi|^q\,|\tau_1-\tau_2|.\end{equation}
Now we take~$\tau_1:=\psi_1(x)$ and~$\tau_2:=\psi_2(x)$:
we remark that $\tau_1$ and~$\tau_2$ range in a bounded set
by our definition of~$B$ and that~$e(\tau_i)=E(\psi_i)$.
Thus~\eqref{E.0} follows from~\eqref{E.2}

Hence, from~\eqref{E.0} and~\eqref{u meno w}, we obtain that
$$ \|E(\psi_1)-E(\psi_2)\|_{\star,\xi}\le \tilde C \eps^{q(n+2s)}\|\psi_1-\psi_2\|_{\star,\xi}.$$
This,~\eqref{diff N} and~\eqref{est fin} give that
\begin{equation}\begin{split}\label{doppio}
&\|\mathcal K_\xi(\psi_1)-\mathcal K_\xi(\psi_2)\|_{\star,\xi} \\
&\quad \le \,
C\left(\|E(\psi_1)-E(\psi_2)\|_{\star,\xi} +\|N(\psi_1)-N(\psi_2)\|_{\star,\xi}\right) \\
&\quad \le \, C\left( 2\,C_2 \left(C_0+C_0^{p-1}\right)\epsilon^{q(n+2s)}+ \tilde C \eps^{q(n+2s)}\right)\|\psi_1-\psi_2\|_{\star,\xi}.
\end{split}\end{equation}
Now, we denote by
\begin{equation*}
\epsilon_2:=\left(\frac{1}{C(2\,C_2(C_0+C_0^{p-1})+\tilde C)}\right)^{1/q(n+2s)}.
\end{equation*}
Therefore, recalling also~\eqref{choiceeeee} and~\eqref{choice2},
we obtain that if
\begin{equation}\label{choice fin}
C_0>2C\bar C {\mbox{ and }}\epsilon<\min\{\epsilon_1,\epsilon_2\}
\end{equation}
then, from~\eqref{doppio} we have that
$$ \|\mathcal K_\xi(\psi_1)-\mathcal K_\xi(\psi_2)\|_{\star,\xi}<\|\psi_1-\psi_2\|_{\star,\xi},$$
which concludes the proof of~\eqref{contr arg}.

From~\eqref{contr arg}, we obtain the existence of a unique
solution to~\eqref{EQ nonlinear} which belongs to~$B$.
This shows~\eqref{small norm} and concludes the proof of
Theorem~\ref{th:psi}.
\end{proof}

For any~$\xi\in\Omega_\epsilon$, we say that
\begin{equation}\label{unique}
{\mbox{$\Psi(\xi)$ is the unique solution to~\eqref{EQ nonlinear}.}}
\end{equation}
Arguing as the proof of Proposition~$5.1$ in~\cite{DDW}, one can also prove the following:

\begin{pro}\label{prop der psi}
The map~$\xi\mapsto\Psi(\xi)$ is of class~$C^1$, and
$$ \left\|\frac{\partial\Psi(\xi)}{\partial\xi}\right\|_{\star,\xi}\le C\left(\|E(\Psi(\xi))\|_{\star,\xi} +\left\|\frac{\partial E(\Psi(\xi))}{\partial\xi}\right\|_{\star,\xi}\right),$$
for some constant~$C>0$.
\end{pro}

\subsection{Derivative estimates}

Here we deal with the derivatives of the solution~$\psi=\Psi(\xi)$
to~\eqref{EQ nonlinear} with respect to~$\xi$. This will
also imply derivative estimates for the error term $\xi\mapsto
E(\Psi(\xi))$.

We first show the following
\begin{lem}\label{lem:aggiunto}
Let~$\psi\in\Psi$ be a solution\footnote{We remark that a solution that
fulfils the assumptions of Lemma~\ref{lem:aggiunto}
is provided by Theorem~\ref{th:psi}, as long as~$\epsilon$ is sufficiently small.}
to~\eqref{EQ nonlinear}, with~$\|\psi\|_{\star,\xi}\le C\eps^{n+2s}$.
Then, there exist positive constants~$C$ and~$\gamma$ such that
$$ \left\|\frac{\partial E(\psi)}{\partial\xi}\right\|_{\star,\xi}\le 
C\left(\eps^{q(n+2s)}\left\|\frac{\partial\psi}{\partial\xi}\right\|_{\star,\xi}
+\epsilon^\gamma\right),$$
where $q$ is defined in \eqref{Eq}.
\end{lem}

\begin{proof}
First of all, we observe that, thanks to Proposition~\ref{prop:der}
(applied here with~$g:=-\left(E(\psi)+N(\psi)\right)$),
the function~$\frac{\partial\psi}{\partial\xi}$ is
well defined.

We make the following computations: from~\eqref{N psi} we have that
\begin{eqnarray*}
\frac{\partial E(\psi)}{\partial\xi} &=&
p(\bar u_\xi+\psi)^{p-1}\,\left( \frac{\partial \bar u_\xi}{\partial \xi}
+\frac{\partial \psi}{\partial \xi}\right)-
p(w_\xi+\psi)^{p-1}\,\left( \frac{\partial w_\xi}{\partial \xi}
+\frac{\partial \psi}{\partial \xi}\right)\\
&=& p\frac{\partial\psi}{\partial\xi}\,\left[ (\bar u_\xi+\psi)^{p-1}-
(w_\xi+\psi)^{p-1}
\right]+p(\bar u_\xi+\psi)^{p-1}\frac{\partial \bar u_\xi}{\partial\xi}
-p(w_\xi+\psi)^{p-1}\frac{\partial w_\xi}{\partial\xi}\\
&=&
p\frac{\partial\psi}{\partial\xi}\,\left[ (\bar u_\xi+\psi)^{p-1}-
(w_\xi+\psi)^{p-1}
\right]+p(\bar u_\xi+\psi)^{p-1}\left(
\frac{\partial \bar u_\xi}{\partial\xi}-
\frac{\partial w_\xi}{\partial\xi}\right)
\\ &&\quad+p\,\Big[
(\bar u_\xi+\psi)^{p-1}-(w_\xi+\psi)^{p-1}
\Big]\frac{\partial w_\xi}{\partial\xi}.
\end{eqnarray*}
Thus, recalling \eqref{E.1},
\eqref{u meno w} and~\eqref{der u meno w}, we infer that
\begin{equation}\label{s9ffjkkjf}
\begin{split}
\left|\frac{\partial E(\psi)}{\partial\xi}\right|\,&\leq
Cp\left|\frac{\partial\psi}{\partial\xi}\right|\,
|\bar u_\xi-w_\xi|^q
+p\,\big(|\bar u_\xi|+|\psi|\big)^{p-1}\,\left|
\frac{\partial \bar u_\xi}{\partial\xi}-
\frac{\partial w_\xi}{\partial\xi}\right|
\\ &\quad+Cp\,|\bar u_\xi-w_\xi|^q
\,\left|\frac{\partial w_\xi}{\partial\xi}\right| \\
&\leq
C\left|\frac{\partial\psi}{\partial\xi}\right|\,
\eps^{q(n+2s)}
+C\,\big(|\bar u_\xi|+|\psi|\big)^{p-1}\,\eps^{\nu_1}
\\ &\quad+C\eps^{q(n+2s)}
\,\left|\frac{\partial w_\xi}{\partial\xi}\right|
\end{split}\end{equation}
for some $C>0$. Now, we claim that
\begin{equation}\begin{split}\label{121bis}
&\sup_{x\in\R^n}(1+|x-\xi|)^\mu\, |\bar u_\xi(x)|^{p-1}\eps^{\nu_1}\le C\eps^{\gamma}\\
{\mbox{and }} \
&\sup_{x\in\R^n}(1+|x-\xi|)^\mu\, |\psi(x)|^{p-1}\eps^{\nu_1}\le C\eps^{\gamma},
\end{split}\end{equation}
for suitable~$C>0$ and~$\gamma>0$. Let us prove the first inequality in~\eqref{121bis}. For this, we use that $\bar u_\xi$ vanishes outside $\Omega_\eps$, together with~\eqref{u meno w} and~\eqref{decay w}, to see that
\begin{equation}\begin{split}\label{8237585fhv}
&\sup_{x\in\R^n}(1+|x-\xi|)^\mu\, |\bar u_\xi(x)|^{p-1}\eps^{\nu_1}\\
=&\,\sup_{x\in\Omega_\eps}(1+|x-\xi|)^\mu\, |\bar u_\xi(x)|^{p-1}\eps^{\nu_1}\\
\le &\, \sup_{x\in\Omega_\eps}(1+|x-\xi|)^\mu\, \eps^{(p-1)(n+2s)}\eps^{\nu_1} + \sup_{x\in\Omega_\eps}(1+|x-\xi|)^\mu\, |w_\xi(x)|^{p-1}\eps^{\nu_1}\\
\le &\, C\,\eps^{-\mu}\, \eps^{(p-1)(n+2s)}\eps^{\nu_1}+
\sup_{x\in\Omega_\eps}(1+|x-\xi|)^{\mu(p-1)}\, |w_\xi(x)|^{p-1}(1+|x-\xi|)^{\mu(2-p)}\eps^{\nu_1}\\
\le &\, C\,\eps^{-\mu}\, \eps^{(p-1)(n+2s)}\eps^{\nu_1}+
\|w_\xi\|_{\star,\xi}^{p-1}\eps^{-\mu(2-p)_+}\eps^{\nu_1}\\
\le &\, C\,\eps^{-\mu}\, \eps^{(p-1)(n+2s)}\eps^{\nu_1}+C\,\eps^{-\mu(2-p)_+}\eps^{\nu_1}.
\end{split}\end{equation}
Now we observe that
\begin{equation}\begin{split}\label{gfhw295ggrgr}
&-\mu+(p-1)(n+2s)+\nu_1 \\
=&\, \min\{-\mu+(p-1)(n+2s)+n+2s+1, -\mu+(p-1)(n+2s)+p(n+2s)\} \\
>&\, \min\{-(n+2s)+(p-1)(n+2s)+n+2s+1, -(n+2s)+(p-1)(n+2s)+p(n+2s)\}\\
=&\,\min\{(p-1)(n+2s)+1, (2p-2)(n+2s)\}>0.
\end{split}\end{equation}
Moreover, if~$p\ge 2$, then
$$ -\mu(2-p)_+ +\nu_1=\nu_1>0, $$
while if~$1<p<2$, then
\begin{equation*}\begin{split}
&-\mu(2-p)_+ +\nu_1 \\
=&\, \min\{-\mu(2-p)+n+2s+1, -\mu(2-p)+p(n+2s)\} \\
>&\, \min\{-(n+2s)(2-p)+n+2s+1, -(n+2s)(2-p)+p(n+2s)\}\\
=&\,\min\{(p-1)(n+2s)+1, (2p-2)(n+2s)\}>0.
\end{split}\end{equation*}
Using this and~\eqref{gfhw295ggrgr} into~\eqref{8237585fhv} we obtain the first
formula in~\eqref{121bis}. Now, we focus on the second inequality: from the
assumptions on~$\psi$ we have
\begin{equation}\begin{split}\label{13gfdgdf}
&\sup_{x\in\R^n}(1+|x-\xi|)^\mu\, |\psi(x)|^{p-1}\eps^{\nu_1} \\
=&\, \sup_{x\in\Omega_\eps}(1+|x-\xi|)^\mu\, |\psi(x)|^{p-1}\eps^{\nu_1}\\
=&\, \sup_{x\in\Omega_\eps}(1+|x-\xi|)^{\mu(p-1)}\, |\psi(x)|^{p-1}\,(1+|x-\xi|)^{\mu(2-p)}\,\eps^{\nu_1}\\
\le &\, \|\psi\|^{p-1}_{\star,\xi}\,\eps^{-\mu(2-p)_+}\,\eps^{\nu_1}\\
\le &\, C\,\eps^{(p-1)(n+2s)}\,\eps^{-\mu(2-p)_+}\,\eps^{\nu_1}.
\end{split}\end{equation}
If~$p\ge2$ we get the second inequality
in~\eqref{121bis}, as
desired, hence we focus on the case~$1<p<2$.
For this, we notice that
\begin{equation*}\begin{split}
&(p-1)(n+2s)-\mu(2-p)_++\nu_1 \\
=&\, \min\{(p-1)(n+2s)-\mu(2-p)+n+2s+1, (p-1)(n+2s)-\mu(2-p)+p(n+2s)\} \\
>&\, \min\{(p-1)(n+2s)-(2-p)(n+2s)+n+2s+1, (p-1)(n+2s)-(2-p)(n+2s)+p(n+2s)\}\\
=&\,\min\{(2p-2)(n+2s)+1, (3p-3)(n+2s)\}>0,
\end{split}\end{equation*}
and this, together with~\eqref{13gfdgdf}, implies the second inequality
in~\eqref{121bis} also in this case.
Hence the proof of~\eqref{121bis} is finished.

Exploiting~\eqref{121bis} and
Lemma~\ref{:decay Z}, we infer from~\eqref{s9ffjkkjf} that
\begin{equation}
\label{s88s8sssa8s}
\left\|\frac{\partial E(\psi)}{\partial\xi}\right\|_{\star,\xi}\leq
C\eps^{q(n+2s)}\left\|\frac{\partial\psi}{\partial\xi}\right\|_{\star,\xi}
+C\eps^{\gamma},
\end{equation}
for suitable~$C>0$ and~$\gamma>0$, and this concludes the proof of Lemma~\ref{lem:aggiunto},
up to renaming the constants.
\end{proof}

\begin{lem}\label{d9goioasdss}
Let~$\psi\in\Psi$ be a solution to~\eqref{EQ nonlinear}, with~$\|\psi\|_{\star,\xi}\le C\eps^{n+2s}$. Then, there
exists a positive constant~$C$ such that
$$ \left\|\frac{\partial\psi}{\partial\xi}\right\|_{\star,\xi}\le C.$$
\end{lem}

\begin{proof}
We observe that, thanks to Proposition~\ref{prop:der}
(applied here with~$g:=-\left(E(\psi)+N(\psi)\right)$),
$$ \left\|\frac{\partial\psi}{\partial\xi}\right\|_{\star,\xi}\le
C\left(\|E(\psi)\|_{\star,\xi}+\|N(\psi)\|_{\star,\xi}+
\left\|\frac{\partial E(\psi)}{\partial\xi}\right\|_{\star,\xi}+
\left\|\frac{\partial N(\psi)}{\partial\xi}\right\|_{\star,\xi}\right).$$
Therefore, from Lemmata~\ref{lem E},~\ref{lemma:N} and~\ref{lem:aggiunto}, we obtain that
\begin{equation}\label{7s9sjjsjsjssss}
\left\|\frac{\partial\psi}{\partial\xi}\right\|_{\star,\xi}\le
C\left(1+\eps^{q(n+2s)}
\left\|\frac{\partial\psi}{\partial\xi}\right\|_{\star,\xi}+
\left\|\frac{\partial N(\psi)}{\partial\xi}\right\|_{\star,\xi}\right).
\end{equation}
Now we observe that, from~\eqref{N psi},
\begin{eqnarray*}
\frac{\partial N(\psi)}{\partial\xi} &=&
p(w_\xi+\psi)^{p-1}\left( \frac{\partial w_\xi}{\partial\xi}+
\frac{\partial \psi}{\partial\xi}\right)-pw_\xi^{p-1}
\frac{\partial w_\xi}{\partial\xi}-
p(p-1)w_\xi^{p-2}\frac{\partial w_\xi}{\partial\xi}\psi
-p w_\xi^{p-1}\frac{\partial \psi}{\partial\xi}
\\ &=&
p \Big[ (w_\xi+\psi)^{p-1}-w_\xi^{p-1}\Big]\,
\frac{\partial \psi}{\partial\xi}
+p \Big[ (w_\xi+\psi)^{p-1}-w_\xi^{p-1}\Big]\,
\frac{\partial w_\xi}{\partial\xi}
-p(p-1)w_\xi^{p-2}\frac{\partial w_\xi}{\partial\xi}\psi.
\end{eqnarray*}
As a consequence, using \eqref{E.1} once again,
\begin{equation}\label{df8gfis666www}
\left|\frac{\partial N(\psi)}{\partial\xi} \right|\le
C\, |\psi|^q\,\left[ \left|
\frac{\partial \psi}{\partial\xi}\right|
+\left|\frac{\partial w_\xi}{\partial\xi}\right|\right]
+C\,w_\xi^{p-2}\,\left|\frac{\partial w_\xi}{\partial\xi}\right|\,|\psi|.
\end{equation}
Now we claim that
\begin{equation}\label{s8s88ssssq123344}
w_\xi^{p-2}\,\left|\frac{\partial w_\xi}{\partial\xi}\right| \le C,
\end{equation}
for some $C>0$. When $p\ge2$, \eqref{s8s88ssssq123344}
follows from \eqref{decay w} and Lemma~\ref{:decay Z}, hence
we focus on the case~$p\in(1,2)$. In this case,
we take $\nu_1$ as in Lemma~\ref{:decay Z} and we
notice that
\begin{eqnarray*}
&& \tilde\nu:=\nu_1-(2-p)(n+2s)=
\min\{ n+2s+1+(p-2)(n+2s), \ (2p-2)(n+2s)\}\\ &&\qquad\qquad=
\min\{ (p-1)(n+2s)+1, \ (2p-2)(n+2s)\}>0.\end{eqnarray*}
Then we use \eqref{d88du44uuuuqqqqq}
and we obtain that
$$ w_\xi^{p-2}\,\left|\frac{\partial w_\xi}{\partial\xi}\right|\le C\,
|x-\xi|^{(2-p)(n+2s)}\,|x-\xi|^{-\nu_1}=C\,|x-\xi|^{-\tilde\nu}.$$
Since $w_\xi$ is positive and smooth in the vicinity of~$\xi$,
this proves~\eqref{s8s88ssssq123344}.

Now, using~\eqref{s8s88ssssq123344}
into~\eqref{df8gfis666www}, we obtain that
\begin{equation}\label{adasdadopoi}
\left|\frac{\partial N(\psi)}{\partial\xi} \right|\le
C\, |\psi|^q\,\left[ \left|
\frac{\partial \psi}{\partial\xi}\right|
+\left|\frac{\partial w_\xi}{\partial\xi}\right|\right]
+C\,|\psi|.
\end{equation}
We claim that
\begin{equation}\label{gheerrr}
\left\|\frac{\partial N(\psi)}{\partial\xi} \right\|_{\star,\xi}\le
C\, \|\psi\|_{\star,\xi}^q\,\left[ \left\|
\frac{\partial \psi}{\partial\xi}\right\|_{\star,\xi}
+\left\|\frac{\partial w_\xi}{\partial\xi}\right\|_{\star,\xi}\right]
+C\,\|\psi\|_{\star,\xi}.
\end{equation}
Indeed, the claim plainly follows from~\eqref{adasdadopoi}
if~$q=1$ (that is~$p\ge2$), hence
we focus on the case~$q=p-1$ (that is~$1<p<2$).
In this case, we observe that
\begin{eqnarray*}
&&(1+|x-\xi|)^{\mu}|\psi|^q\,\left[ \left|
\frac{\partial \psi}{\partial\xi}\right|
+\left|\frac{\partial w_\xi}{\partial\xi}\right|\right]\\
&=& (1+|x-\xi|)^{\mu\,q}|\psi|^q\,(1+|x-\xi|)^{\mu}\left[ \left|
\frac{\partial \psi}{\partial\xi}\right|
+\left|\frac{\partial w_\xi}{\partial\xi}\right|\right]\,(1+|x-\xi|)^{-\mu\,q}\\
& \le & C\|\psi\|^q_{\star,\xi}\,\left[ \left\|
\frac{\partial \psi}{\partial\xi}\right\|_{\star,\xi}
+\left\|\frac{\partial w_\xi}{\partial\xi}\right\|_{\star,\xi}\right],
\end{eqnarray*}
and this implies~\eqref{gheerrr} also in this case.

Hence, using our assumptions on~$\psi$, we deduce that
$$
\left\|\frac{\partial N(\psi)}{\partial\xi} \right\|_{\star,\xi}\le
C\eps^{q(n+2s)}\,\left[ \left\|
\frac{\partial \psi}{\partial\xi}\right\|_{\star,\xi}
+1\right]+C,
$$
up to renaming constants.
By inserting this into~\eqref{7s9sjjsjsjssss}
we conclude that
$$ \left\|\frac{\partial\psi}{\partial\xi}\right\|_{\star,\xi}\le
C+
\frac12\left\|\frac{\partial\psi}{\partial\xi}\right\|_{\star,\xi},$$
as long as $\eps$ is sufficiently small. By reabsorbing one term
into the left hand side,
we obtain the desired result.
\end{proof}

\begin{lem}\label{lem der E}
Let~$\psi\in\Psi$ be a solution to~\eqref{EQ nonlinear}, with~$\|\psi\|_{\star,\xi}\le C\eps^{n+2s}$. Then, there exist positive constants~$\tilde C$ and~$\gamma$ such that
$$ \left\|\frac{\partial E(\psi)}{\partial\xi}\right\|_{\star,\xi}\le \tilde C\,\epsilon^{\gamma}.$$
\end{lem}

\begin{proof}
The proof easily follows from Lemmata~\ref{lem:aggiunto} and~\ref{d9goioasdss},
up to renaming the constants.
\end{proof}

\subsection{The variational reduction}
We are looking for solutions to~\eqref{EQ eps} of the form~\eqref{ruybc},
that is, recalling also~\eqref{unique},
\begin{equation}\label{def u xi}
u_\xi= \bar u_\xi +\Psi(\xi).
\end{equation}

We observe that, thanks to~\eqref{EQ bar u} and~\eqref{EQ nonlinear}, the function~$u_\xi$ satisfies the equation
\begin{equation}\label{EQ final}
(-\Delta)^s u_\xi +u_\xi =u_\xi^p+\sum_{i=1}^nc_i\,Z_i {\mbox{ in }}\Omega_\epsilon.
\end{equation}
Notice that if~$c_i=0$ for any~$i=1,\ldots,n$ then we will have a solution to~\eqref{EQ eps}. Hence, aim of this subsection is to find a suitable point~$\xi\in\Omega_\epsilon$
such that all the coefficients~$c_i$,~$i=1,\ldots,n$, in~\eqref{EQ final} vanish.

In order to do this, we define the
functional~$J_\epsilon:\Omega_\epsilon\rightarrow\R$ as
\begin{equation}\label{func red}
J_\epsilon(\xi):=I_\epsilon(\bar u_\xi+\Psi(\xi))=I_\epsilon(u_\xi) {\mbox{ for any }}\xi\in\Omega_\epsilon,
\end{equation}
where~$I_\epsilon$ is introduced in~\eqref{energy functional}.
We have the following characterization:
\begin{lem}\label{lem:reduction}
If~$\epsilon>0$ is sufficiently small, the coefficients~$c_i$,~$i=1,\ldots,n$, in~\eqref{EQ final} are equal to zero
if and only if~$\xi$ satisfies the following condition
$$ \frac{\partial J_\epsilon}{\partial\xi}(\xi)=0.$$
\end{lem}

\begin{proof}
We make a preliminary observation.
We write~$\xi=(\xi_1,\ldots,\xi_n)$ and,
for any~$j=1,\ldots,n$, we take the derivative of~$u_\xi$ with respect to~$\xi_j$.

We observe that
\begin{equation}\label{199}
\frac{\partial u_\xi}{\partial\xi_j}=\frac{\partial\bar u_\xi}{\partial\xi_j}+\frac{\partial\Psi(\xi)}{\partial\xi_j}.
\end{equation}
Thanks to~\eqref{der u meno w}, we have that
\begin{equation}\label{299}
\frac{\partial\bar u_\xi}{\partial\xi_j}=
\frac{\partial w_\xi}{\partial\xi_j}+O(\epsilon^{\nu_1}).
\end{equation}
Moreover, from Proposition~\ref{prop der psi} and Lemmata~\ref{lem E}
and~\ref{lem der E}, we obtain that
\begin{equation}\label{399}
\frac{\partial\Psi(\xi)}{\partial\xi_j}=O(\epsilon^{\gamma}),
\end{equation}
where~$\gamma>0$.
Hence,~\eqref{199},~\eqref{299} and~\eqref{399} imply that
$$ \frac{\partial u_\xi}{\partial\xi_j}=\frac{\partial w_\xi}{\partial\xi_j}+O(\epsilon^{\gamma}),$$
which means, recalling~\eqref{Zj} and using the fact that~$\frac{\partial w_\xi}{\partial\xi_j}=-\frac{\partial w_\xi}{\partial x_j}$, that
\begin{equation}\label{uuu}
\frac{\partial u_\xi}{\partial\xi_j}=-Z_j +O(\epsilon^{\gamma}).
\end{equation}
In particular,
\begin{equation}\label{pre u bounded}
\int_{\Omega_\epsilon}Z_i\,
\frac{\partial u_\xi}{\partial\xi_j}\, dx=
\int_{\Omega_\epsilon}Z_i\,\Big( -Z_j +O(\epsilon^{\gamma})\Big)\, dx
=-\int_{\Omega_\epsilon}Z_i\,Z_j \,dx+O(\epsilon^{\gamma})
\end{equation}
and, from Lemma~\ref{:decay Z}, we deduce that
\begin{equation}\label{u bounded}
\left|\frac{\partial u_\xi}{\partial\xi_j}\right|\le C_1(|Z_j|+\epsilon^{\gamma})\le C_2.
\end{equation}
With this, we introduce the matrix~$M\in {\mbox{Mat}}(n\times n)$
whose entries are given by
\begin{equation}\label{D M} M_{ji}:=\int_{\Omega_\epsilon}Z_i\,
\frac{\partial u_\xi}{\partial\xi_j}\, dx.\end{equation}
We claim that
\begin{equation}\label{I M} {\mbox{the matrix $M$
is invertible.}}\end{equation}
To prove this, we use~\eqref{pre u bounded}, Corollary~\ref{cor:orto}
and the fact that~$\alpha>0$ (recall~\eqref{alfaj}): namely we compute
\begin{eqnarray*}
M_{ji}
&=&-\int_{\Omega_\epsilon}Z_i\,Z_j\,dx +O(\epsilon^{\gamma})
\\ &=&-\alpha\delta_{ij}+O(\epsilon^{\gamma}).
\end{eqnarray*}
This says that the matrix~$-\alpha^{-1}M$ is a perturbation
of the identity and therefore it is invertible for~$\epsilon$
sufficiently small, hence~\eqref{I M} readily follows.

Now, we multiply~\eqref{EQ final} by~$\frac{\partial u_\xi}{\partial\xi}$, obtaining that
$$  \left((-\Delta)^s u_\xi +u_\xi -u_\xi^p\right)\frac{\partial u_\xi}{\partial\xi}=\sum_{i=1}^nc_i\,Z_i\frac{\partial u_\xi}{\partial\xi} {\mbox{ in }}\Omega_\epsilon,$$
and therefore
$$ \left|\left((-\Delta)^s u_\xi +u_\xi -u_\xi^p\right)\frac{\partial u_\xi}{\partial\xi}\right|\le \sum_{i=1}^n|c_i|\,|Z_i|\,\left|\frac{\partial u_\xi}{\partial\xi}\right|.$$
This, together with~\eqref{u bounded} and Lemma~\ref{:decay Z}, implies that
the function~$\left((-\Delta)^s u_\xi +u_\xi -u_\xi^p\right)\frac{\partial u_\xi}{\partial\xi}$ is in~$L^\infty(\Omega_\eps)$, and so in~$L^1(\Omega_\eps)$
uniformly with respect to~$\xi$.

This allows us to compute the derivative of~$J_\eps$ with respect to~$\xi_j$
as follows:
\begin{equation}\begin{split}\label{33333}
\frac{\partial J_\epsilon}{\partial\xi_j}(\xi) &= \, \frac{\partial}{\partial\xi_j}I_\epsilon(u_\xi)\\
&=\,\frac{\partial}{\partial\xi_j}\left(\int_{\Omega_\epsilon}\frac12(-\Delta)^su_\xi\,u_\xi +\frac12u_\xi^2-\frac{1}{p+1}u_\xi^{p+1}\,dx\right)\\
&=\, \int_{\Omega_\epsilon} \frac12(-\Delta)^s \frac{\partial u_\xi}{\partial \xi_j}\,u_\xi +\frac12(-\Delta)^s u_\xi \frac{\partial u_\xi}{\partial\xi_j} + \, \frac{\partial u_\xi}{\partial\xi_j}\,u_\xi -u_\xi^{p}\,\frac{\partial u_\xi}{\partial\xi_j}\,dx\\
&=\, \int_{\Omega_\epsilon} \left( (-\Delta)^s u_\xi +u_\xi -u_\xi^{p}\right) \frac{\partial u_\xi}{\partial\xi_j}\,dx\\
&=\, \sum_{i=1}^nc_i\int_{\Omega_\epsilon}Z_i\,
\frac{\partial u_\xi}{\partial\xi_j}\, dx,
\end{split}\end{equation}
where we have used~\eqref{EQ final} in the last step.
Thus, recalling~\eqref{D M}, we can write
$$ \frac{\partial J_\epsilon}{\partial\xi_j}(\xi)=
\sum_{i=1}^n c_i M_{ji},$$
for any~$j\in\{1,\dots,n\}$, that is the vector~$
\frac{\partial J_\epsilon}{\partial\xi}(\xi):=\left(\frac{\partial J_\epsilon}{\partial\xi_1}(\xi),\dots,\frac{\partial J_\epsilon}{\partial\xi_n}(\xi)\right)$
is equal to the product between the matrix~$M$ and the vector~$c:=
(c_1,\dots,c_n)$. {F}rom~\eqref{I M} we obtain that~$\frac{\partial J_\epsilon}{\partial\xi}(\xi)$ is equal to zero if and only if~$c$
is equal to zero, as desired.
\end{proof}

Thanks to Lemma~\ref{lem:reduction}, the problem of finding a solution to~\eqref{EQ eps} reduces to the one of finding critical points of the functional
defined in~\eqref{func red}.To this end, we obtain an expansion of~$J_\epsilon$.

\begin{thm}\label{th:exp}
We have the following expansion of the functional~$J_\eps$:
$$ J_\epsilon(\xi)=I_\epsilon(\bar u_\xi)+o(\epsilon^{n+4s}).$$
\end{thm}

\begin{proof}
We know that
$$ J_\epsilon(\xi)=I_\epsilon(\bar u_\xi+\Psi(\xi)).$$
Hence, we can Taylor expand in the vicinity of~$\bar u_\xi$, thus obtaining
\begin{eqnarray*}
J_\epsilon(\xi)&=& I_\epsilon(\bar u_\xi)+ I'_\epsilon(\bar u_\xi)[\Psi(\xi)]+ I''(\bar u_\xi)[\Psi(\xi),\Psi(\xi)]+ O(|\Psi(\xi)|^3)\\
&=&  I_\epsilon(\bar u_\xi)+ \int_{\Omega_\epsilon}(-\Delta)^s\bar u_\xi\,\Psi(\xi)+\bar u_\xi\,\Psi(\xi)- \bar u_\xi^p\,\Psi(\xi)\, dx \\
&&\qquad + \int_{\Omega_\epsilon}(-\Delta)^s\Psi(\xi)\,\Psi(\xi)+\Psi^2(\xi)-p\bar u_\xi^{p-1}\Psi^2(\xi)\, dx+ O(|\Psi(\xi)|^3)\\
&=& I_\epsilon(\bar u_\xi)+ \int_{\Omega_\epsilon}\left((-\Delta)^s u_\xi + u_\xi - u_\xi^p\right)\Psi(\xi)\, dx \\
&&\qquad -\int_{\Omega_\epsilon}\left((-\Delta)^s(u_\xi-\bar u_\xi) + u_\xi-\bar u_\xi - u_\xi^p +\bar u_\xi^p\right)\Psi(\xi)\, dx\\
&&\qquad + \int_{\Omega_\epsilon}(-\Delta)^s\Psi(\xi)\,\Psi(\xi)+\Psi^2(\xi)-p\bar u_\xi^{p-1}\Psi^2(\xi)\, dx+ O(|\Psi(\xi)|^3).
\end{eqnarray*}
Therefore, using~\eqref{def u xi}, we have that
\begin{equation}\begin{split}\label{asqwrfnbgfn}
J_\epsilon(\xi)&=\, I_\epsilon(\bar u_\xi)+ \int_{\Omega_\epsilon}\left((-\Delta)^s u_\xi + u_\xi - u_\xi^p\right)\Psi(\xi)\, dx \\
&\qquad +\int_{\Omega_\epsilon}\left(u_\xi^p-\bar u_\xi^p-p\bar u_\xi^{p-1}\Psi(\xi)\right)\Psi(\xi)\, dx+ O(|\Psi(\xi)|^3).
\end{split}\end{equation}
We notice that
$$ \int_{\Omega_\epsilon}\left((-\Delta)^s u_\xi + u_\xi - u_\xi^p\right)\Psi(\xi)\, dx=0, $$
thanks to~\eqref{EQ final} and the fact that~$\Psi(\xi)$ is orthogonal in~$L^2(\Omega_\epsilon)$ to any function in the space~$\mathcal Z$.
Hence,~\eqref{asqwrfnbgfn} becomes
\begin{equation}\label{ssuduudu00}
J_\epsilon(\xi)= I_\epsilon(\bar u_\xi)+ \int_{\Omega_\epsilon}\left(u_\xi^p-\bar u_\xi^p-p\bar u_\xi^{p-1}\Psi(\xi)\right)\Psi(\xi)\, dx+ O(|\Psi(\xi)|^3).
\end{equation}
Now, we observe that
$$ |u_\xi^p-\bar u_\xi^p-p\bar u_\xi^{p-1}\Psi(\xi)| \le
|u_\xi^p-\bar u_\xi^p|+p\,|\bar u_\xi^{p-1}\Psi(\xi)|\le C\,
|\bar u_\xi^{p-1}\Psi(\xi)|,$$
for a positive constant~$C$, and so, using also~\eqref{u meno w}, we have
\begin{equation}\begin{split}\label{999999}
&\left| \int_{\Omega_\epsilon}\left(u_\xi^p-\bar u_\xi^p-p\bar u_\xi^{p-1}\Psi(\xi)\right)\Psi(\xi)\, dx\right| \\
&\qquad \le  C\int_{\Omega_\epsilon}|\bar u_\xi|^{p-1}|\Psi(\xi)|^2\,dx\\
& \qquad \le  C\|\Psi(\xi)\|_{\star,\xi}^2\int_{\Omega_\epsilon}|\bar u_\xi|^{p-1}\rho_\xi^2\,dx\\
& \qquad \le  C\|\Psi(\xi)\|_{\star,\xi}^2\int_{\Omega_\epsilon}|w_\xi+O(\epsilon^{n+2s})|^{p-1}\rho_\xi^2\,dx\\
& \qquad \le
C\|\Psi(\xi)\|_{\star,\xi}^2\int_{\Omega_\epsilon}|w_\xi|^{p-1}\rho_\xi^2\,dx
+ C\,\epsilon^{(p-1)(n+2s)}\|\Psi(\xi)\|_{\star,\xi}^2
\int_{\Omega_\epsilon}\rho_\xi^2\,dx.
\end{split}\end{equation}
Recalling the definition of~$\rho_\xi$ in~\eqref{rho} and the fact that~$\mu>n/2$,
we have that
\begin{equation}\label{conto int}
\int_{\Omega_\epsilon}\rho_\xi^2\,dx\le C_1,
\end{equation}
for a suitable constant~$C_1>0$.
Moreover, thanks to~\eqref{small norm} (recall also~\eqref{unique}), we obtain
$$ \epsilon^{(p-1)(n+2s)}\|\Psi(\xi)\|_{\star,\xi}^2
\le C_2\,\epsilon^{(p-1)(n+2s)}\,\epsilon^{2(n+2s)}=C_2\,\epsilon^{(p+1)(n+2s)}, $$
which, together with~\eqref{conto int}, says that
\begin{equation}\label{ufffff}
C\,\epsilon^{(p-1)(n+2s)}\|\Psi(\xi)\|_{\star,\xi}^2
\int_{\Omega_\epsilon}\rho_\xi^2\,dx=o(\epsilon^{n+4s}).
\end{equation}
Also, using~\eqref{decay w} we have that
$$ \int_{\Omega_\epsilon}|w_\xi|^{p-1}\rho_\xi^2\,dx\le
C_3 \int_{\Omega_\epsilon}\frac{1}{(1+|x-\xi|)^{(p-1)(n+2s)}}
\frac{1}{(1+|x-\xi|)^{2\mu}}\, dx\le C_4, $$
and so, using also~\eqref{small norm} we have that
$$  C\|\Psi(\xi)\|_{\star,\xi}^2\int_{\Omega_\epsilon}|w_\xi|^{p-1}
\rho_\xi^2\,dx=o(\epsilon^{n+4s}). $$
This,~\eqref{ssuduudu00}, \eqref{999999}
and \eqref{ufffff} give
the desired claim in Theorem~\ref{th:exp}.
\end{proof}

\section{Proof of Theorem \ref{TH}}\label{end:p}

In this section we complete the proof of Theorem~\ref{TH}.
For this, we notice that, thanks to Theorems~\ref{TH expansion} and~\ref{th:exp},
we have that, for any~$\xi\in\Omega_\eps$ with~$\dist(\xi,\partial\Omega_\eps)\ge\delta/\eps$ (for some~$\delta\in(0,1)$),
\begin{equation}\label{fin expa}
J_\eps(\xi)= I(w)+\frac12\mathcal H_\eps(\xi)+o(\eps^{n+4s}),
\end{equation}
where~$J_\eps$ and~$I$ are defined in~\eqref{func red} and~\eqref{entire} respectively
(see also~\eqref{unique}), where~$\mathcal H_\eps$ is given by~\eqref{def H cal}.

Also, we recall the definition of the set~$\Omega_{\eps,\delta}$ given in~\eqref{omega delta}, and we claim that~$J_\eps$
has an interior minimum, namely
\begin{equation}\label{minimo omega delta}
{\mbox{there exists~$\bar\xi\in\Omega_{\eps,\delta}$ such that~$J_\eps(\bar\xi)=
\displaystyle\min_{\xi\in\overline\Omega_{\eps,\delta}}J_\eps(\xi)$.}}
\end{equation}
For this, we observe that~$J_\eps$ is a continuous functional, and therefore
\begin{equation}\label{interno2}
{\mbox{$J_\eps$ admits a minimizer~$\bar\xi\in\overline\Omega_{\eps,\delta}$.}}
\end{equation}
We have that
\begin{equation}\label{interno}
\bar\xi\in\Omega_{\eps,\delta}.
\end{equation}
Indeed, suppose by contradiction that~$\bar\xi\in\partial\Omega_{\eps,\delta}$.
Then, from~\eqref{fin expa}, we have that
\begin{equation}\begin{split}\label{interno1}
J_\eps(\bar\xi)&=\, I(w)+\frac12\mathcal H_\eps(\bar\xi)+o(\eps^{n+4s})\\
& \ge \, I(w)+\frac12\min_{\partial\Omega_{\eps,\delta}}\mathcal H_\eps
+o(\eps^{n+4s}).
\end{split}\end{equation}
On the other hand, by Proposition~\ref{P 3.8},
we know that~$\mathcal H_\eps$ has a strict interior minimum:
more precisely, there
exists~$\xi_o\in\Omega_{\eps,\delta}$ such that
\begin{equation}\label{0-0--1}
{\mathcal{H}}_\eps(\xi_o)=
\min_{\Omega_{\eps,\delta}} {\mathcal{H}}_\eps\le
c_1\eps^{n+4s} \end{equation}
and
\begin{equation}\label{0-0--2}
\min_{\partial\Omega_{\eps,\delta}} {\mathcal{H}}_\eps
\ge c_2\,\left(\frac{\eps}\delta\right)^{n+4s},\end{equation}
for suitable~$c_1$, $c_2>0$.
Also, the minimality of~$\bar\xi$ and~\eqref{fin expa}
say that
\begin{eqnarray*}
J_\eps(\bar\xi)&=&\min_{\xi\in\overline\Omega_{\eps,\delta}}J_\eps(\xi)\\
&\le & J_\eps(\xi_o) \\
&= & I(w)+\frac12 \mathcal H_\eps(\xi_o) +o(\eps^{n+4s}).
\end{eqnarray*}
By comparing this with~\eqref{interno1} and using~\eqref{0-0--1}
and~\eqref{0-0--2} we obtain
\begin{eqnarray*}
\frac{c_2\eps^{n+4s}}{2\,\delta^{n+4s}}+o(\eps^{n+4s})
&\le&
\frac12\min_{\partial\Omega_{\eps,\delta}}\mathcal H_\eps
+o(\eps^{n+4s})
\\ &\le& J_\eps(\bar\xi)-I(w)
\\ &\le& \frac12 \mathcal H_\eps(\xi_o) +o(\eps^{n+4s})
\\ &\le& \frac{c_1\eps^{n+4s}}{2}+o(\eps^{n+4s}).
\end{eqnarray*}
So, a division by~$\eps^{n+4s}$ and a limit argument give that
$$ \frac{c_2}{2\,\delta^{n+4s}} \le
\frac{c_1}{2}.$$
This is a contradiction when~$\delta$ is sufficiently small,
thus~\eqref{interno} is proved.
Hence~\eqref{minimo omega delta} follows from~\eqref{interno2} and~\eqref{interno}.

{F}rom~\eqref{minimo omega delta}, since~$\Omega_{\eps,\delta}$ is open,
we conclude that
$$\frac{\partial J_\eps}{\partial\xi}(\bar\xi)=0.$$
Therefore,
from Lemma~\ref{lem:reduction} we obtain the existence of a solution
to~\eqref{EQ} that satisfies~\eqref{vicina} for~$\eps$ sufficiently small,
and this concludes the proof of Theorem~\ref{TH}.

\section*{Acknowledgements} 
This work has been supported by Fondecyt grants 1110181, 1130360 and Fondo Basal CMM-Chile, 
EPSRC grant EP/K024566/1 ``Monotonicity formula methods for nonlinear PDE'' and 
ERC grant 277749 ``EPSILON Elliptic Pde's and Symmetry of Interfaces and Layers for Odd
Nonlinearities''.

\appendix

\section{Some physical motivation}\label{S:aa}

Equation \eqref{EQ} is a particular case
of the fractional Schr\"odinger equation
\begin{equation}\label{SCH}
i\hbar\partial_t \psi= \hbar^{2s} (-\Delta)^s\psi+V\,\psi,
\end{equation}
in case the wave function $\psi$ is a standing wave (i.e. $\psi(x,t)=
U(x) e^{it/\hbar}$) and the potential $V$ is a suitable power
of the density function (i.e. $V=V(|\psi|)=-|\psi|^{p-1}$).
As usual, $\hbar$ is the Planck's constant (then we write $\eps:=\hbar$
in~\eqref{EQ})
and $\psi=\psi(x,t)$ is the quantum mechanical
probability amplitude for a given
particle (for simplicity, of unit mass)
to have position~$x$ at time~$t$ (the corresponding
probability density is~$|\psi|^2$).

In this setting our Theorem \ref{TH}
describes the confinement of a particle inside a given domain $\Omega$:
for small values of $\hbar$ the wave function
concentrates to a material particle well inside the domain.

Equation~\eqref{SCH} is now quite popular
(say, for instance, popular enough to
have its own page on Wikipedia, see~\cite{WI})
and it is based on the classical
Schr\"odinger equation (corresponding to the case~$s=1$)
in which the Brownian motion of the quantum
paths is replaced by a L\'evy flight.
We refer to~\cites{Las, Las-atom, Las-book} for a throughout physical
discussion and detailed motivation
of equation~\eqref{SCH} (see in particular formula~(18)
in~\cite{Las}), but we present here some very
sketchy heuristics about it.

The idea is that the evolution of the wave function $\psi(x,t)$
from its initial state $\psi_0(x):=\psi(x,0)$
is run by a quantum mechanics kernel (or amplitude) $K$
which produces the forthcoming values of the wave function
by integration with the initial state, i.e.
\begin{equation}\label{0.11}
\psi(x,t)=\int_{\R^n} dy\; K(x,y,t) \,\psi_0(y).
\end{equation}
The main assumption is that such amplitude $K(x,y,t)$
is modulated by an action functional~$S_t$ via the contributions of
all the possible paths~$\gamma$ that join~$x$ to~$y$ in time~$t$,
that is
\begin{equation}\label{0.12}
K(x,y,t)=\int_{{\mathcal{F}}(x,y,t)}\,d\gamma \; e^{-i S_t(\gamma)/\hbar}.
\end{equation}
The above integral denotes the Feynman path integral
over ``all possible histories of the system'', that is
over ``all possible'' continuous paths~$\gamma:[0,t]\rightarrow\R^n$
with~$\gamma(0)=y$ and~$\gamma(t)=x$, see~\cite{Fey1}.
We remark that such integral is indeed a functional integral,
that is the
domain of integration~${\mathcal{F}}(x,y,t)$
is not a region of a finite dimensional space,
but a space of functions.
The mathematical treatment of Feynman path integrals
is by no means trivial: as a matter of fact, the convergence
must rely on the highly oscillatory behavior of the system
which produces the necessary cancellations. In some cases,
a rigorous justification can be provided by the theory
of Wiener spaces, but a complete treatment of this topic is far beyond
the scopes of this appendix (see e.g. \cites{Cam1, Cam2} and~\cites{Gro, Alb}).

The next structural ansatz we take is that
the action functional~$S_t$
is the superposition of a (complex) diffusive operator~$H_0$ and
a potential term~$V$.

Though the diffusion and the potential
operate ``simultaneously'', with some approximation we may suppose that, at each tiny
time step, they operate just one at the time, interchanging
their action\footnote{In a sense, this is the
quantum mechanics version of the
Lie--Trotter product formula $$e^{A+B}=
\lim_{N\rightarrow+\infty} \big(e^{A/N} e^{B/N}\big)^N$$
for $A,B\in \,{\rm Mat}\,(n\times n)$.
The procedure of
disentangling mixed exponentials is indeed
crucial in quantum mechanics computations, see e.g.~\cite{Fey2}.
In our computation,
a more rigorous approximation scheme lies in explicitly
writing $S_t(\gamma)$ as an integral from $0$ to $t$
of the Lagrangian along the path $\gamma$, then one splits
the integral in $N$ time steps of size $t/N$ by supposing
that in each of these time steps the Lagrangian
is, approximatively, constant. One may also suppose
that the Lagrangian involved in the action is a classical
one, i.e. it is the sum of a kinetic term
and the potential $V$. Then the effect
of taking the integral over all the possible paths
averages out the kinetic part reducing it to
a diffusive operator. Since here we are not aiming
at a rigorous justification of all these delicate procedures
(such as infinite dimensional integrals, limit exchanges
and so on), for simplicity
we are just taking $H_0$ to be a diffusive operator from
the beginning. In this spirit, it is also convenient
to suppose that the potential is an operator, that is we identify $V$
with the operation of multiplying a function by $V$.}
at a very high frequency. Namely, we discretize a path $\gamma$
into~$N$ adjacent paths of time range~$t/N$, say
$\gamma_1,\dots,\gamma_N:[0,t/N]\rightarrow\R^n$, with~$\gamma_1(0)=y$
and~$\gamma_N(t/N)=x$, and we suppose that, along each~$\gamma_j$
the action reduces to the subsequent nonoverlapping
superpositions of diffusion and potential terms, according
to the formula
\begin{equation}\label{0.13}
e^{-i S_t(\gamma)/\hbar}=\lim_{N\rightarrow+\infty}
\big(e^{-itH_0/(\hbar N)} e^{-itV/(\hbar N)}\big)^N.
\end{equation}
Once more, we do not indulge into a rigorous mathematical discussion of
such a limit and we just plug~\eqref{0.12} and~\eqref{0.13}
into~\eqref{0.11}. We obtain
\begin{equation}\label{0.14}\begin{split}
\psi(x,t)\,&=\int_{\R^n} dy\; \int_{{\mathcal{F}}(x,y,t)}\,d\gamma \;
e^{-i S_t(\gamma)/\hbar}\,\psi_0(y)\\
&=\lim_{N\rightarrow+\infty}
\int_{\R^n} dy\; \int_{{\mathcal{F}}(x,y,t)}\,d\gamma \;
\big(e^{-itH_0/(\hbar N)} e^{-itV/(\hbar N)}\big)^N
\,\psi_0(y).
\end{split}\end{equation}
Therefore, if we formally\footnote{The disentangling procedure allows to take derivative
of the exponentials of the operators
``as they were commuting ones''. Namely,
by the Zassenhaus formula,
$$e^{t(A+B)}=e^{tA} e^{tB} e^{O(t^2)}=e^{tA} e^{tB}(1+O(t^2))$$
that in our case gives
$$ e^{-itH_0/(\hbar N)} e^{-itV/(\hbar N)}=
e^{-it(H_0+V)/(\hbar N)} \big(1+O(t^2/N^2)\big)$$
and so
$$ \big(e^{-itH_0/(\hbar N)} e^{-itV/(\hbar N)}\big)^N=
e^{-it(H_0+V)/\hbar} \big(1+O(t^2/N^2)\big)$$
Hence
\begin{eqnarray*}
&&\lim_{N\rightarrow+\infty}
\partial_t\big(e^{-itH_0/(\hbar N)} e^{-itV/(\hbar N)}\big)^N
\\ &=&\lim_{N\rightarrow+\infty} -\frac{i(H_0+V)}{\hbar}
e^{-it(H_0+V)/\hbar} \big(1+O(t^2/N^2)\big) +
O(t/N^2)\\
&=& -\frac{i(H_0+V)}{\hbar}.\end{eqnarray*}
Moreover, we point out that a couple of additional
approximations are likely to be hidden in the computation in~\eqref{6C6}.
Namely, first of all, we do not differentiate
the functional domain of the Feynman integral.
This is consistent with the ansatz that the set of the paths
joining two points at a macroscopic scale in time $t$
``does not vary much'' for small variations of $t$.
Furthermore, we replace the action of $H_0$ and $V$
along the infinitesimal paths with their effective action
after averaging, so that we take $(H_0+V)$
outside the integral.} apply the time derivative to~\eqref{0.14}
we obtain that
\begin{equation}\label{6C6}\begin{split}
& i\hbar\partial_t\psi(x,t)\\ =\,&
\lim_{N\rightarrow+\infty}
\int_{\R^n} dy\; \int_{{\mathcal{F}}(x,y,t)}\,d\gamma \;
N\,\left(\frac{H_0}{N}e^{-itH_0/(\hbar N)}
e^{-itV/(\hbar N)}+\frac{V}{N}e^{-itH_0/(\hbar N)}e^{-itV/(\hbar N)}\right)
\\ &\cdot\,
\big(e^{-itH_0/(\hbar N)}e^{-itV/(\hbar N)}\big)^{N-1}
\,\psi_0(y)
\\ =\,&
\lim_{N\rightarrow+\infty}
\int_{\R^n} dy\; \int_{{\mathcal{F}}(x,y,t)}\,d\gamma \;
\left(H_0 e^{-itH_0/(\hbar N)}e^{-itV/(\hbar N)}
+V e^{-itH_0/(\hbar N)}e^{-itV/(\hbar N)}\right)
\\ &\cdot\,
\big(e^{-itH_0/(\hbar N)}e^{-itV/(\hbar N)}\big)^{N-1}
\,\psi_0(y)
\\ =\,&\left(H_0 +V \right)\,
\int_{\R^n} dy\; \int_{{\mathcal{F}}(x,y,t)}\,d\gamma \;
e^{-i S_t(\gamma)/\hbar}
\,\psi_0(y)\\
=\,& \left(H_0 +V \right) \psi
\end{split}\end{equation}
by \eqref{0.11}, \eqref{0.12} and~\eqref{0.13}.
The classical
Schr\"odinger equation follows by taking~$H_0:=-\hbar^2\Delta$,
that is the Gaussian diffusive process,
while~\eqref{SCH} follows by taking~$H_0:=\hbar^{2s}(-\Delta)^s$,
that is the $2s$-stable diffusive process with polynomial tail.

Having given a brief justification of~\eqref{SCH},
we also recall that the fractional
Schr\"odinger case presents interesting differences with
respect to the classical one. For instance,
the energy of a particle of unit mass
is proportional to~$|p|^{2s}$
(instead of~$|p|^2$, see e.g. formula~(12)
in~\cite{Las}). Also
%%% (in spite of some
%%% controversies concerning quantum fractal dimensions,
%%% see e.g.~\cite{amir-azizi-hey-morris}),
the space/time scaling of the process gives that
the fractal dimension of
the L\'evy paths is~$2s$ (differently from the classical Brownian case in which
it is~$2$), see pages~300--301 of~\cite{Las}.

Now, for completeness, we discuss a nonlocal notion of
canonical quantization, together with the associated
Heisenberg Uncertainty Principle
(see e.g. pages~17--28 of~\cite{Giu}
for the classical canonical quantization and related issues).

For this, we introduce
the canonical operators for $k\in\{1,\dots,n\}$
\begin{equation}\label{CQ}
P_k:=-i\hbar^s\partial_k (-\Delta)^{(s-1)/2} \ {\mbox{
and }} \
Q_k:=x_k.\end{equation}
Notice that $Q_k$ is the classical position operator, namely
the multiplication by the $k$th space coordinate.
On the other hand, $P_k$ is a fractional momentum
operator, that reduces\footnote{Of course, from
the point of view of physical dimensions, the fractional momentum
is not a momentum, since it has physical dimension ${\tt[Planck constant]}^s/
{\tt [length]}^s$, while the classical momentum
has physical dimension~${\tt[Planck constant]}/
{\tt [length]}$. Namely, the physical dimension of
the fractional momentum is a fractional power of
the physical dimension of
the classical momentum. Clearly, the same phenomenon occurs
for the physical dimension of the fractional Laplace
operators in terms of the usual Laplacian.}
to the classical momentum $-i\hbar\partial_k$
when~$s=1$. In this setting, our goal is to check that the
commutator
$$[Q,P]:=\sum_{k=1}^n [Q_k,P_k]$$
does not vanish. For this, we suppose~$0<\sigma<n/2$
and use the Riesz potential
representation of the inverse of the fractional Laplacian
of order~$\sigma$, that is
\begin{equation}\label{L11}
(-\Delta)^{-\sigma} \psi(x)=
c(n,s) \int_{\R^n}\frac{\psi(x-y)}{|y|^{n-2\sigma}}\,dy=
c(n,s) \int_{\R^n}\frac{\psi(y)}{|x-y|^{n-2\sigma}}\,dy,
\end{equation}
for a suitable~$c(n,s)>0$,
see~\cite{LAND}.

In our case we use~\eqref{L11}
with~$\sigma:=(1-s)/2 \in(0,1/2)\subseteq(0,n/2)$.
Then
$$ P_k\psi(x)=
-c(n,s)\,i\hbar^s\partial_k
\int_{\R^n}\frac{\psi(y)}{|x-y|^{n+s-1}}\,dy=
c(n,s)\, i\,\hbar^s \,(n+s-1)\,\int_{\R^n} \frac{(x_k-y_k)\,\psi(y)}{
|x-y|^{n+s+1}}\,dy$$
and so
\begin{eqnarray*}
P_k Q_k\psi (x)= P_k(x_k\psi(x))
=c(n,s)\, i\,\hbar^s \,(n+s-1)\,\int_{\R^n} \frac{(x_k-y_k)\,y_k\,\psi(y)}{
|x-y|^{n+s+1}}\,dy.\end{eqnarray*}
This gives that
\begin{eqnarray*}
&&Q_kP_k\psi-P_kQ_k\psi
\\&=& c(n,s)\, i\,\hbar^s \,(n+s-1)\,\left[
\int_{\R^n} \frac{x_k\,(x_k-y_k)\,\psi(y)}{
|x-y|^{n+s+1}}\,dy
-\int_{\R^n} \frac{(x_k-y_k)\,y_k\,\psi(y)}{
|x-y|^{n+s+1}}\,dy\right]
\\ &=& c(n,s)\, i\,\hbar^s \,(n+s-1)\,
\int_{\R^n} \frac{(x_k-y_k)^2\,\psi(y)}{
|x-y|^{n+s+1}}\,dy ,
\end{eqnarray*}
and so, by summing up\footnote{Alternatively, one can also perform
the commutator calculation in Fourier space
and then reduce to the original variable by an inverse Fourier
transform. This computation can be done easily by using the facts
that the Fourier transform sends products into convolutions and that
(up to constants)
\begin{eqnarray*}
(\hat x_k * g)(\xi)&=& {\mathcal{F}}
\big(x_k {\mathcal{F}}^{-1} g(x)\big) (\xi)
\\ &=& \int_{\R^n} \,dx \int_{\R^n} \,dy \; e^{ix\cdot (y-\xi)} x_k g(y)
\\
&=& i^{-1}
\int_{\R^n} \,dx \int_{\R^n} \,dy \;
\partial_{y_k} e^{ix\cdot (y-\xi)} g(y)\\
&=& i\int_{\R^n} \,dx \int_{\R^n} \,dy \;
e^{ix\cdot (y-\xi)} \partial_k g(y)\\
&=& i
\int_{\R^n} \,dx\;
e^{-ix\cdot \xi} {\mathcal{F}}^{-1}(\partial_k g)(x)\\
&=& i {\mathcal{F}}\big({\mathcal{F}}^{-1}(\partial_k g)\big)(\xi)\\
&=& i \partial_k g(\xi).
\end{eqnarray*}
Then we leave to the reader the computation of ${\mathcal{F}}([Q,P]\psi)(\xi)$.}
and recalling~\eqref{L11}, we conclude that
$$ [Q,P]\psi=
c(n,s)\, i\,\hbar^s \,(n+s-1)\,
\int_{\R^n} \frac{\psi(y)}{
|x-y|^{n+s-1}}\,dy=
i\,(n+s-1)\,\hbar^s (-\Delta)^{(s-1)/2}\psi.$$
Notice that, as $s\rightarrow1$, this formula reduces to the
the classical Heisenberg Uncertainty Principle.

We also point out that a similar computation shows that,
differently from the local quantum momentum,
the $k$th fractional quantum momentum does not commute
with the $m$th spatial coordinates even when $k\ne m$, namely
$[Q_m, P_k]\psi(x)$ is, up to normalizing constants,
$$ i\hbar^s \int_{\R^n} \frac{(x_m-y_m)(x_k-y_k)\,\psi(y)}{|x-y|^{n+s+1}}
\,dy.$$
This Heisenberg Uncertainty Principle
is also compatible with equation~\eqref{SCH},
in the sense that the diffusive operator~$H_0$
is exactly the one obtained by the canonical quantization in~\eqref{CQ}:
indeed
\begin{eqnarray*}
\sum_{k=1}^n P_k^2 &=&\sum_{k=1}^n
\Big( -i\hbar^s\partial_k (-\Delta)^{(s-1)/2}\Big)
\Big( -i\hbar^s\partial_k (-\Delta)^{(s-1)/2}\Big)\\
&=& -\hbar^{2s} \sum_{k=1}^n \partial_k^2 (-\Delta)^{s-1}
\\ &=& -\hbar^{2s} \Delta \,(-\Delta)^{s-1}\\
&=& \hbar^{2s} (-\Delta)^s\\
&=& H_0.
\end{eqnarray*}
Moreover, we mention
that the fractional Laplace operator also arises
naturally in the high energy Hamiltonians of relativistic theories.
For further motivation of the fractional Laplacian in modern physics
see e.g.~\cite{Chen} and references therein.

\end{document}